\documentclass[11pt]{amsart}
\usepackage{amsmath,amsthm,amssymb,tikz,mathrsfs,pgfplots}
\usepackage[numbers,comma,square,sort&compress]{natbib}
\usepackage{enumitem}

\RequirePackage{color}
\definecolor{my-link}{rgb}{0.5,0.0,0.0}
\definecolor{my-blue}{rgb}{0.0,0.0,0.6}
\definecolor{my-red}{rgb}{0.5,0.0,0.0}
\definecolor{my-green}{rgb}{0.0,0.5,0.0}
\definecolor{nicos-red}{rgb}{0.75,0.0,0.0}
\definecolor{really-light-gray}{gray}{0.8}
\definecolor{darkgreen}{rgb}{0.0,0.5,0.0}
\definecolor{darkblue}{rgb}{0.0,0.0,0.3}
\definecolor{nicosred}{rgb}{0.65,0.1,0.1}
\definecolor{light-gray}{gray}{0.7}
\RequirePackage[colorlinks,urlcolor=my-blue,linkcolor=my-red,citecolor=my-green]{hyperref}

\usepackage{xargs}
\usepackage[colorinlistoftodos,prependcaption,textsize=footnotesize]{todonotes}
\newcommandx{\addmath}[2][1=]{\todo[linecolor=red,backgroundcolor=red!25,bordercolor=red,#1]{#2}}
\newcommandx{\fixtext}[2][1=]{\todo[linecolor=blue,backgroundcolor=blue!25,bordercolor=blue,#1]{#2}}
\newcommandx{\note}[2][1=]{\todo[linecolor=yellow,backgroundcolor=yellow!25,bordercolor=yellow,#1]{#2}}

\usepackage[color]{changebar}
\cbcolor{red}
\usepackage{bm}

\usepackage[left=1in,top=1in,right=1in,bottom=1in]{geometry}
\usepackage{scalerel}    %%%%  used first for scaling \bullet  
\usepackage{stmaryrd}   %%%% for integer interval notation 
\usepackage{mathabx}   %% made \widecheck work  
\usepackage{hyphenat}
\usepackage{scalerel}    %%%%  used first for scaling \bullet  

\newcommand\bbullet{{\raisebox{0.5pt}{\scaleobj{0.5}{\bullet}}}} %% looks good for paths x_\bbullet 
\newcommand\acdot{{\raisebox{-1pt}{\scaleobj{1.4}{\cdot}}}}  %%dot inside function 

\newtheorem{theorem}{Theorem}[section]
\newtheorem{lemma}[theorem]{Lemma}

\newtheorem{corollary}[theorem]{Corollary}

\theoremstyle{definition}
\newtheorem{definition}[theorem]{Definition}

\theoremstyle{remark}
\newtheorem{remark}[theorem]{Remark}

\allowdisplaybreaks[3]
\numberwithin{figure}{section}
\numberwithin{equation}{section}

\definecolor{sussexg}{rgb}{0,0.7,0.7}
\definecolor{sussexp}{rgb}{0.4,0,0.4}
\definecolor{sussexb}{rgb}{0.4,0.4,0.7}
\definecolor{mygray}{rgb}{0.75,0.75,0.75}

\newcommand{\E}{\mathbb{E}}

\newcommand{\cD}{\mathcal{D}}
\newcommand{\cE}{\mathcal{E}}

\newcommand{\cM}{\mathcal{M}}
\newcommand{\cB}{\mathcal{B}}

\newcommand{\cT}{\mathcal{T}}
\renewcommand{\P}{\mathbb{P}}
\newcommand{\V}{\mathbb{V}}
\newcommand{\Phat}{\widehat{\mathbb{P}}}
\newcommand{\Ehat}{\widehat{\mathbb{E}}}

\newcommand{\R}{\mathbb{R}}

\newcommand{\Z}{\mathbb{Z}}
\newcommand{\N}{\mathbb{N}}
\newcommand{\X}{\mathbb{X}}
\newcommand{\T}{\mathbb{T}}

\newcommand{\cX}{\mathcal{X}}
\newcommand{\cXx}{\mathcal{X}^x}
\newcommand{\e}{\varepsilon}

\newcommand{\fl}[1]{\lfloor{#1}\rfloor}

\DeclareMathOperator{\bbE}{\mathbb{E}}

\DeclareMathOperator{\bbN}{\mathbb{N}}

\DeclareMathOperator{\bbP}{\mathbb{P}}

\DeclareMathOperator{\bbR}{\mathbb{R}}

\DeclareMathOperator{\bbV}{\mathbb{V}}

\DeclareMathOperator{\bbX}{\mathbb{X}}

\DeclareMathOperator{\bbZ}{\mathbb{Z}}

\DeclareMathOperator{\bfE}{\mathbf{E}}

\DeclareMathOperator{\bfP}{\mathbf{P}}

\DeclareMathOperator{\bfk}{\mathbf{k}}

\DeclareMathOperator{\mfg}{\mathfrak{g}}

\DeclareMathOperator{\mfG}{\mathfrak{G}}
\DeclareMathOperator{\mfS}{\mathfrak{S}}

\DeclareMathOperator{\sA}{\mathcal{A}}
\DeclareMathOperator{\sB}{\mathcal{B}}

\DeclareMathOperator{\sF}{\mathcal{F}}
\DeclareMathOperator{\sG}{\mathcal{G}}
\DeclareMathOperator{\sH}{\mathcal{H}}
\DeclareMathOperator{\sI}{\mathcal{I}}

\DeclareMathOperator{\sL}{\mathcal{L}}
\DeclareMathOperator{\sM}{\mathcal{M}}

\DeclareMathOperator{\sR}{\mathcal{R}}

\DeclareMathOperator{\sX}{\mathcal{X}}

\DeclareMathOperator{\lf}{\lfloor}
\DeclareMathOperator{\rf}{\rfloor}
\DeclareMathOperator{\Var}{{\mathbb V}ar}

\def\Diff{\cD}

\def\w{\omega}
\def\wbar{{\bar\w}}
\def\wtil{{\tilde\w}}
\newcommand{\ehat}{\widehat{e}}
\def\unif{\vartheta}

\def\eB{\kappa}

\providecommand{\abs}[1]{\vert#1\vert}
\def\what{{\widehat\w}}
\def\wh{\widehat}
\def\wt{\widetilde}
\def\Omhat{\widehat\Omega}

\def\fe{\Lambda}
\def\fepl{\fe_{\rm{pl}}}
\def\Zne{Z^{\rm NE}}

\def\ximin{{\underline\xi}}
\def\ximax{{\overline\xi}}
\def\ximinmin{\underline{\underline\xi}}
\def\ximaxmax{\overline{\overline\xi}}

\def\zetaminmin{\underline{\underline\zeta}}

\def\etamin{{\underline\eta}}

\def\zetamin{{\underline\zeta}}
\def\zetamax{{\overline\zeta}}
\DeclareMathOperator{\ri}{ri}    %rel interior  
\DeclareMathOperator{\ext}{ext}    %extremes 
\DeclareMathOperator{\dist}{dist} 
\def\XDLR{\overline X^{\,x,\w}}
\def\Xtree{\widehat X^{\,\zeta,\w}}
\def\Xtreen{\widehat X^{\,\zeta,\w,(n)}}
\def\DLR{\mathrm{DLR}}
\newcommand{\siDLR}{\overrightarrow{\DLR}}
\newcommand{\biDLR}{\overleftrightarrow{\DLR}}
\def\CI{\phi}
\def\cid{\xi_*}
\def\pici{\pi^{\rm{cif}}}
\def\ker{\kappa}

\def\Bbar{\overline B}

\def\Pibar{\overline\Pi}

\newcommand{\That}{\widehat{T}}

\def\Omtemphat{\widehat\Omega_0}
\def\Ommono{\hyperref[Ommonohat]{\textcolor{black}{\Omega_{\rm{coc}}}}} %the Omega on which monotonicity (and other properties?) of cocycles holds
\def\Ommonohat{\hyperref[Ommonohat]{\textcolor{black}{\widehat\Omega_{\rm{coc}}}}}%the Omega on which monotonicity (and other properties?) of cocycles holds
\def\Omcont#1{\hyperref[Omconthat]{\textcolor{black}{\Omega_{\rm{cont},#1}}}} %the Omega on which continuity holds at \xi
\def\Omconthat#1{\hyperref[Omconthat]{\textcolor{black}{\widehat\Omega_{\rm{cont},#1}}}} %the Omega on which continuity holds at \xi
\def\Omconsthat#1{\hyperref[Omconsthat]{\textcolor{black}{\widehat\Omega_{\rm{tilt},#1}}}} %the Omega on which h(B) is constan
\def\Omback#1{\hyperref[Omback]{\textcolor{black}{\Omega_{#1,\downarrow0}}}} %the Omega on which the B-tree dies out
\def\OmBus{\hyperref[OmBus]{\textcolor{black}{\Omega'_{[\ximin,\ximax]}}}} %the Omega on which liminf and limsup bounds hold for Busemann limits for the given \xi
\def\Ombi{\hyperref[Ombi]{\textcolor{black}{\Omega_{\rm{bi},[\ximin,\ximax]}}}} %the Omega on which liminf and limsup bounds hold for Busemann limits for the given \xi
\def\OmBuszeta{\hyperref[OmBus]{\textcolor{black}{\Omega'_{[\zetamin,\zetamax]}}}} %the Omega on which liminf and limsup bounds hold for Busemann limits for the given \zeta
\def\OmBusnice{\hyperref[OmBusnice]{\textcolor{black}{\Omega_1}}} %the Omega on which Busemann limits exist on a dense set
\def\OmBushat{\hyperref[OmBushat]{\textcolor{black}{\widehat\Omega'_{[\ximin,\ximax]}}}} %the Omega on which Busemann limit holds
\def\OmBusallhat{\hyperref[OmBusallhat]{\textcolor{black}{\widehat\Omega_{\rm{Bus}}}}} %the Omega on which liminf and limsup bounds hold for Busemann limits for all \xi
\def\OmBusall{\hyperref[OmBusallhat]{\textcolor{black}{\Omega_{\rm{Bus}}}}} %the Omega on which liminf and limsup bounds hold for Busemann limits for all \xi
\def\Omexist{\hyperref[Omexist]{\textcolor{black}{\Omega_{\rm{exist}}}}} %the Omega on which directedness of cocycle DLRs holds
\def\Omexisthat{\hyperref[Omexisthat]{\textcolor{black}{\widehat\Omega_{\rm{exist}}}}} %the Omega on which directedness of cocycle DLRs holds
\def\Omdir{\hyperref[Omdir]{\textcolor{black}{\Omega_{\rm{dir}}}}} %the Omega on which every DLR is directed
\def\Omdirhat{\hyperref[Omdirhat]{\textcolor{black}{\widehat\Omega_{\rm{dir}}}}} %the Omega on which every DLR is directed
\def\OmBusPi{\hyperref[OmBusPi]{\textcolor{black}{\Omega_{[\ximin,\ximax]}}}} %the Omega on which DLR is Busemann for the given \xi
\def\OmBusPizeta{\hyperref[OmBusPi]{\textcolor{black}{\Omega_{[\zetamin,\zetamax]}}}} %the Omega on which DLR is Busemann for the given \xi
\def\OmBusPihat{\hyperref[OmBusPihat]{\textcolor{black}{\widehat\Omega_{[\ximin,\ximax]}}}} %the Omega on which DLR is Busemann for the given \xi
\def\OmBusPiall{\hyperref[OmBusPiall]{\textcolor{black}{\Omega_{\rm{uniq}}}}} %the Omega on which Busemann limits will hold for all directions at once
\def\Omnondeg{\hyperref[Omnondeg]{\textcolor{black}{\Omega_{\rm{nondeg}}}}}  %the Omega on which the only nondegenerate DLR solutions are \Pi^{e_i}
\def\OmBhat#1{\hyperref[OmBhat]{\textcolor{black}{\widehat\Omega_{#1}}}}  %the Omega on which the DLR solution corresponding to B is directed
\def\Omreg{\hyperref[Omreg]{\textcolor{black}{\Omega_{\rm{reg}}}}} % the Omega on which regular conditional probabilities are supported on the projection
\def\Omcif{\hyperref[Omcif]{\textcolor{black}{\Omega_{\rm{cif}}}}} % the Omega on which regular conditional probabilities are supported on the projection
\def\Omeihat{\hyperref[Omeihat]{\textcolor{black}{\widehat\Omega_{e_1,e_2}}}}

\def\WLLN{WLLN$_\xi$}
\def\SLLN{SLLN$_\xi$}

\makeatletter
\DeclareRobustCommand{\cev}[1]{%
  \mathpalette\do@cev{#1}%
}
\newcommand{\do@cev}[2]{%
  \fix@cev{#1}{+}%
  \reflectbox{$\m@th#1\vec{\reflectbox{$\fix@cev{#1}{-}\m@th#1#2\fix@cev{#1}{+}$}}$}%
  \fix@cev{#1}{-}%
}
\newcommand{\fix@cev}[2]{%
  \ifx#1\displaystyle
    \mkern#22mu
  \else
    \ifx#1\textstyle
      \mkern#22mu
    \else
      \ifx#1\scriptstyle
        \mkern#22mu
      \else
        \mkern#22mu
      \fi
    \fi
  \fi
}

\makeatother

\def\backpix{{\cev{\pi}}^{\,x}}

\def\esssup{\mathop{\mathrm{ess\,sup}}}
\def\range{\mathcal R}
\def\Uset{\mathcal U}
\def\Usetnonuniq{\Uset_x^\w}
\def\Usetnonuniqz{\Uset_0^\w}
\def\Udense{\Uset_0}
\def\Ddense{\Diff_0}

\font \mymathbb = bbold10 at 11pt
\newcommand{\one}{\mbox{\mymathbb{1}}}    %indicator function

\usepackage{filemod}

\begin{document}

\title[Polymer Gibbs Measures]
{Busemann functions and Gibbs measures\\ in directed polymer models on $\mathbb{Z}^2$} 

\author[C.~Janjigian]{Christopher Janjigian}
\address{Christopher Janjigian\\ University of Utah\\  Mathematics Department\\ 155 S 1400 E\\   Salt Lake City, UT 84112\\ USA.}
\email{janjigia@math.utah.edu}
\urladdr{http://www.math.utah.edu/~janjigia}
\thanks{C.\ Janjigian was partially supported by a postdoctoral grant from the Fondation Sciences Math\'ematiques de Paris while working at Universit\'e Paris Diderot.}

\author[F.~Rassoul-Agha]{Firas Rassoul-Agha}
\address{Firas Rassoul-Agha\\ University of Utah\\  Mathematics Department\\ 155S 1400E\\   Salt Lake City, UT 84112\\ USA.}
\email{firas@math.utah.edu}
\urladdr{http://www.math.utah.edu/~firas}
\thanks{F.\ Rassoul-Agha was partially supported by National Science Foundation grants DMS-1407574 and DMS-1811090}


%\date{October 8, 2018}   %Submitted
%\date{March 1, 2019}    %Revised
%\date{May 16, 2019}    %Accepted
%\date{May 21, 2019}  %Some modifications
\date{October 3, 2019} %Final modifications
%\date{Last modified on \today\ at\ \filemodprinttime{\jobname}}

\begin{abstract}  
We consider random walk in a space-time random potential, also known as directed random polymer measures, on the planar square lattice
with nearest-neighbor steps and general i.i.d.\ weights on the vertices.
We construct covariant cocycles and use them to prove new results on existence, uniqueness/non-uniqueness, and asymptotic directions of semi-infinite polymer measures (solutions to 
the Dobrushin-Lanford-Ruelle equations). We also prove non-existence of covariant or deterministically directed  bi-infinite polymer measures.
Along the way, we prove almost sure existence of Busemann function limits in directions where
the limiting free energy has some regularity. 
\end{abstract}
\maketitle

\setcounter{tocdepth}{1}
\tableofcontents

\section{Introduction} 
%In this paper, 
We study a class of probability measures on nearest-neighbor up-right random walk paths in the two-dimensional square lattice. 
The vertices of the lattice are populated with i.i.d.\ random variables called {\sl weights} and
the energy of a finite path is given by the sum of the weights along the path. We assume that these weights are nondegenerate and have finite $2+\e$ moments, but they are otherwise general. The {\sl point-to-point quenched polymer measures} are probability measures on admissible paths between two fixed sites in which the probability of a path is proportional to the exponential of its energy.
This model is known as the {\sl directed polymer with bulk disorder} and it was introduced in the statistical physics literature by Huse and Henley \cite{Hus-Hen-85} in 1985 to model the domain wall in the ferromagnetic Ising model with random impurities. It has been the subject of intense study over the past  three decades;
see the recent surveys \cite{Com-Shi-Yos-04,Hol-09,Com-17}.

Many of our main results concern {\sl semi-infinite polymer measures}, which we will also call {\sl semi-infinite DLR solutions} or {\sl Gibbs measures} to help connect our results to the usual language of statistical mechanics. Semi-infinite polymer measures are probability measures on infinite length admissible up-right paths emanating from a fixed site which are consistent with the point-to-point quenched polymer measures.  Some of the natural questions about such measures include whether all such measures must satisfy a law of large numbers (LLN), whether measures exist which satisfy a LLN with any given direction, and under what conditions such measures are unique. Ideally one would like to answer these questions for almost every realization of the environment simultaneously for all directions.

This is the third paper to consider these questions in 1+1 dimensional directed polymer models; the recent \cite{Geo-etal-15} and \cite{Bak-Li-18-}  address similar questions in related models which have more structure than the models considered here. 
%while the latter does not consider a solvable model, but makes use of a certain symmetry in their model that leads to a 
%quadratic limiting free energy.   See Section \ref{sec:works} for more details and a longer discussion of past literature.
%

\cite{Geo-etal-15} studies the model first introduced in \cite{Sep-12-corr}, which is a special case of the model studied in this paper where the weights have the log-gamma distribution. 
The authors use the solvability of the model (i.e.\ the possibility of exact computations) 
to introduce semi-infinite polymer measures which satisfy a LLN with any fixed direction for that model. 
As alluded to in the fourth paragraph on page 2283 of \cite{Geo-etal-15}, the authors expected their structures and conclusions to generalize. We demonstrate that they do, but in addition to studying more general models, the present paper considers a much wider class of problems than \cite{Geo-etal-15}; hence most of the results we discuss are new even in this solvable setting.

\cite{Bak-Li-18-} studies 1+1 dimensional directed polymers in continuous space and discrete time, where the underlying random walk has Gaussian increments. The authors show existence and uniqueness of semi-infinite polymer measures satisfying the law of large numbers with a fixed deterministic direction\textemdash but, the event on which this holds depends on the direction chosen. While the model considered in \cite{Bak-Li-18-} is not solvable, a symmetry in the model inherited from the Gaussian walk leads to a quadratic limiting free energy. This is a critical feature of the model, since the method used in that project relies in an essential way on having a curvature bound for the free energy.

Some of our results, specifically ones concerning existence and uniqueness of semi-infinite polymer measures in deterministic directions, 
can likely be obtained with the techniques of \cite{Bak-Li-18-} if one assumes or proves a curvature condition on the limiting free energy, which we will denote by $\fe$. 
Proving such a condition is a long-standing open problem.
We prefer to avoid {\sl a priori} curvature assumptions for two reasons: 
%first, our results hold under moment assumptions which are weaker than those which are believed to characterize the KPZ class \cite{Bir-Bou-Pot-07}, so it is possible that some of the models we consider have free energies which do not in fact satisfy global curvature conditions; second, 
first, most of our theorems are valid  under no assumptions on $\fe$
and second, as we will see in Section \ref{sec:cocycles}, the stochastic process that is our main tool, 
the {\sl Busemann process}, is naturally indexed by elements of the superdifferential of $\fe$, and
we believe that understanding the structure of this object without any {\sl a priori} regularity assumptions might provide a path to proving differentiability or strict concavity of $\fe$.  

%The approach in \cite{Bak-Li-18-} relies crucially on the curvature of the free energy and does not work in general unless 
%a general theorem on  curvature is proved.  
%In this paper, we develop an approach that circumvents the need for a priori regularity assumptions and in fact 
%offers the chance of progress on questions of regularity of the free energy.  
%See the end of Section \ref{sec:works} for more.

We now sketch what we can show about semi-infinite polymers in more detail. Before beginning, we remark that the set of semi-infinite polymer measures is convex and it  suffices to study the extreme points. 
Although most of our theorems apply without {\sl a priori} assumptions on $\fe$, 
they take their nicest form when $\fe$ is both  differentiable and strictly concave. 
This is conjectured to be the case in general.
In this case, our results say that except for a single null set of weights all of the following hold. Every extremal measure satisfies a strong LLN (Corollary \ref{cor:main2}). For every direction
in $\Uset=\{(t,1-t):0\le t\le 1\}$ 
there is at least one extremal semi-infinite polymer measure with that asymptotic direction (Corollary \ref{cor:main1}). Except for possibly a random countable set of directions, this measure is unique (Theorem \ref{th:main4}(\ref{th:main4:e})). The directions of non-uniqueness are precisely the directions at which %a microscopic gradient-like object given by 
the Busemann process is discontinuous (Theorem \ref{th:main4}(\ref{th:main4:e})). 
%HERE
This set of directions is either always empty or always infinite %dense in $\Uset$ 
(Theorem \ref{th:main4}(\ref{th:main4:b''})). 
The connection between the non-uniqueness set and discontinuities of the Busemann process has not previously been observed. Moreover, this is the first time the countability of this set has been shown in positive temperature.

We do not resolve the question of whether or not the set of non-uniqueness directions is actually empty almost surely. 
As mentioned above, this is equivalent to the almost sure continuity of the process of Busemann functions viewed as a function of the direction. This latter question can likely be answered for the log-gamma polymer, where it is natural to expect that the distribution of the Busemann process can be described explicitly using positive temperature analogues of the ideas in \cite{Fan-Sep-18-}.
It is known that this set is not empty in last-passage percolation (LPP), the zero-temperature version of the polymer model.  See Theorem 2.8 and Lemma 5.2 in \cite{Geo-Ras-Sep-17-ptrf-2}. 
%It is plausible to expect that positive temperature adds enough noise that the set is in fact empty for the polymer model. 

Aside from the problems discussed above, we study a number of natural related questions. For example, based on analogies to bi-infinite geodesics in percolation, it is natural to expect that nontrivial  {\sl bi-infinite polymer measures} should not exist. We are able to prove non-existence of shift-covariant bi-infinite polymer measures and of bi-infinite polymer measures satisfying a LLN with a given fixed direction, 
but do not otherwise address non-covariant measures.
%fully resolve the question. 
We further study the {\sl competition interface}, introduced in \cite{Geo-etal-15} as a positive-temperature analogue of the object from last-passage percolation \cite{Fer-Pim-05}. In particular, we prove that the interface satisfies a LLN and characterize its random direction in terms of  the Busemann process. 

Our results can also be interpreted in terms of existence and uniqueness of global 
stationary solutions and pull-back attractors of a discrete viscous stochastic Burgers equation. 
This is the main focus of our companion paper \cite{Jan-Ras-19-pams-}.
See also \cite{Bak-Kha-18} and the discussion in \cite{Bak-Li-18-}, which focuses on this viewpoint.

\subsection{Related works}\label{sec:works}

In his seminal paper \cite{Sin-91} Sinai proved existence and uniqueness of 
stationary global solutions to the stochastic viscous Burgers equation with a forcing that is periodic in space and either also periodic in time or a white noise in time. 
Later, \cite{Gom-etal-05} extended Sinai's results to the multidimensional setting using a 
stochastic control approach and \cite{Dir-Sou-05}  used PDE methods to prove similar results for both viscous and inviscid Hamilton-Jacobi equations with periodic spatial dependence. 
Periodicity was relaxed in \cite{Sin-93,Bak-Kha-10},
where the random potential was assumed to have a special form (not stationary in space) that 
ensures localization of the reference random walk near the origin and makes the situation essentially 
compact so the arguments from \cite{Sin-91} could be used. A similar multidimensional model is treated in \cite{Bak-Kha-10}.  See also \cite{E-etal-00,Hoa-Kha-03,Itu-Kha-03,Bak-13} for zero temperature results using similar methods. 

The connection between solving the stochastic viscous Burgers equation and the existence of Busemann limits in related directed polymer models was observed in  \cite{Kif-97} where they treated the case of 
strong forcing (high viscosity) or, in statistical mechanics terms, weak disorder (high temperature).  See also the Markov chains constructed by Comets-Yoshida \cite{Com-Yos-06}, Yilmaz \cite{Yil-09-aop}, Section 6 in \cite{Ras-Sep-14}, and  Example 7.7 in \cite{Geo-Ras-Sep-16}. The model we consider is in 1+1 space-time dimensions, which is known to be always 
%\note{Technically, \cite{Com-Var-06} assumes that $\E[e^{\beta\w_0}]<\infty$ for every $\beta>0$, and \cite{Lac-10} weakens this assumption to finiteness for some $\beta>0$.} 
in strong disorder \cite{Com-Var-06,Lac-10}.

The recent papers \cite{Bak-Li-18-} and \cite{Geo-etal-15}, mentioned earlier, are more closely related to this work as both study strictly positive temperature polymers in a non-compact setting and in the strong disorder regime. 

%\cite{Bak-Li-18-} studies 1+1 dimensional directed polymers in a discrete-time continuous-space setting where the underlying random walk has Gaussian increments. 
%In \cite{Bak-Li-18-}, the authors prove existence and uniqueness of semi-infinite shift-covariant Gibbs measures satisfying a LLN with a given deterministic direction and explore related results in the dynamical systems framework described above. From the dynamical systems point of view, their results give existence and uniqueness of shift-covariant solutions, with a given deterministic velocity, to the corresponding stochastic viscous Burgers equation with spatially ergodic kick forcing at integer times. Additionally, the authors describe the basins of attraction of these global solutions.

Currently, there are two major approaches to studying the general structure of infinite and semi-infinite directed polymers in zero or positive temperature. The first approach was  introduced by Newman and coauthors \cite{New-95, Lic-New-96, How-New-97, How-New-01} in the context of first-passage percolation (FPP). 
This approach requires control of the curvature of $\fe$.
This property is used to prove straightness estimates for the quenched point-to-point polymer measures. Existence and uniqueness results then come as consequences, as well as existence of {\sl Busemann functions},
which are defined through limits of ratios of partition functions.  This is the approach taken by \cite{Bak-Li-18-}. See also \cite{Bak-16,Bak-Cat-Kha-14,Cat-Pim-11,Cat-Pim-12,Cat-Pim-13,Fer-Pim-05,Wut-02} for other papers following this approach in zero temperature.

In this paper, we take the other, more recent, approach in which Busemann functions are the fundamental object.
The use of Busemann functions to study the structure of semi-infinite geodesics traces back to the seminal work of Hoffman \cite{Hof-05,Hof-08} on FPP.
Here, we construct covariant cocycles which are consistent with the weights on an extension of our probability space and then use a coupling argument and planarity to prove existence and properties of Busemann functions. 
The bulk of the work then goes towards using this process of Busemann functions to prove the results about infinite and semi-infinite polymer measures. This program was first achieved in zero temperature by \cite{Dam-Han-14,Dam-Han-17} in FPP and \cite{Geo-Ras-Sep-17-ptrf-1,Geo-Ras-Sep-17-ptrf-2} in LPP. \cite{Car-Sou-17} also takes this approach to construct correctors, which are the counterparts of Busemann functions, in their study of stochastic homogenization of viscous Hamilton-Jacobi equations.

In \cite{Geo-etal-15} the desired cocycles were constructed using the solvability of the model.
In the present paper we build cocycles using weak subsequential Ces\`aro limits of ratios of 
partition functions, which is a version of the method Damron and Hanson \cite{Dam-Han-14} used in their study of FPP. Our situation requires overcoming some nontrivial technical hurdles not encountered there which arise due to the path directedness in our model. An alternative approach to producing cocycles based on lifting the queueing theoretic arguments of \cite{Mai-Pra-03} to positive temperature is also possible. These queueing theoretic results furnished the desired cocycles in 
\cite{Geo-Ras-Sep-17-ptrf-1, Geo-Ras-Sep-17-ptrf-2}. It is noteworthy that the queuing results rely on a specific choice of admissible path increments, while the weak convergence idea seems to work more generally.

\subsection{Organization}
Our paper is structured as follows. We start with some notation in Section \ref{sub:notation} then introduce the model in Section \ref{sub:model}.
Section \ref{sub:DLR} introduces semi-infinite and bi-infinite polymer measures (DLR solutions).
Our main results are stated in Section \ref{sec:results}.  
In Section \ref{sec:coc} we address existence of covariant cocycles and Busemann functions.
Using these cocycles we 
prove (more general versions of) our main results on semi-infinite DLR solutions in Section \ref{sec:semiDLR}.
In Section \ref{sec:biDLR}, we use these results to show non-existence of covariant or deterministically directed bi-infinite DLR solutions. 
A number of technical results are deferred to the appendix. 
One such result on almost sure coalescence of coupled random walks in a common random environment, Theorem \ref{thm:RWREcoal}, may be of independent interest to some readers.

\section{Setting}\label{sec:setting}
After establishing some notation, we introduce the quenched polymer measures and the 
Gibbs measures formulation.

\subsection{Notation}\label{sub:notation}
Throughout the paper $(\Omega, \sF,\bbP)$ is a Polish probability space equipped with a group of $\sF$-measurable $\bbP$-preserving transformations $T_x:\Omega\to\Omega$,  $x\in\Z^2$, such that $T_0$ is the identity map and  $T_xT_y=T_{x+y}$ for all $x,y\in\Z^2$.  $\E$ is expectation relative to $\P$.
A generic point in this space will be denoted by $\w \in \Omega$. We assume that there exists a family $\{\w_x(\w) : x \in \bbZ^2\}$ of real-valued random variables called {\sl weights} such that  
	\begin{align}\label{main-assump}
	&\text{$\{\w_x\}$ are i.i.d.\ under $\bbP$, $\exists p>2$\,: $\bbE[\abs{\w_0}^p] < \infty$, and $\Var(\w_0)>0$.}
	\end{align}
We assume further that  $\w_y(T_x \w) = \w_{x+y}(\w)$ for all $x,y \in \bbZ^2$. 
An example is the canonical setting of a product space $\Gamma = \bbR^{\bbZ^2}$ equipped with the product topology, product Borel $\sigma$-algebra $\mfS$, the product measure $\P_0^{\otimes\Z^2}$ with $\P_0$ a probability measure on $\R$, 
the natural shift maps, and with $\w_x$ denoting the natural coordinate projection.

We study probability measures on paths with increments $\range = \{e_1, e_2\}$, the standard basis of $\bbR^2$. 
Let $\Uset$ denote the convex hull of $\range$ with $\ri\Uset$ its relative interior.
Write $\ehat = e_1 + e_2$. For $m \in \bbZ$ denote by $\V_m = \{x \in \Z^2 : x \cdot \ehat = m\}$.
% the collection of sites at {\sl level} $m$.
We denote sequences of sites by $x_{m,n}=(x_i:m\le i\le n)$ where $-\infty \leq m\leq n \leq \infty$. 
We require throughout that $x_i \in \V_i$.

For $x \in \V_m$ and $y \in \V_n$ with $m \leq n$, the collection of admissible paths from $x$ to $y$ is denoted 
$\X_x^y = \{x_{m,n} : x_m = x, x_n = y, x_i - x_{i-1} \in \range\}$. This set is empty unless $x \leq y$. ($x\le y$ is understood coordinatewise.) 
The collection of {\sl admissible} paths from $x$ to {\sl level} $n$ is denoted 
$\X_x^{(n)} = \{x_{m,n} : x_m = x, x_{i}-x_{i-1} \in \range\}$. The collection of  {\sl semi-infinite} paths 
{\sl rooted {\rm(}or starting{\rm)} at} $x$ is denoted by 
$\X_x = \{x_{m,\infty} : x_m = x, x_i - x_{i-1} \in \range \}$ and the collection of {\sl bi-infinite} paths is  
$\X = \{x_{-\infty,\infty} : x_i - x_{i-1} \in \range\}$. 
The spaces $\X_x^y, \X_x^{(n)},$ and $\X_x$ are compact and therefore separable. 
The space $\X$ can be viewed naturally as $\V_0 \times \{e_1,e_2\}^{\bbZ}$
which is separable but not compact. We equip these spaces with the associated Borel $\sigma$-algebras 
$\sX^{x,y}, \sX^{x,(n)}, \sX^{x}$ and $\sX$. Given a subset of indices $A$, we denote by $\sX^{x,y}_A, \sX^{x,(n)}_A, \sX^{x}_A$ and $\sX_A$ the associated sub $\sigma$-algebra generated by the coordinate projections $\{x_i : i \in A\}$. 
It will at times be necessary to {\sl concatenate} or {\sl split} admissible paths. These operations will be denoted via the convention $x_{m,n} = x_{m,k}x_{k,n}$, where $x_{m,k} \in \X_{x_m}^{x_k}$ and $x_{k,n} \in \X_{x_k}^{x_n}$. Note that the upper and lower endpoint $x_k$ must match in order for the concatenation to be admissible.

For a $\sigma$-algebra $\cB$, $b\cB$ denotes the set of bounded $\cB$-measurable functions. 
%Quantities of the form $A_{x,y}$ or $A_{x,y}(\w)$ are sometimes written as $A(x,y)$ or $A(x,y,\w)$, and vice-versa.
The space of probability measures on a metric measure space $(\Gamma,\cB)$, equipped with the topology of weak convergence, is denoted $\cM_1(\Gamma,\cB)$. Expectation with respect to a measure $\mu$ is denoted $E^{\mu}$. For $u,v\in\R^2$ we use the notation $[u,v]=\{su+(1-s)v:s\in[0,1]\}$ and $]u,v[=\{su+(1-s)v:s\in(0,1)\}$. The set of extreme points of a convex set $C$ is denoted by $\ext C$.

\subsection{Finite polymer measures}\label{sub:model}
For an {\sl inverse temperature} $\beta \in (0,\infty)$, $x\in\V_m$, and $y \in \V_n$, with $m,n\in\Z$, and $x \leq y$, the {\sl quenched point-to-point partition function} and {\sl free energy} are
\begin{align*}
Z_{x,y}^{\beta} &= \sum_{x_{m,n} \in \X_x^y} e^{\beta \sum_{i=m}^{n-1} \w_{x_i}}\quad\text{and}\quad  F_{x,y}^{\beta} =\frac{1}{\beta} \log  Z_{x,y}^{\beta}.
\end{align*}
We take the convention that $Z_{x,x}^\beta=1$ and $F_{x,x}^\beta=0$ while $Z_{x,y}^\beta = 0$ and $F_{x,y}^{\beta} = -\infty$ whenever we do not have $x \leq y$. 
Similarly, we define the {\sl last passage time} to be the zero temperature $(\beta = \infty)$ free energy:
\begin{align*}
G_{x,y} = F_{x,y}^{\infty} = \max_{x_{m,n} \in \X_x^y} \Bigl\{\sum_{i=m}^{n -1} \w_{x_i}\Bigr\}.
\end{align*}

%For a subset of admissible paths, $A \subset \X_x^y$ we define
%\begin{align*}
%Z_{x,y}^{\beta} (A) &= \sum_{x_{m,n} \in A} e^{\beta \sum_{i=m}^{n-1} \w_{x_i}},
%\end{align*}
%with the convention that an empty sum is $0$.
%For $x\le y$ 
The {\sl quenched point-to-point polymer measure} is the probability measure on $(\X_x^y,\cX^{x,y})$ given by
\begin{align*}
Q_{x,y}^{\w,\beta}(A)= \frac{1}{Z_{x,y}^\beta} \sum_{x_{m,n} \in A} e^{\beta \sum_{i=m}^{n-1} \w_{x_i}}
\end{align*}
for a subset $A \subset \X_x^y$, with the convention that an empty sum is $0$.

For a {\sl tilt} (or {\sl external field}) $h \in \bbR^2$, $n \in \bbZ$ and $x \in \V_m$ with $m \leq n$, the quenched {\sl tilted point-to-line} partition function and free energy are
\begin{align*}
Z_{x,(n)}^{\beta,h} = \sum_{x_{m,n} \in \X_x^{(n)}}e^{ \beta \sum_{i=m}^{n-1}\w_{x_i} + \beta h \cdot(x_n - x_m)} \quad \text{and}\quad F_{x,(n)}^{\beta,h} &=\frac{1}{\beta} \log Z_{x,(n)}^{\beta,h}.
\end{align*}
We take the convention that $Z_{x,(m)}^{\beta,h}=1$ and $F_{x,(m)}^{\beta,h}=0$ while $Z_{x,(n)}^{\beta,h} = 0$ and $F_{x,(n)}^{\beta,h} = -\infty$ if $n < m$. 
Again, we define the point-to-line last passage time to be the zero temperature free energy:
\begin{align*}
G_{x,(n)}^h= F_{x,(n)}^{\infty,h} = \max_{x_{m,n} \in \X_x^{(n)}} \Bigl\{\sum_{i=m}^{n -1} \w_{x_i} + h \cdot(x_n - x_m)\Bigr\}.
\end{align*}
%Following the same convention as above, for $A \subset \X_x^{(n)}$ define
%\begin{align*}
%Z_{x,(n)}^{\beta,h}(A) &= \sum_{x_{m,n} \in A} e^{\beta \sum_{i=m}^{n-1} \w_{x_i} + \beta h \cdot(x_n - x_m)}
%\end{align*}
The quenched tilted point-to-line polymer measure is 
\begin{align*}
Q_{x,(n)}^{\w,\beta,h}(A) &=
\frac1{Z_{x,(n)}^{\beta,h}}\sum_{x_{m,n} \in A} e^{\beta \sum_{i=m}^{n-1} \w_{x_i} + \beta h \cdot(x_n - x_m)}\quad\text{for }A \subset \X_x^{(n)}.
\end{align*}
We will denote by $E_{x,y}^{\w,\beta}$ the expectation with respect to $Q_{x,y}^{\w,\beta}$ and similarly $E_{x,(n)}^{\w,\beta,h}$ will denote the expectation with respect to $Q_{x,(n)}^{\w,\beta,h}$. The random variable given by the natural coordinate projection to level $i$ is denoted by $X_i$.
We will frequently abbreviate the event $\{X_{m,n}=x_{m,n}\}$ by $\{x_{m,n}\}$. 

\subsection{Limiting free energy}\label{sec:fe}
For  $\beta \in (0,\infty]$ 
there are deterministic functions $\fe^\beta:\bbR_+^2 \to \bbR$ and $\fepl^\beta: \bbR^2 \to \bbR$ 
%and an event $\Omshape$ so that 
%$\P(\Omshape)=1$ and for all $\w\in\Omshape$ 
such that  $\P$-a.s.\ for all $0<C<\infty$
\begin{align}\label{shape}
\lim_{n \to \infty} \max_{\substack{x\in\Z_+\\\abs{x}_1\le Cn}} \frac{\abs{F_{0,x}^{\beta} - \fe^\beta(x) }}{n}  &=  \lim_{n\to\infty} \sup_{\abs{h}_1\le C} \frac{\abs{F_{0,(n)}^{\beta,h}  -\fepl^\beta(h)}}{n}=0.
\end{align}
These are called {\sl shape theorems}.
The first limit comes from the point-to-point free energy limit (2.3) in \cite{Ras-Sep-14} and the now standard argument in \cite{Mar-04}.
The second equality comes from the point-to-line free energy limit (2.4) in \cite{Ras-Sep-14} and 
	\begin{align}\label{1-Lip}
	\abs{F_{0,(n)}^{\beta,h}-F_{0,(n)}^{\beta,h'}}\le\abs{h-h'}_1.
	\end{align}

$\fe^\beta$ is concave, $1$-homogenous, and continuous on $\R_+^2$. 
$\fepl^\beta$ is convex and Lipschitz on $\R^2$. 
Lattice symmetry and i.i.d.\ weights imply that
\begin{align*}
\fe^\beta(\xi_1 e_1 + \xi_2 e_2) = \fe^\beta(\xi_2 e_1 + \xi_1 e_2).
\end{align*}
By (4.3-4.4) in \cite{Geo-Ras-Sep-16} $\fe^\beta$ and $\fepl^\beta$ are related  via the  duality
\begin{align}
\fepl^\beta(h)= \sup_{\xi \in \Uset}\{\fe^\beta(\xi) + h \cdot \xi \}\quad\text{and}\quad
\fe^\beta(\xi) = \inf_{h \in \R^2} \{\fepl^\beta(h) - h\cdot\xi\}. \label{eq:ppplduality}
\end{align} 
%We note here that $\fepl^\beta(h)$ is not the Legendre-Fenchel transform of $-\fe^\beta(\xi)$ since the last supremum is over $\Uset$ and not $\bbR^2$.  
$h\in\R^2$ and $\xi\in\ri\Uset$ are said to be in duality if
	\[\fepl^\beta(h)=h\cdot\xi+\fe^\beta(\xi).\]
We denote the set of directions dual to $h$ by $\Uset_h^\beta\subset\ri\Uset$. 
In the arguments that follow, the superdifferential of $\fe^\beta$ at $\xi \in \bbR_+^2$,
\begin{align}
\partial \fe^\beta(\xi) = \bigl\{v \in \bbR^2 : v \cdot (\xi - \zeta) \leq \fe^\beta(\xi) - \fe^\beta(\zeta) \quad \forall \zeta \in \bbR_+^2 \bigr\}, \label{eq:superdef}
\end{align}
will play a key role. We also introduce notation for the image of $\Uset$ under the superdifferential map via
\begin{align*}
\partial \fe^\beta(\Uset) &= \bigl\{v \in \bbR^2 : \exists \xi \in \ri \Uset : v \in \partial \fe^\beta(\xi)\bigr\}.
\end{align*}
The following lemma gives a useful characterization of $\partial \fe^{\beta}(\Uset)$. The proof is a straightforward exercise in convex analysis and can be found in Appendix \ref{app:lemmas}.

\begin{lemma}\label{lem:insub}
For $h \in \bbR^2$,  $-h \in \partial \fe^\beta(\Uset)$ if and only if  $\fepl^\beta(h) = 0$. Moreover, if $-h \in \partial\fe^\beta(\xi)$ for $\xi \in \ri \Uset$, then $\xi \cdot h+\fe^\beta(\xi)=0$.
\end{lemma}
Concavity implies the existence of one-sided derivatives:
\begin{align*}
\nabla \fe^\beta(\xi \pm) \cdot e_1 &= \lim_{\e \searrow 0} \frac{\fe^\beta(\xi \pm \e e_1) - \fe^\beta(\xi)}{\pm \e}\quad{and}\\
\nabla \fe^\beta(\xi \pm) \cdot e_2 &= \lim_{\e \searrow 0} \frac{\fe^\beta(\xi \mp \e e_2) - \fe^\beta(\xi)}{\mp \e}.
\end{align*}
By Lemma \ref{lem:superdiffprop}\eqref{item-c} these are the two extreme points of the convex set $\partial \fe^\beta(\xi).$ 
The collection of directions of differentiability of $\fe^\beta$ will be denoted by
\begin{align*}
\Diff^\beta = \bigl\{\xi \in \ri\Uset  : \fe^\beta \text{ is differentiable at }\xi\bigr\}.
\end{align*}
\cite[Theorem 25.2]{Roc-70} shows that $\xi \in \Diff^\beta$ is the same as $\nabla \fe^\beta(\xi+) = \nabla \fe^\beta(\xi -)$.

Abusing notation, for $\xi \in \ri \Uset$ define the maximal linear segments
\begin{align*}
\Uset_{\xi\pm}^\beta &= \bigl\{\zeta \in \ri\Uset : \fe^\beta(\zeta) - \fe^\beta(\xi) = \nabla \fe^\beta(\xi\pm) \cdot (\zeta - \xi)\bigr\}=\Uset^\beta_{-\nabla\fe^\beta(\xi\pm)}.
\end{align*}
Although we abuse notation, it should be clear from context whether we are referring to sets indexed by directions or tilts. 

%It is possible that either or both of these sets can degenerate to a point. 
We say $\fe^\beta$ is {\sl strictly concave} at $\xi\in\ri\Uset$ if $\Uset_{\xi-}^{\beta}=\Uset_{\xi+}^{\beta}=\{\xi\}$.
The usual notion of strict concavity on an open  subinterval of $\Uset$ is the same as having our 
strict concavity at $\xi$ for all $\xi$ in the interval.
Let
\begin{align*}
\Uset_{\xi}^\beta &= \Uset_{\xi-}^\beta \cup \Uset_{\xi +}^\beta = [\ximin^\beta,\ximax^\beta],\quad\text{with }\ximin^\beta \cdot e_1 \le \ximax^\beta\cdot e_1.
\end{align*}

Lemma \ref{lem:polymermartin} justifies setting $\Uset_{e_i}^\beta = \{e_i\}$ for $i \in \{1,2\}$, since it implies that the free energy is not locally linear near the boundary.

If $\xi\in\Diff^\beta$ then $\Uset_{\xi-}^\beta=\Uset_{\xi+}^\beta=\Uset_\xi^\beta$ while if $\xi\not\in\Diff^\beta$ then $\Uset_{\xi-}^\beta\cap\Uset_{\xi+}^\beta=\{\xi\}$.
For $h\in\R^2$ the set $\Uset_h^\beta$ is either a singleton $\{\xi\}$
or equals $\Uset_{\xi-}^{\beta}$ or $\Uset_{\xi+}^{\beta}$, for some $\xi\in\ri\Uset$ dual to $h$. In particular, it is a closed nonempty interval.  \smallskip

With the exception of Section \ref{sec:cocycles}, our results are for a fixed $\beta<\infty$. Therefore, in the rest of the paper, except in Section \ref{sec:cocycles}, 
we will assume without loss of generality that $\beta = 1$ and will omit the $\beta$ from our notation.

\subsection{Random polymers as semi-infinite Gibbs measures}\label{sub:DLR}
Given $\w\in\Omega$, integers $\ell\ge k\ge m$, $x\in\V_m$,
and up-right paths $x_{m,k}$ and $x_{\ell,\infty}$ with $x_m=x$ and $x_\ell\ge x_k$, use the point-to-point quenched measures to
define a probability measure $\ker^\w_{k,\ell}(x_{m,\infty},dy_{m,\infty})$ on $(\X_x,\cXx)$ via its integrals of $f \in b\cXx$: 
	\begin{align*}
	\ker_{k,\ell}^\w f(x_{m,\infty})
	&=\int f(y_{m,\infty})\, \ker^\w_{k,\ell}(x_{m,\infty},dy_{m,\infty})\\
	&=\sum_{y_{k,\ell} \in \bbX_{x_k}^{x_\ell}}f(x_{m,k}y_{k,\ell}x_{\ell,\infty})Q^\w_{x_k,x_\ell}(y_{k,\ell})\,.
	\end{align*}
$\ker^\w_{k,\ell}$ is a stochastic kernel from 
$(\X_x,\cXx)$ to $(\X_x,\cXx_{(k,\ell)^c})$; see \cite[Section 7.3]{Ras-Sep-15-ldp}. 
It is also $\cXx_{(k,\ell)}$-proper: if $g\in b\cXx_{(k,\ell)^c}$ and $f\in b\cXx$ then
	\begin{align}\label{proper}
	\ker_{k,\ell}^\w(gf)=g\ker_{k,\ell}^\w f.
	\end{align}

Stochastic kernels push measures forward: $\int f\, d\mu \ker_{k,\ell}^\w=E^\mu[\ker_{k,\ell}^\w f]$.
%Namely, given a probability measure
%$\mu$ on $(\X_x,\cXx)$ 
%we define the probability measure $\mu \ker_{k,\ell}^\w$ on $(\X_x,\cXx)$ by 
%	\begin{align*}
%	\int f\, d\mu \ker_{k,\ell}^\w 
%	&=E^\mu[\ker_{k,\ell}^\w f]
%	= \iint f(\bar X_{m,\infty})\, \mu(dX_{m,\infty})\,\ker^\w_{k,\ell}(X_{m,k},X_{\ell,\infty},d\bar X_{m,\infty})\\
%	&= \sum_{x_{m,k}}\int\Bigl(\sum_{x_{k,\ell}} f(x_{m,\ell}X_{\ell,\infty})\,Q^\w_{x_k,x_\ell}(x_{k,\ell})\Bigr)\one\{X_{m,k}=x_{m,k},X_\ell=x_\ell\} \, \mu(dX_{m,\infty})\,,
%	\end{align*}
%for all bounded measurable functions $f$ on $\X_x$.  (The second sum is over up-right paths $x_{k,\ell}$ that start at $x_k$, which is determined by the first sum. In the integral $\mu(dX_{m,\infty})$, 
%$X_{m,k}$ and $X_\ell$ are determined by $x_{m,k}$ and $x_\ell$ from the two sums. $x_{m,\ell}X_{\ell,\infty}$, with $X_\ell=x_\ell$, means we concatenate the two paths into one path in $\X_x$.)
Thus, they can be composed
%: if $m\le k\le r\le s \le \ell$ and $f$ is bounded $\cXx$-measurable then
%	\[\ker_{k,\ell}^\w\ker_{r,s}^\w f(x_{0,k},x_{\ell,\infty})=\int f(\bar X_{m,\infty}) \ker^\w_{k,\ell}(x_{m,k},x_{\ell,\infty},dX_{m,\infty})\ker^\w_{r,s}(X_{m,r},X_{s,\infty},d\bar X_{m,\infty}).\]
and a computation (Appendix \ref{app:lemmas}) checks the following.

\begin{lemma}\label{ker-consist}
Fix $\w\in\Omega$, $m\in\Z$, and $x\in\V_m$. Then the kernels are consistent: $\ker_{k,\ell}^\w \ker_{r,s}^\w=\ker_{k,\ell}^\w$ for all integers $\ell\ge s\ge r\ge k\ge m$.
\end{lemma}

This consistency along with $\ker^\w_{k,\ell}$ being $\cXx_{(k,\ell)}$-proper mean that the kernels $\{\ker^\w_{k,\ell} : m \leq k \leq \ell\}$ form a {\sl specification}. See \cite[Definition 7.8]{Ras-Sep-15-ldp}.

\begin{definition}
Given $\w\in\Omega$ and $x\in\V_m$, $m\in\Z$, 
a probability measure $\Pi_x$ on $(\X_x,\cXx)$ is said to be a {\sl semi-infinite or rooted Gibbs measure} in environment $\w$ {\sl rooted at} $x$ if for all $\ell\ge k\ge m$ and any bounded measurable function on $\X_x$
we have $E^{\Pi_x}[f\,|\,\cXx_{(k,\ell)^c}]=\ker_{k,\ell}^\w f$. The set of Gibbs measures (or DLR solutions) in environment $\w$ rooted at $x$ is denoted $\DLR_x^\w$.
\end{definition}

Next is a standard characterization of Gibbs measures. See Definition 7.12 and Lemma 7.13 in \cite{Ras-Sep-15-ldp}. For the proof see Appendix \ref{app:lemmas}

\begin{lemma}\label{lm:DLR}
Given $\w\in\Omega$ and $x\in\V_m$, $m\in\Z$, $\Pi_x \in \DLR_x^{\w}$ if and only if for all $\ell\ge k\ge m$ 
	\begin{align}\label{eq:DLR}
	\Pi_x \ker_{k,\ell}^\w=\Pi_x.
	\end{align}
\end{lemma}

Equations \eqref{eq:DLR} are the {\sl Dobrushin-Lanford-Ruelle {\rm(}DLR{\rm)} equations}. Note that the DLR equations  only involve the weights $\{\w_y:y\ge x\}$. Hence, $\DLR_x^\w=\DLR_x^{\overline\w}$ if 
$\w_y=\overline\w_y$ for $y\ge x$. We call  measurability with respect to $\sigma(\w_v:v\ge x)$
{\sl forward-measurability}. 

The next lemma says that our setting is Markovian. The proof is deferred to Appendix \ref{app:lemmas}.

\begin{lemma}\label{lm:Pi-cons}
Given $\w\in\Omega$ and $x\in\V_m$, $m\in\Z$, $\Pi_x\in\DLR_x^\w$  if and only if for all $n\ge m$ 
and all up-right paths $x_{m,n}$ with $x_m=x$:
	\begin{align}\label{Pi-cons}
	\Pi_x(X_{m,n}=x_{m,n})=\Pi_x(X_n=x_n)\,Q^\w_{x_m,x_n}(X_{m,n}=x_{m,n}). 
	%=\frac{\Pi_x(X_n=x_n)\,e^{\sum_{i=0}^{n-1}\w_{x_i}}}{Z_{x,x_n}}.
	\end{align}
\end{lemma}

Due to the above, $\Pi_x\in\DLR_x^\w$ are also called {\sl semi-infinite or rooted quenched polymer measures} in environment $\w$, rooted at $x$.
Note that \eqref{Pi-cons} is the positive-temperature analogue of the definition of a semi-infinite geodesic
in percolation.

The DLR equations \eqref{eq:DLR} show that $\DLR_x^\w$ is a closed convex subset of the compact space $\sM_1(\bbX_x,\sX_x)$, which we view as a subspace of the complex Radon measures on $\bbX_x$.
To see this, note that for $y_{m,n} \in \bbX_{y_m}^{y_n}$ the function $x_{m,\infty} \mapsto \one\{x_{m,n}=y_{m,n}\}$ is bounded and continuous on $\bbX_x$. 
%under the metric described in Section \ref{sub:model}. 
Since $\X_x$ is a compact Polish space
the collection of DLR solutions (being a closed subset) is compact. Since the collection of signed measures on paths equipped with the weak-$*$ topology is a locally convex Hausdorff topological vector space and the unit ball is metrizible in this setting we can apply Choquet's theorem.
By Choquet's theorem \cite[Section 3]{Phe-01}, each element in $\DLR_x^\w$ is a convex integral mixture of extremal elements of $\DLR_x^\w$.

\subsection{Bi-infinite Gibbs measures}\label{sub:biDLR}
Given $\w\in\Omega$ and integers $m \leq n$ define the stochastic kernel $\ker^\w_{m,n}$ from $(\X,\cX)$ to  $(\X,\cX_{(m,n)^c})$ by:  
	\[\ker_{m,n}^\w f(x_{-\infty, \infty})
	%=\int f(X_{-\infty,\infty})\, \ker^\w_{m,n}(x_{-\infty,m},x_{n,\infty},dX_{-\infty,\infty})
	=\sum_{y_{m,n}\in\X_{x_m}^{x_n}}f(x_{-\infty,m}y_{m,n}x_{n,\infty})Q^\w_{x_m,x_n}(y_{m,n})\,.\]

%The above formula shows that $\ker^\w$ is in fact Markovian, in the sense that the probability measure it induces on $(\X_{[m,n]},\cX_{[m,n]})$, given the boundary condition $x_{-\infty,m}$ and $x_{n,\infty}$ only depends 
%on $x_m$ and $x_n$: if $f$ is $\cX_{[m,n]}$-measurable, then $\ker^\w_{m,n}f$ is only a function of $x_m$ and $x_n$.

The kernels $\{\ker^\w_{m,n} :  m,n \in \bbZ, m \leq n \}$ form a specification. 
They are also $\cX_{(m,n)}$-proper: 
if $g\in b\cX_{(m,n)^c}$ and $f\in b\cX$ then 
	\begin{align}\label{proper2}
	\ker_{m,n}^\w(gf)=g\ker_{m,n}^\w f.
	\end{align}
%
%Stochastic kernels push measures forward. Namely, given a probability measure
%$\mu$ on $(\X,\cX)$ 
%we define the probability measure $\mu \ker_{m,n}^\w$ on $(\X,\cX)$ by 
%	\begin{align*}
%	\int f\, d\mu \ker_{m,n}^\w 
%	&=E^\mu[\ker_{m,n}^\w f]
%	= \iint f(\bar X_{-\infty,\infty})\, \mu(dX_{-\infty,\infty})\,\ker^\w_{m,n}(X_{-\infty,m},X_{n,\infty},d\bar X_{-\infty,\infty})\\
%	&= \sum_{x_{m,n}}\iint f(X_{-\infty,m}x_{m,n}X_{n,\infty})\,Q^\w_{x_m,x_n}(X_{m,n}=x_{m,n})\one_{\{X_m=x_m,X_n=x_n\}}  \, \mu(dX_{-\infty,m},dX_{n,\infty})\,,
%	\end{align*}
%for all bounded measurable functions $f$ on $\X$.  %(The second sum is over up-right paths $x_{m,n}$ that start at $x_m$, which is determined by the first sum.)
%1
%One can thus compose kernels: if $m\le k\le \ell \le n$ and $f$ is bounded $\cX$-measurable then
%	\[\ker_{m,n}^\w\ker_{k,\ell}^\w f(x_{-\infty,m},x_{n,\infty})=\int f(\bar X_{-\infty,\infty}) \ker^\w_{m,n}(x_{-\infty,m},x_{n,\infty},dX_{-\infty,\infty})\ker^\w_{k,\ell}(X_{-\infty,k},X_{\ell,\infty},d\bar X_{-\infty,\infty}).\]
%
Moreover, they are consistent:  $\ker_{m,n}^\w \ker_{k,\ell}^\w=\ker_{m,n}^\w$ for all $n\ge \ell\ge k\ge m$. 
The proof is identical to that of Lemma \ref{ker-consist}.
%See Lemma \ref{ker-cons2}.

\begin{definition}
Given $\w\in\Omega$, $\Pi \in \sM_1(\X,\cX)$ is said to be a {\sl bi-infinite Gibbs measure} in environment $\w$ if for all $n\ge m$ and any bounded measurable function on $\X$
we have $E^{\Pi}[f\,|\,\cX_{(m,n)^c}]=\ker_{m,n}^\w f$. We denote the set of bi-infinite Gibbs measures in environment $\w$ by $\biDLR^\w$.
\end{definition}

As in the semi-infinite case, Gibbs measures solve the DLR equations.

\begin{lemma}\label{lm:DLR2}
Given $\w\in\Omega$, a probability measure $\Pi \in \sM_1(\bbX, \sX)$ is a Gibbs measure in environment $\w$ if, and only if, for all $n\ge m$ we have 
	\begin{align}\label{eq:DLR2}
	\Pi \ker_{m,n}^\w=\Pi.
	\end{align}
\end{lemma}

Once again we have a Markovian structure. 

%\begin{lemma}\label{DLR-Markov}
%Fix $\w\in\Omega$ and $\mu\in\DLR^\w$. Then for any $m\le n$ in $\Z$,  any admissible path $x_{m,n}$, any $\ell\le m$ in $\Z\cup\{-\infty\}$, and any $k\ge n$ in $\Z\cup\{\infty\}$ we have $\mu$-almost surely
%	\begin{align*}
%	\mu\{X_{m,n}=x_{m,n}\,|\,\cX_{(m,n)^c}\}
%	&=\mu\{X_{m,n}=x_{m,n}\,|\,X_m,X_n\}\\
%	&=\mu\{X_{m,n}=x_{m,n}\,|\,X_{\ell,m},X_{n,k}\}.
%	\end{align*}
%\end{lemma}

%\begin{proof}
%Let $f(\bar x_{-\infty,\infty})=\one\{\bar x_{m,n}=x_{m,n}\}$. Then $\mu\{X_{m,n}=x_{m,n}\,|\,\cX_{(m,n)^c}\}=\ker^\w_{m,n}f$. Since $f$ is $\cX_{[m,n]}$-measurable, $\ker^\w_{m,n}f(\bar x_{-\infty,m},\bar x_{n,\infty})$ is just a function of $\bar x_m$ and $\bar x_n$.
%This means that $\mu\{X_{m,n}=x_{m,n}\,|\,\cX_{(m,n)^c}\}$ is a measurable function of $X_m$ and $X_n$. Now, 
%for any bounded measurable function $g(X_m,X_n)$ we have
%	\begin{align*}
%	E^\mu[g(X_m,X_n)\mu\{X_{m,n}=x_{m,n}\,|\,\cX_{(m,n)^c}\}]
%	&=E^\mu[g(X_m,X_n)\one\{X_{m,n}=x_{m,n}\}]\\
%	&=E^\mu\bigl[g(X_m,X_n)\mu\{X_{m,n}=x_{m,n}\,|\,X_m,X_n\}].
%	\end{align*}
%The first equality in the claim follows. The other equality comes similarly.
%\end{proof}

%The next lemma is a direct consequence of Lemma \ref{DLR-Markov} and the form of the kernel.

\begin{lemma}\label{lm:Pi-cons2}
Given $\w\in\Omega$, a probability measure $\Pi\in\biDLR^\w$  if and only if for all $n\ge m$ 
and any up-right path $x_{m,n}$ the following holds:
	\begin{align}\label{Pi-cons2}
	\Pi(X_{m,n}=x_{m,n})=\Pi(X_m=x_m,X_n=x_n)\,Q^\w_{x_m,x_n}(X_{m,n}=x_{m,n}). 
	%=\frac{\Pi(X_n=x_n)\,e^{\sum_{i=0}^{n-1}\w_{x_i}}}{Z_{x,x_n}}.
	\end{align}
\end{lemma}

%The proof is given in Appendix \ref{app:lemmas}. 
The proofs are identical to those of Lemmas \ref{lm:DLR} and \ref{lm:Pi-cons}.
Due to this last result, measures $\Pi\in\biDLR^\w$ are also called {\sl bi-infinite quenched polymer measures} in environment $\w$. 
Note that \eqref{Pi-cons2} is the positive-temperature analogue of the definition of a bi-infinite geodesic
in percolation.

Naturally, conditioning DLR solutions on passing through a point produces rooted DLR solutions. 
The proof of the following lemma is a straightforward application of \eqref{Pi-cons} and \eqref{Pi-cons2}.
%See Appendix \ref{app:lemmas} for the short proof.

\begin{lemma}\label{lm:cond-bi-to-root}
Fix $\w\in\Omega$. The following hold:
\begin{enumerate}[label={\rm(\alph*)}, ref={\rm\alph*}]
\item Fix $x\in\V_m$, $m\in\Z$, $\Pi_x\in\DLR_x^\w$, and $y\ge x$ with $y\in\V_n$, $n\ge m$. Assume $\Pi_x(X_n=y)>0$.
Let $\Pi_y$ be the probability measure on $(\X_y,\cX^y)$ defined by 
	\[\Pi_y(X_{n,\ell}=x_{n,\ell})=\Pi_x(X_{n,\ell}=x_{n,\ell}\,|\,X_n=y),\]
for any admissible path $x_{n,\ell}$ starting at $x_n=y$.
Then $\Pi_y\in\DLR^\w_y$.
\item Fix $\Pi\in\biDLR^\w$. Fix $x\in\V_m$, $m\in\Z$, such that $\Pi(X_m=x)>0$. Let $\Pi_x$ be the probability measure on $(\X_x,\cXx)$ defined by
\begin{align}\label{mux}
\Pi_x(X_{m,n}=x_{m,n})=\Pi(X_{m,n}=x_{m,n}\,|\,X_m=x),
\end{align} 
for any up-right path $x_{m,n}$ with $x_m=x$.
Then $\Pi_x\in\DLR^\w_x$.
\end{enumerate}
\end{lemma}

%A consequence of the above is that for any $\ell\le m\le n$ and any admissible path $x_{m,n}$ we have $\mu$-almost surely
%	\begin{align}\label{DLR-Markov-1}
%	&\mu\{X_{m,n}=x_{m,n}\,|\, X_m,X_{-\infty,\ell}\}=\mu\{X_{m,n}=x_{m,n}\,|\,X_m,X_\ell\}. %\quad\text{and}\\
%%	&\mu\{X_{m,n}=x_{m,n}\,|\,X_{n,\infty}\}=\mu\{X_{m,n}=x_{m,n}\,|\,X_n\}.
%	\end{align}

We also study consistent and covariant families of DLR solutions, in the sense of the following two definitions.

\begin{definition}
Given $\w\in\Omega$ we say  $\{\Pi_x:x\in\Z^2\}$ is a family of consistent rooted (or semi-infinite) DLR solutions (in environment $\w$) if for all $x\in\Z^2$, $\Pi_x\in\DLR_x^\w$ and the following holds:
For each $y\in\V_m$, $m\in\Z$,  $x\le y$, $n\ge m$, and for each up-right path $x_{m,n}$ with $x_m=y$
	\[\Pi_x(X_{m,n}=x_{m,n}\,|\,X_m=y)=\Pi_y(X_{m,n}=x_{m,n}).\]
We will denote the set of such families by $\siDLR^\w$.
%When $\Pi_x$ is non-degenerate for all $x\in\Z^2$ we say the family is non-generate.
\end{definition}

Define the shift $\theta_z$ acting on up-right paths %$x_{m,n}$, $m,n\in\Z\cup\{\pm\infty\}$, 
by $\theta_z x_{m,n}=z+x_{m,n}$.

\begin{definition}\label{def:cov}
A family $\{\Pi_x^\w:x\in\Z^2,\w\in\Omega\}$ is said to be a $T$-covariant family of consistent rooted (or semi-infinite) DLR solutions if for each $x\in\Z^2$, $\w\mapsto\Pi_x^\w$ is measurable,
there exists a full-measure $T$-invariant event $\Omega'\subset\Omega$ 
such that for each $\w\in\Omega'$, $\{\Pi_x^\w:x\in\Z^2\}$ is consistent in environment $\w$,
and for all $z\in\Z^2$, $\Pi_{x-z}^{T_z\w}\circ\theta_{-z}=\Pi_x^\w$.
\end{definition}

\section{Main results}\label{sec:results}
\subsection{Semi-infinite polymer measures}
We begin with a definition of directedness. 
For $A\subset\R^2$ and $\xi\in\R^2$ let $\dist(\xi,A)=\inf_{\zeta\in A}\abs{\xi-\zeta}_1$.

\begin{definition}
For a set $A\subset\Uset$, a sequence $x_n\in\Z^2$ is said to be $A$-{\sl directed} if $\abs{x_n}_1\to\infty$ and  the set of limit points of $x_n/\abs{x_n}_1$ is included in $A$.
We say that $\Pi$ is {\sl strongly} $A$-directed if
	\[\Pi\bigl\{(X_n)\text{ is $A$-directed}\bigr\}=1.\]
We say that $\Pi$ is {\sl weakly} $A$-directed if for any $\e>0$
	\[\lim_{n\to\infty}\Pi\bigl\{\dist(X_n/n,A)>\e\bigr\}=0.\]
A family of probability measures is said to be weakly/strongly $A$-directed if each member of the family is.
Sometimes we say {\sl directed into} $A$ instead of $A$-directed, almost surely directed instead of strongly directed, and directed in probability instead of weakly directed.
\end{definition}

When $A=\{\xi\}$ is a singleton, weak directedness into $A$ means $\Pi$ satisfies the weak law of large numbers (WLLN) while
strong directedness means the strong law of large numbers (SLLN) holds, with asymptotic direction $\xi$ in either case. We then say that $\Pi$ satisfies \WLLN\ and \SLLN, respectively.

First, we address the existence of directed DLR solutions. Recall at this point that we set $\beta=1$ throughout this section.

\begin{theorem}\label{th:main1}
There exists an event $\Omexist$ such that $\P(\Omexist)=1$ and 
for every $\w\in\Omexist$ and every $\xi\in\Uset$ there exists 
a consistent family in $\siDLR^\w$
that is strongly $\Uset_{\xi-}$-directed and a consistent family in $\siDLR^\w$ that is strongly $\Uset_{\xi+}$-directed. 
If $\xi\not\in\Diff$ then for each $x\in\Z^2$ the members rooted at $x$, from each family, are different. 
\end{theorem}

The following is an immediate corollary.

\begin{corollary}\label{cor:main1}
For any $\w\in\Omexist$ and for any $\xi\in\ri\Uset$ at which $\fe$ is strictly concave, there exists at least one consistent family in $\siDLR^\w$ satisfying \SLLN.
If, furthermore, $\xi\not\in\Diff$, then there exist at least two such families. 
\end{corollary}

%\begin{remark}
%In both Theorem \ref{th:main1} and Corollary \ref{cor:main1}, for each fixed $\xi\in\Uset$ 
%the $\Uset_{\xi-}$-directed and the $\Uset_{\xi+}$-directed families are shift-covariant,
%but the event $\Omega'$ in Definition \ref{def:cov} depends on $\xi$.
%\end{remark}

For $x\in\V_m$, $m\in\Z$, two  trivial (and degenerate) elements of $\DLR_x^\w$ are given by $\Pi_x^{e_i}=\delta_{x_{m,\infty}}$ with $x_k=x+(k-m)e_i$, $k\ge m$, $i\in\{1,2\}$. 
These two solutions are clearly extreme in $\DLR_x^\w$.
%If not, then $\Pi_x^{e_i}=\alpha\mu+(1-\alpha)\nu$ with $\mu,\nu\in\DLR_x^\w$ and $\alpha\in(0,1)$.
%But $\Pi_x^{e_i}$ being a delta mass forces $\mu$ and $\nu$ to equal that same delta mass.

We say that $\Pi_x\in\DLR_x^\w$ is {\sl nondegenerate} if it satisfies
%\addmath{The definition $\mu(x_{m,n}) > 0$ for all admissible finite paths is a stronger condition and the theorem is still true. If we ever re-visit the global Gibbs stuff, that should be the definition, since we know how to build degenerate Gibbs measures which are not degenerate in this sense in that setting.}
%	\begin{align}\label{nondeg}
%	\mu(y)>0\quad\text{for all }y\ge x.  %\mu(y\in X_\bbullet)
%	\end{align}
	\begin{align}\label{nondeg}
	\Pi_x(x_{m,n})>0\quad\text{for all admissible finite paths with $x_m=x$}.  %\mu(y\in X_\bbullet)
	\end{align}

By \eqref{Pi-cons} this definition 
%This definition 
is equivalent to the weaker condition that every point $y\ge x$ is accessible, i.e.\ $\Pi_x(y)>0$ for all $y\ge x$.
%Recall definition \eqref{nondeg} of non-degeneracy.
%It is easy to see that non-degeneracy 
%\addmath{is equivalent?}
%implies that any finite admissible path has a positive probability.
%
%\begin{lemma}\label{nondeg-path}
%Fix $x\in\V_m$, $m\in\Z$. Then $\mu\in\DLR_x^\w$ is nondegenerate if and only if $\mu(y)>0$ for all $y\ge x$.
%\end{lemma}	
%
%\begin{proof}
%The claim follows from observing that for any up-right path $x_{m,n}$ from $x$ to $y$
%	\[\mu(x_{m,n})=\mu(y)\,Q^\w_{x,y}(x_{m,n})\]
%and that $Q^\w_{x,y}(x_{m,n})>0$. 
%\end{proof}

The next lemma states that outside one null set of weights $\w$, convex combinations of $\Pi_x^{e_i}$ are the only degenerate DLR solutions. 

\begin{lemma}\label{lm:nondeg}
There exists an event $\Omnondeg$ such that $\P(\Omnondeg)=1$ and for all 
$\w\in\Omnondeg$ and $x\in\Z^2$, any solution $\Pi_x\in\DLR_x^\w$ that is not a convex combination of $\Pi_x^{e_i}$, $i\in\{1,2\}$,
is nondegenerate. 
\end{lemma}

The next result is on directedness of DLR solutions.

\begin{theorem}\label{th:main2}
There exists an event $\Omdir$ 
such $\P(\Omdir)=1$ and for all $\w\in\Omdir$, all $x\in\Z^2$ and any extreme nondegenerate solution $\Pi_x\in\DLR_x^\w$ there exists a $\xi\in\ri\Uset$ such that 
one of the following three holds: 
\begin{enumerate}[label={\rm(\alph*)}, ref={\rm\alph*}]
\item\label{case-a} $\Pi_x$ satisfies \WLLN\ and is strongly $\Uset_{\ximax}$-directed or strongly $\Uset_{\ximin}$-directed,
\item\label{case-b} $\Pi_x$ is strongly $\Uset_\xi$-directed, or
\item\label{case-c} $\xi\in\Diff$ and $\Pi_x$ is weakly $\Uset_\xi$-directed and strongly directed into $\Uset_{\ximin}\cup\Uset_{\ximax}$.
\end{enumerate}
\end{theorem}

If $\w\in\Omnondeg$, then Lemma \ref{lm:nondeg} says the only extreme degenerate solutions of the DLR equations are $\Pi_x^{e_i}$, $i\in\{1,2\}$, which are $\{e_i\}$-directed.
Theorem \ref{th:main2} shows that if $\w\in\Omdir$, then there are no nondegenerate extreme DLR solutions directed weakly into $\{e_1\}$ or $\{e_2\}$.

Note that when $\fe$ is differentiable on $\ri\Uset$ we have $\Uset_\xi=\Uset_{\xi\pm}=\Uset_{\ximin}=\Uset_{\ximax}$ for all $\xi\in\Uset$.
When $\fe$ is strictly concave at a point $\xi$ we have $\Uset_\xi=\Uset_{\xi\pm}=\Uset_{\ximin}=\Uset_{\ximax}=\{\xi\}$. 
Thus, the following is an immediate corollary. 

\begin{corollary}\label{cor:main2}
The following hold.
\begin{enumerate}[label={\rm(\alph*)}, ref={\rm\alph*}]
\item Assume $\fe$ is differentiable on $\ri\Uset$. For any $\w\in\Omdir$, for all $x\in\Z^2$, any extreme solution in $\DLR_x^\w$  is strongly $\Uset_\xi$-directed for some $\xi\in\Uset$.
\item Assume $\fe$ is strictly concave on $\ri\Uset$. 
Then for any $\w\in\Omdir$, for all $x\in\Z^2$, any extreme solution in $\DLR_x^\w$ satisfies \SLLN\ for some $\xi\in\Uset$.
\end{enumerate} 
\end{corollary}

We next show existence and uniqueness of DLR solutions.

\begin{theorem}\label{th:main3}
Fix $\xi\in\Diff$ such that $\ximin,\ximax\in\Diff$. There exists a $T$-invariant event $\OmBusPi\subset\Omega$ such that $\P(\OmBusPi)=1$ and for every $\w\in\OmBusPi$ and $x\in\Z^2$, 
there exists a unique weakly $\Uset_\xi$-directed solution $\Pi_x^{\xi,\w}\in\DLR_x^\w$. This $\Pi_x^{\xi,\w}$ is strongly $\Uset_\xi$-directed and for any $\Uset_\xi$-directed sequence $(x_n)$ 
the sequence of quenched point-to-point polymer measures $Q^\w_{x,x_n}$ converges weakly to $\Pi_x^{\xi,\w}$.  
The family $\{\Pi_x^{\xi,\w}:x\in\Z^2,\w\in\Omega\}$ is consistent and $T$-covariant.
\end{theorem}

Our next result shows existence of Busemann functions in directions $\xi$ with $\xi,\ximin,\ximax \in \Diff$ or, equivalently, $\partial \fe(\zeta) = \{h\}$ for some $h$ and all $\zeta \in \Uset_{\xi}$.

\begin{theorem}\label{thm:main:Bus}
Fix $\xi\in\Diff$ such that $\ximin,\ximax\in\Diff$ and let $\{h\} = \partial \fe(\xi)$. There exists a $T$-invariant event $\OmBus$ with $\P(\OmBus)=1$ such that for all  $\w\in\OmBus$, $x,y\in\Z^2$, and all $\Uset_{\xi}$-directed sequences $x_n\in\V_n$, the following limits exist and are equal
	\begin{align}
	B^\xi(x,y,\w)&=\lim_{n\to\infty}\bigl(\log Z_{x,x_n}-\log Z_{y,x_n}\bigr) \label{Bus-lim-nice} \\
	&= \lim_{n\to\infty}\bigl(\log Z_{x,(n)}^h-\log Z_{y,(n)}^h\bigr)-h\cdot(y-x). \label{Bus-lim-p2l-nice}
	\end{align}
Additionally, if $\zeta\in\Diff$ is such that $\zetamin,\zetamax\in\Diff$ and $\xi\cdot e_1<\zeta\cdot e_1$, then for
$\w\in\OmBus\cap\OmBuszeta$ and $x\in\Z^2$, we have
	\begin{align}\label{mono-nice}
	\begin{split}
	&B^\xi(x,x+e_1,\w)\ge B^\zeta(x,x+e_1,\w)\quad\text{and}\\
	&B^\xi(x,x+e_2,\w)\le B^\zeta(x,x+e_2,\w).
	\end{split}
	\end{align}
\end{theorem}

As a consequence of the above theorem, the unique DLR measures from Theorem \ref{th:main3} have  a concrete structure, as the next corollary shows.

\begin{corollary}\label{cor:main3+Bus}
Fix $\xi\in\Diff$ such that $\ximin,\ximax\in\Diff$ and $\w\in\OmBusPi\cap\OmBus$. Then $\Pi_x^{\xi,\w}$ is a Markov chain starting at $x$, with transition probabilities $\pi^{\xi,\w}_{y,y+e_i}=e^{\w_y-B^\xi(y,y+e_i,\w)}$, $y\in\Z^2$, $i\in\{1,2\}$.
The family $\{\Pi_x^{\xi,\w}:x\in\Z^2,\w\in\OmBusPi\cap\OmBus\}$ is $T$-covariant.
\end{corollary}

In contrast to Theorem \ref{th:main3}, Theorem \ref{th:main1} demonstrated non-uniqueness at points of non\hyp{}differentiability of $\fe$. It is conjectured that $\Diff=\ri\Uset$; if true, then Theorem \ref{th:main3} would cover all directions in $\ri\Uset$ and there would not exist directions to which the non-uniqueness claim in Theorem \ref{th:main1} would apply. The event on which Theorem \ref{th:main3} holds, however, depends on the direction chosen. It leaves open the possibility of random directions of non-uniqueness. Our next result says that under a mild regularity assumption, with the exception of one null set of environments, uniqueness holds for all but countably many points in $\Uset$. The assumption we need for this is:

\begin{align}\label{La-reg}
\begin{split}
&\text{$\fe$ is strictly concave at all $\xi\not\in\Diff$, or equivalently}\\
&\text{$\fe$ is differentiable at the endpoints of its linear segments.}
\end{split}
\end{align}

The above condition is also equivalent to the existence of a countable dense set $\Ddense\subset\Diff$ 
such that for each $\zeta\in\Ddense$ we also have $\zetamin,\zetamax\in\Diff$. 

Assume \eqref{La-reg} and fix such a set $\Ddense$.
Using monotonicity \eqref{mono-nice} we define processes $B^{\xi\pm}(x,x+e_i,\w)$ for $\xi\in\ri\Uset$ and \label{OmBusnice}$\w\in\OmBusnice=\bigcap_{\xi\in\Ddense}\OmBus$:
\begin{align}
\label{eq:BuseProcess} 
\begin{split}
B^{\xi+}(x,x+e_i)&=\lim_{\substack{\zeta \cdot e_1 \searrow \xi \cdot e_1\\ \zeta \in \Ddense}}B^\zeta(x,x+e_i)\quad\text{and}\\ 
B^{\xi-}(x,x+e_i)&=\lim_{\substack{\zeta \cdot e_1 \nearrow \xi \cdot e_1\\ \zeta \in \Ddense}}B^\zeta(x,x+e_i).
\end{split}
\end{align}
For $\w\in\OmBusnice$ let
 	\begin{align}\label{Usetnonuniq}
	\Usetnonuniq=\bigl\{\xi\in\ri\Uset:\exists y\ge x:B^{\xi-}(y,y+e_1,\w)\ne B^{\xi+}(y,y+e_1,\w)\bigr\}.
	\end{align}
For $\w\not\in\OmBusnice$  set $\Usetnonuniq=\varnothing$. Note that for any $\w\in\Omega$, $\Usetnonuniq$ is countable. 

The following theorem can be viewed as our main result. Its primary content is contained in part III, which shows that the discontinuity set of the Busemann processes ahead of $x$ defined in \eqref{Usetnonuniq} is exactly the set of directions for which uniqueness of DLR solutions rooted at $x$ fails. This connection has not been observed before in the positive or zero temperature literature. As a consequence, we  obtain that the set of directions for which uniqueness may fail is countable, which is new in positive temperature. As noted in the introduction, this connection also provides an avenue for  answering the question of whether or not on a single event of full measure uniqueness holds simultaneously in all directions.
%If true, that result would be new in any setting.

\begin{theorem}\label{th:main4}
Assume \eqref{La-reg}. There exists an event $\OmBusPiall$ with $\P(\OmBusPiall)=1$ such that the following hold for all $x\in\Z^2$.\smallskip

\noindent{\rm I. Structure of $\Usetnonuniq$:}

\begin{enumerate}[label={\rm(\alph*)}, ref={\rm\alph*}]
\item\label{th:main4:b} For any $\w\in\OmBusPiall$, $(\ri\Uset)\setminus\Diff\subset\Usetnonuniq$. 
For each $\xi\in\Diff$,  $\P\{\xi\in\Usetnonuniq\}=0$.
\item\label{th:main4:b'} For any $\w\in\OmBusPiall$, $\Usetnonuniq$ is supported outside the linear segments of $\fe$:
For any $\xi\in\ri\Uset$ with $\ximin\ne\ximax$, $[\ximin,\ximax]\cap\Usetnonuniq=\varnothing$. 
\item\label{th:main4:b''} For any distinct $\eta,\zeta\in\Uset$,  $\P([\eta,\zeta]\cap\Usetnonuniqz\neq\varnothing)\in\{0,1\}$. If $[\eta,\zeta]\cap\ri\Uset\subset\Diff$ and $\P\{[\eta,\zeta]\cap\,\Usetnonuniq\neq\varnothing\}=1$,
%then $\P\bigl\{[\eta,\zeta]\cap\,\Usetnonuniq\text{ is dense in }[\eta,\zeta]\bigr\}=1.$
%then $\P\bigl\{[\eta,\zeta]\cap\,\Usetnonuniq\text{ is infinite}\bigr\}=1.$    %HERE
then the set of $\xi \in [\eta,\zeta]$ satisfying $\P\bigl\{\xi \text{ is an accumulation point of } \Usetnonuniqz\}=1$ is infinite and has no isolated points.
\end{enumerate}

\noindent{\rm II. Directedness of DLR solutions:}

\begin{enumerate}[resume,label={\rm(\alph*)}, ref={\rm\alph*}]
\item\label{th:main4:c} For any $\w\in\OmBusPiall$, every nondegenerate extreme solution is strongly $\Uset_\xi$-directed for some $\xi\in\ri\Uset$. The only degenerate extreme solutions are $\Pi_x^{e_i}$, $i\in\{1,2\}$.
\item\label{th:main4:d} For any $\w\in\OmBusPiall$ and $\xi\in\Uset$ 
any weakly $\Uset_\xi$-directed solution is strongly $\Uset_\xi$-directed.
\end{enumerate}\smallskip

\noindent{\rm III. $\Usetnonuniq$ and the uniqueness of DLR solutions:}

\begin{enumerate}[resume,label={\rm(\alph*)}, ref={\rm\alph*}]
\item\label{th:main4:e} For any $\w\in\OmBusPiall$ and $\xi\in\Uset\setminus\Usetnonuniq$ there exists a unique strongly $\Uset_\xi$-directed solution $\Pi_x^{\xi,\w}\in\DLR_x^\w$. Moreover, $\Pi_x^{\xi,\w}$ is an extreme point of $\DLR_x^\w$
and for any $\Uset_\xi$-directed sequence $(x_n)$ 
the sequence  $Q^\w_{x,x_n}$ converges weakly to $\Pi_x^{\xi,\w}$. The family $\{\Pi_x^{\xi,\w}:x\in\Z^2\}$ is consistent.
\item\label{th:main4:f} For any $\w\in\OmBusPiall$ and $\xi\in\Usetnonuniq$ there exist at least two extreme strongly $\Uset_\xi$-directed solutions in $\DLR_x^\w$.
\end{enumerate}
\end{theorem}

When $\fe$ is strictly concave, i.e.~$\Uset_\xi=\{\xi\}$ for all $\xi\in\Uset$, the above theorem 
states that outside one null set of weights $\w$, and except for an $\w$-dependent set of directions (countable and possibly empty), 
there is a unique DLR solution in environment $\w$ satisfying \WLLN\ (and in fact \SLLN). 
%HERE
%When $\fe$ is differentiable at all points 
%%we can take $[\eta,\zeta]=\Uset$ in Theorem \ref{th:main4}\eqref{th:main4:b''} and 
%we get that if $\Usetnonuniq$ is not almost surely empty, then it is almost surely infinite. %is dense in $\Uset$.

\subsection{The competition interface}

An easy computation, done in Appendix \ref{app:lemmas}, checks:

\begin{lemma}\label{lm:p2pQ-pi} 
For $x\le y$ the quenched polymer measure $Q^\w_{x,y}$ is the same as the distribution of the backward Markov chain 
 starting at $y$ and taking steps in $\{-e_1,-e_2\}$ with transition probabilities
 	\[\backpix_{u,u-e_i}(\w)=\frac{e^{\w_{u-e_i}}Z_{x,u-e_i}}{Z_{x,u}},\quad
	%\text{and}\quad\backpix_{x+ke_i,x+(k-1)e_i}(\w)=1,\quad u>x,\ \ k\in\N,\ i\in\{1,2\}.\]
	u\ge x.\]
\end{lemma}

Couple the backward Markov chains $\{Q^\w_{x,y}:y\ge x\}$ by a quenched probability measure $Q_x^\w$ on the space $\T_x$ of 
trees that span $x+\Z_+^2$. Precisely, for each $y\in x+\Z_+^2\setminus\{0\}$ choose a parent 
$\gamma(y)=y-e_i$ with probability $\backpix_{y,y-e_i}(\w)$, $i\in\{1,2\}$. We denote the random tree by $\cT^\w_x\in\T_x$.
%
%Let $\{\unif(y):y\in\Z^2\}$ be i.i.d.\ uniform($0,1$) random variables with probability distribution denoted by $\bfP$ defined on the probability space $\Unif=[0,1]^{\Z^2}$, equipped with the product topology and Borel $\sigma$-algebra. 
%Given $x\in\Z^2$, $\w\in\Omega$, and $\unif\in\Unif$ let $\cT^\w_x(\unif)$ denote the $x+\N^2$-indexed $\{0,1\}$-valued sequence where at each site $y\in x+\N^2$ we put $0$ if $\unif(y)<\backpix_{y,y-e_1}(\w)$ and $1$ otherwise.
%Endow $\{0,1\}^{\N^2}$ with the product topology and the corresponding Borel $\sigma$-algebra. Then $(\w,\unif)\in\Omega\times\Gamma\mapsto\cT^\w_x(\unif)\in\{0,1\}^{\Z^2}$ is measurable.
%
%By viewing a $0$ at $y$ as an arrow from $y$ to $y-e_1$ and a $1$ as an arrow from $y$ to $y-e_2$, and by adding arrows from $x+ke_i$ pointing to $x+(k-1)e_i$, $i\in\{1,2\}$, $k\in\N$, we can view
%$\cT^\w_x(\unif)$ as a spanning tree of $x+\Z_+^2$.
For any $y\ge x$ there is a unique up-right path from $x$ to $y$ on $\cT^\w_x$. Lemma \ref{lm:p2pQ-pi} implies that the distribution of this path under $Q_x^\w$ is exactly the polymer measure $Q^\w_{x,y}$.
%Therefore, intuitively speaking, this tree should contain all the information about the model.

Fix the starting point to be $x=0$. Consider the two (random) subtrees $\cT^\w_{0,e_i}$ of $\cT^\w_0$, rooted at $e_i$, $i\in\{1,2\}$. Following \cite{Geo-etal-15}, define the path $\CI_n^\w$ such that 
$\CI_0^\w=0$ and for each $n\in\N$ and $i\in\{1,2\}$, $\CI_{n}^\w-\CI_{n-1}^\w\in\{e_1,e_2\}$ and $\{\CI_n^\w + ke_i : k\in\N\} \subset \cT^\w_{0,e_i}$.
The path $\{(1/2,1/2)+\CI_n^\w:n\in\Z_+\}$ threads in between the two trees $\cT^\w_{0,e_i}$, $i\in\{1,2\}$, and is hence called the {\sl competition interface}. See Figure \ref{fig:cif}.

\begin{figure}[h]
\begin{tikzpicture}[>=latex,scale=0.5]
%\draw[->] (0,0)--(10,0);
%\draw[->] (0,0)--(0,10);

\draw(-0.37,-0.35)node{$0$};
\draw(-0.57,1)node{$e_2$};
\draw(1,-0.5)node{$e_1$};

% TREE

\draw[line width=1pt](9,0)--(0,0)--(0,9);     %width of tree lines
\draw[line width=1pt](2,0)--(2,4)--(5,4)--(5,7)--(7,7)--(7,9)--(9,9);
\draw[line width=1pt](1,0)--(1,3);
\draw[line width=1pt](0,4)--(1,4)--(1,5)--(2,5)--(2,7);
\draw[line width=1pt](2,5)--(3,5);
\draw[line width=1pt](4,4)--(4,5);
\draw[line width=1pt](2,6)--(3,6)--(3,8)--(5,8)--(5,9)--(6,9);
\draw[line width=1pt](3,6)--(4,6);
\draw[line width=1pt](3,7)--(4,7);
\draw[line width=1pt](6,7)--(6,8);
\draw[line width=1pt](0,6)--(1,6)--(1,8)--(2,8)--(2,9)--(3,9);
\draw[line width=1pt](4,8)--(4,9);
\draw[line width=1pt](0,9)--(1,9);
\draw[line width=1pt](5,4)--(7,4)--(7,5)--(9,5)--(9,8);
\draw[line width=1pt](6,4)--(6,6)--(7,6);
\draw[line width=1pt](8,5)--(8,8);
\draw[line width=1pt](2,1)--(8,1)--(8,2)--(9,2);
\draw[line width=1pt](2,2)--(4,2)--(4,3)--(6,3);
\draw[line width=1pt](9,0)--(9,1);

\draw[line width=1pt](5,1)--(5,2)--(7,2)--(7,3);
\draw[line width=1pt](8,2)--(8,4)--(9,4);
\draw[line width=1pt](3,2)--(3,3);
\draw[line width=1pt](8,3)--(9,3);

%%% Competition interface

\draw[line width=2.0pt](.5,.5)--(0.5,3.5)--(1.5,3.5)--(1.5, 4.5)--(3.5,4.5)--(3.5,5.5)--(4.5,5.5)--(4.5,7.5)--(5.5,7.5)--(5.5,8.5)--(6.5,8.5)--(6.5, 9.5);   %width of comp int
%\draw[line width=3.5pt, densely dotted](.5,.5)--(0.5,3.5)--(1.5,3.5)--(1.5, 4.5)--(4.5,4.5)--(4.5,7.5)--(5.5,7.5)--(5.5,8.5)--(6.5,8.5)--(6.5, 9.5);    %style of comp int line
%\draw[line width=3.5pt, dashed](.5,.5)--(0.5,3.5)--(1.5,3.5)--(1.5, 4.5)--(4.5,4.5)--(4.5,7.5)--(5.5,7.5)--(5.5,8.5)--(6.5,8.5)--(6.5, 9.5);
\foreach \x in {0,...,9}{
              \foreach \y in {0,...,9}{
\shade[ball color=light-gray](\x,\y)circle(1.7mm);    %size of bullet
}} 

\end{tikzpicture}
\caption{\small The competition interface shifted by $(1/2, 1/2)$ {\rm(}solid line{\rm)} separating
the   subtrees $\cT_{0,e_1}^\w$ and $\cT_{0,e_2}^\w$. }
\label{fig:cif}  
\end{figure}
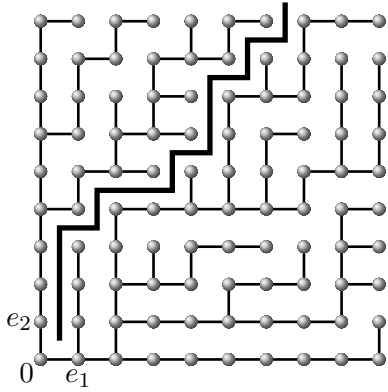

By Lemma 2.2 in \cite{Geo-etal-15} there exists a unique such path and its distribution under $Q_0^\w$ is that of a Markov chain that starts at $0$ and has transitions
	\begin{align*}
	\pici_{y,y+e_i}=\frac{e^{-\w_{y+e_i}}/Z_{0,y+e_i}}{e^{-\w_{y+e_1}}/Z_{0,y+e_1}+e^{-\w_{y+e_2}}/Z_{0,y+e_2}}\,.
	\end{align*}
The partition functions $Z_{0,y}$ in \cite{Geo-etal-15} include the weight $\w_y$ and exclude $\w_0$, while we do the opposite. 
This is the reason for which our formula for $\pici$ is not as clean as the one in \cite{Geo-etal-15}.

The above says that $\phi_n^\w$ is in fact a random walk in random environment, but with highly correlated transition probabilities.
Our next result concerns the law of large numbers.

\begin{theorem}\label{th:cif}
Assume \eqref{La-reg}.
There exists a measurable $\cid:\Omega\times\T_0\to\ri\Uset$ 
and an event $\Omcif$ such that $\P(\Omcif)=1$ and
for every $\w\in\Omcif$:
\begin{enumerate}[label={\rm(\alph*)}, ref={\rm\alph*}]
\item\label{th:cif:a} The competition interface has a strong law of large numbers: 
	\[Q_0^\w\{\CI^\w_n/n\to\cid\}=1.\]
\item\label{th:cif:c} $\cid$ has cumulative distribution function  
	\begin{align}\label{cid-CDF}
	Q_0^\w\{\cid\cdot e_1\le \xi\cdot e_1\}=e^{\w_0-B^{\xi+}(0,e_1,\w)}\quad\text{for $\xi\in\ri\Uset$}.
	\end{align}
Thus, $Q_0^\w(\cid=\xi)>0$ if and only if $B^{\xi-}(0,e_1,\w)\ne B^{\xi+}(0,e_1,\w)$. 
\item\label{th:cif:b} $\cid(\w,\acdot)$ is supported outside the linear segments of $\fe$: If $\eta,\zeta\in\ri\Uset$ are such that $\fe$ is linear on $]\eta,\zeta[$, then $Q_0^\w(\eta\cdot e_1<\cid\cdot e_1<\zeta\cdot e_1)=0$.
\item\label{th:cif:d} For any $\xi\in\ri\Uset$, $\E Q_0^\w(\cid=\xi)>0$ if and only if $\xi\in(\ri\Uset)\setminus\Diff$. 
\end{enumerate}
\end{theorem}

Part \eqref{th:cif:d} in the above theorem says that if $\fe$ is differentiable at all points then the distribution of $\cid$ induced by the averaged measure $Q_0^\w(d\cT_0^\w)\P(d\w)$ is continuous. Even in this case, part \eqref{th:cif:c} leaves open the possibility that for a fixed $\w$ the distribution of $\cid$ under the quenched measure $Q_0^\w$ has atoms. 
%Recall though that $\Uset_0^\w$ is expected to be empty and hence part \eqref{th:cif:c} would imply that for each $\w\in\Omcif$, the distribution of $\cid$ under $Q_0^\w$ is continuous.

\subsection{Bi-infinite polymer measures}
Theorem \ref{th:main3} and a variant of the Burton-Keane lack of space argument \cite{Bur-Kea-89} allow us to prove that deterministically $\Uset_{\xi}$-directed bi-infinite Gibbs measures do not exist if $\Uset_{\xi}\subset \Diff$.

\begin{theorem}\label{thm:nodirbi}
Suppose that $\xi,\ximin,\ximax\in\Diff$.
Then there exists an event $\Ombi$ with $\P(\Ombi) =1$ such that for all $\w \in \Ombi$ there is no weakly $\Uset_{\xi}$-directed measure  $\Pi \in \biDLR^{\w}$.
\end{theorem}

%The lack of space argument requires covariance, so we are not able to extend this result to all directions in $\Usetnonuniq$ under condition \eqref{La-reg}. 
 
We now turn to non-existence of covariant bi-infinite Gibbs measures. A similar question has been studied for spin systems including the random field Ising model; see \cite{Aiz-Weh-90, Arg-etal-14, New-97, Weh-Was-16}.
\begin{definition}\label{de:scGibbs}
A $T$-\emph{covariant bi-infinite Gibbs measure} or \emph{metastate} is an $\sM_1(\bbX, \cX)$-valued random variable $\Pi^{\w}$ satisfying the following:
\begin{enumerate}[label={\rm(\alph*)}, ref={\rm\alph*}]
\item \label{de:scGibbs1} The map $\Omega\to\sM_1(\bbX,\cX):\w \mapsto \Pi^{\w}$ is measurable.
%$((\Omega, \sF),(\sM_1(\bbX, \cX), \sB_{\sM_1(\bbX, \cX)}))$ measurable.
\item \label{de:scGibbs2} $\bbP\bigl(\Pi^{\w} \in \biDLR^{\w}\bigr) = 1$.
\item  \label{de:scGibbs3}For each $z \in \bbZ^2$, $\bbP\left(\Pi^{T_z \w}\circ \theta_{-z} = \Pi^{\w}\right) = 1.$
\end{enumerate}
\end{definition}
A quick proof checks that not only do metastates not exist, but in fact there are no shift-covariant measures on $\bbX$. This can be compared to the corresponding result showing non-existence of metastates for the random field Ising model, proven in \cite{Weh-Was-16}, where the mechanism is different.
\begin{lemma}\label{lem:nocovbi}
There does not exist a random variable satisfying Definition \ref{de:scGibbs}\eqref{de:scGibbs1} and Definition \ref{de:scGibbs}\eqref{de:scGibbs3}.
\end{lemma}

\section{Shift-covariant cocycles}\label{sec:coc}
We now introduce our main tools, cocycles and correctors, and address their existence and regularity properties.

\begin{definition}\label{def:covcoc}
A {\sl shift-covariant cocycle} is a Borel-measurable function $B: \Z^2 \times \Z^2 \times \Omega \to \bbR$ which satisfies the following for all $x,y,z\in\Z^2$:
\begin{enumerate}[label={\rm(\alph*)}, ref={\rm\alph*}]
\item (Shift-covariance) $\P\bigl\{B(x+z,y+z,\w) = B(x,y,T_z \w)\bigr\}=1$.
\item\label{def:covcoc:b} (Cocycle property) $\P\bigl\{B(x,y) + B(y,z)=B(x,z) \bigr\}=1$.
\end{enumerate}
\end{definition}

\begin{remark}
We will also use the term cocycle to denote a function satisfying 
Definition \ref{def:covcoc}\eqref{def:covcoc:b} only when $x,y,z\ge u$ for some $u\in\Z^2$.
\end{remark}

As has already been done in the above definition, we will typically suppress the $\w$ from the arguments unless it adds clarity. 
A shift-covariant cocycle is said to be $L^p(\bbP)$ if $\bbE[\abs{B(0,e_i)}^p]< \infty$ for  $i\in\{1,2\}$. 

We are interested in cocycles that are consistent with the weights $\w_x(\w)$ in the following sense:
\begin{definition}
For $\beta \in (0,\infty]$, a shift-covariant cocycle $B$ satisfies $\beta$-{\sl recovery} if for all $x\in\Z^2$ and $\P$-almost every $\w$:
	\begin{align}
	&e^{-\beta B(x,x+e_1,\w)}+e^{-\beta B(x,x+e_2,\w)}=e^{-\beta\w_x(\w)},\quad\text{if }0<\beta<\infty,\label{beta-rec}\\
	&\min\bigl\{B(x,x+e_1,\w),B(x,x+e_2,\w)\bigr\}=\w_x(\w),\quad\text{if }\beta=\infty.\notag
	\end{align}
Such cocycles are called {\sl correctors}.
\end{definition}

For a shift-covariant $L^1(\P)$ cocycle define the random vector $h(B)\in\R^2$ via 
	\[h(B)\cdot e_i=-\E[B(0,e_i)\,|\,\sI]\] 
where $\sI$ is the $\sigma$-algebra generated by $T$-invariant events.

The next result is a special case of an extension of Theorem A.3 of \cite{Geo-etal-15} to the stationary setting. 
See Appendix \ref{app:BuseShape}. 
Alternatively, one could pass through the ergodic decomposition theorem.  

\begin{theorem}\label{thm:BuseShape}
Fix $\beta\in(0,\infty]$.
Suppose  $B$ is a shift-covariant $L^1(\bbP)$ $\beta$-recovering cocycle.
%Then there exists a random vector $h(B):\Omega \to \R^2$ in $L^1(\P)$  which satisfies $\P\{h = h\circ T_x\}=1$
Then 
\begin{align}\label{shape-B}
\lim_{n\to\infty} \max_{x\in n\Uset\cap\Z^2_+} \frac{\abs{B(0,x) + h(B)\cdot x}}{n} = 0\quad\P\text{-almost surely}.
\end{align} 
\end{theorem}

The next lemma shows that $\beta$-recovering covariant cocycles are naturally indexed by elements of the superdifferential $\partial \fe^{\beta}(\Uset)$. This explains why we only consider cocycles with mean vectors lying in the superdifferential when we construct recovering cocycles in the next subsection. A similar observation in FPP appears in \cite[Theorem 4.6]{Dam-Han-14}.

\begin{lemma}\label{lm:h-xi}
Assume the setting of Theorem \ref{thm:BuseShape}.
The following hold.
\begin{enumerate}[label={\rm(\alph*)}, ref={\rm\alph*}]
\item $-h(B) \in \partial \fe^\beta (\Uset)$ almost surely.
\item If $-\bbE[h(B)] \in\partial\fe^\beta (\xi)$ for  $\xi \in \Uset$ then $-h(B) \in \partial \fe^\beta (\xi)$ almost surely.
\item\label{lm:h-xi:c} If $-\bbE[h(B)]\in\ext\partial\fe^\beta(\xi)$ for some $\xi \in \Uset$ then $h(B) = \bbE[h(B)]$ $\P$-a.s.
\end{enumerate}
\end{lemma}

\begin{proof}
Iterating the recovery property shows that almost surely
\begin{align*}
1 &=  \sum_{x \in n \Uset \cap \bbZ_+^2} Z_{0,x}^{\beta}e^{-\beta B(0,x)},\quad\text{if }\beta<\infty,\text{ and}\\
0 &= \max_{x \in n \Uset \cap \bbZ_+^2}\left\{G_{0,x} - B(0,x)\right\},\quad\text{if }\beta=\infty.
\end{align*}
Take logs, divide by $n \beta $ if $\beta < \infty$ and $n$ if $\beta = \infty$ then send $n\to\infty$ to get
\begin{align*}
0 &= \max_{\xi \in \Uset}\bigl\{\fe^\beta(\xi)+h(B)\cdot\xi\bigr\} = \fepl^\beta(h(B))\quad\P\text{-almost surely}.
\end{align*}
The first equality comes by an application of \eqref{shape} and Theorem \ref{thm:BuseShape} and a fairly standard argument (e.g.\ the proof of Lemma 2.9 in \cite{Ras-Sep-14}). The second equality is \eqref{eq:ppplduality}.
By Lemma \ref{lem:insub}, the above implies $-h(B) \in \partial \fe^\beta(\Uset)$.

Since $\fepl^\beta(h(B)) = 0$, we have almost surely $\xi \cdot h(B) + \fe^\beta(\xi) \leq 0$ for any $\xi\in\Uset$.
If now $\xi$ is such that $-\bbE[h(B)]\in \partial \fe^\beta(\xi)$, then again by Lemma \ref{lem:insub}
$\xi \cdot \bbE[ h(B)] + \fe^\beta(\xi) = 0$,
and therefore we must have $\xi \cdot h(B) + \fe^\beta(\xi) = 0$ almost surely. Again, we deduce that  $-h(B) \in \partial\fe^\beta(\xi)$ almost surely. 

If in addition we know  that $-\bbE[h(B)]\in\ext\partial\fe^\beta(\xi)$ then we must have $h(B) = \bbE[h(B)]$ almost surely by definition of an extreme point. 
\end{proof}

Before discussing existence of shift-covariant cocycles, we mention a few more basic properties of the superdifferential $\partial \fe^{\beta}(\Uset)$. 

\begin{lemma}\label{euler}
For all $\xi\in\ri\Uset$ we have $\xi\cdot\nabla\fe^\beta(\xi\pm)=\fe^\beta(\xi)$.
\end{lemma}

\begin{proof}
Fix $\e>0$ and let $t=(\e+\xi\cdot e_1)/(\xi\cdot e_1)$ and $\delta=\e(\xi\cdot e_1)/(\xi\cdot e_1+\e)$.
Then $\xi+\e e_1=t(\xi-\delta e_2)$.  Rearranging terms and using homogeneity of $\fe^\beta$ we get
	\[(\xi\cdot e_1)\,\e^{-1}\bigl(\fe^\beta(\xi+\e e_1)-\fe^\beta(\xi)\bigr)
	+(\xi\cdot e_2)\,\delta^{-1}\bigl(\fe^\beta(\xi)-\fe^\beta(\xi-\delta e_2)\bigr)=\fe^\beta(\xi).\]
Take $\e$ and $\delta$ to $0$ to get $\xi\cdot\nabla\fe^\beta(\xi+)=\fe^\beta(\xi)$. The other identity is similar.
\end{proof}

\begin{lemma}\label{lem:superdiffprop} The superdifferential map has the following properties:
\begin{enumerate}[label={\rm(\alph*)}, ref={\rm\alph*}]
\item\label{item-b} Let $\xi,\xi'\in\ri\Uset$, $-h\in\partial\fe^\beta(\xi)$, and $-h' \in \partial \fe^\beta(\xi')$. 
If for some $i \in \{1,2\}$, $h \cdot e_i < h' \cdot e_i$ then $h \cdot e_{3-i} > h' \cdot e_{3-i}$. Consequently, if $h \cdot e_i = h' \cdot e_i$ then $h = h'$. If for some $i \in \{1,2\}$, $\xi\cdot e_i<\xi'\cdot e_i$, then $h\cdot e_i\le h'\cdot e_i$.  
\item\label{item-a} $\partial \fe^\beta(\Uset)$ is closed in $\bbR^2$; if $\xi_n \to \xi \in \ri\Uset$ and $h_n \to h$ with $-h_n \in \partial \fe^\beta(\xi_n)$, then $-h \in \partial \fe^\beta(\xi)$.  If $h \in-\partial \fe^\beta(\Uset)$ and $\e >0$, there exist $h',h'' \in -\partial \fe^\beta(\Uset)$ with $h\cdot e_1 - \e<h'\cdot e_1<h\cdot e_1<h''\cdot e_1<h\cdot e_1 + \e$.
\item\label{item-c} For each $\xi \in  \ri\Uset$, $\partial \fe^\beta(\xi)=[\nabla\fe^\beta(\xi+),\nabla\fe^\beta(\xi-)]$. This line segment is nontrivial for countably many $\xi \in \ri\Uset$.
\item\label{item-d} For each $h \in \bbR^2$ there exists a unique $t \in \bbR$ so that $-h+t(e_1 + e_2) \in \partial \fe^\beta(\Uset)$.
\end{enumerate}
\end{lemma}

\begin{proof}
Fix $h \in \bbR^2$. Lemma \ref{lem:insub} and the identity
\begin{align*}
\fepl^\beta(h -t(e_1 + e_2)) = \fepl^\beta(h) - t
\end{align*}
imply that the unique $t$ for which $-h+t(e_1 + e_2) \in \partial \fe^\beta(\Uset)$ is given by $t=\fepl^\beta(h)$. This proves \eqref{item-d}.

Take $\xi,\xi',h,h'$ as in \eqref{item-b}. By Lemma \ref{lem:insub} we have $h\cdot \xi+\fe^\beta(\xi)=0$. Then 
	\[0 = \fepl^{\beta}(h')\ge h'\cdot\xi+\fe^\beta(\xi) = h\cdot\xi+\fe^\beta(\xi)+(h'-h)\cdot\xi=(h'-h)\cdot\xi.\] 
That is, $h\cdot\xi\ge h'\cdot\xi$. Similarly, $h\cdot\xi'\le h'\cdot\xi'$. Part \eqref{item-b} follows. 
For example, if 
%$h\cdot e_1=h'\cdot e_1$ then the two inequalities say that $h\cdot e_2\ge h'\cdot e_2$ and $h\cdot e_2\le h'\cdot e_2$, i.e.\ that $h\cdot e_2=h'\cdot e_2$.  If, on the other hand, 
$h>h'$ coordinatewise, then we would have $h\cdot\xi'>h'\cdot\xi'$, a contradiction. 
And now if we have at the same time $\xi\cdot e_1<\xi'\cdot e_1$ and $h\cdot e_1>h'\cdot e_1$, then 
we would have $h\cdot e_2<h'\cdot e_2$ and hence $0\ge(h-h')\cdot\xi'>(h-h')\cdot\xi\ge0$, again a contradiction.

Suppose that $h_n$ is a Cauchy sequence such that $-h_n\in\partial \fe^\beta(\Uset)$; let $h$ be its limit point in $\bbR^2$. By definition, there exist $\xi_n \in \Uset$ with $-h_n \in \partial \fe^\beta(\xi_n)$. Since the $\xi_n$ lie in a compact set, there is a further subsequence along which $\xi_n$ converges to some $\xi \in \Uset$. Since $\fe^\beta$ is continuous on $\Uset$, we may pass to the limit in (\ref{eq:superdef}) to see that $-h \in \partial \fe^\beta(\xi)$. This shows that $\partial\fe^\beta(\Uset)$ is closed as well as the last statement in
\eqref{item-a}.

Now suppose that $-h \in \partial \fe^{\beta}(\xi)$ for $\xi \in \ri\Uset$ and take $\e>0$. 
Let 
	\[h_\e=h+\e e_1-\fepl^{\beta}(h+\e e_1)(e_1 + e_2).\]
Then part \eqref{item-d} says that $-h_\e\in\partial\fe^\beta(\Uset)$. Furthermore, taking $n\to\infty$ in \eqref{1-Lip} implies that
	\[\abs{\fepl^\beta(h+\e e_1)}=\abs{\fepl^\beta(h)-\fepl^\beta(h+\e e_1)}\le \e.\]
Hence, $h_\e\cdot e_1=h\cdot e_1+\e-\fepl^{\beta}(h+\e e_1)$ satisfies
	\[h\cdot e_1\le h_\e\cdot e_1\le h\cdot e_1+2\e.\]
The first inequality must be strict for otherwise, if $h\cdot e_1=h_\e\cdot e_1$, then part \eqref{item-b} would imply that $h_\e\cdot e_2=h\cdot e_2$ which implies 
$\fepl^\beta(h+\e e_1)=0$ and hence $h_\e=h+\e e_1$. But then this would imply that $h_\e\cdot e_1>h\cdot e_1$,  
a contradiction.  A similar reasoning shows that
	\[h_{-\e}=h-\e e_1-\fepl^{\beta}(h-\e e_1)(e_1 + e_2)\]
satisfies $-h_{-\e}\in\partial\fe^\beta(\Uset)$ and
	\[h\cdot e_1-2\e\le h_{-\e}\cdot e_1<h\cdot e_1.\]
Furthermore, $\abs{h_{\pm\e}\cdot e_2-h\cdot e_2}=\abs{\fepl^\beta(h\pm\e e_1)}\le\e$ and part \eqref{item-b} implies that $h_\e\cdot e_2<h\cdot e_2$ and $h_{-\e}\cdot e_1>h\cdot e_1$.
Part \eqref{item-a} is thus proved.

For $\xi\in\ri\Uset$, Lemma \ref{lem:insub} along with the definition implies that $v\in\partial\fe^\beta(\xi)$ if and only if 
	\begin{align}\label{superdiff2}
		v \cdot \xi = \fe^{\beta}(\xi) \quad \text{ and }\quad v\cdot\zeta\ge\fe^\beta(\zeta)\quad\text{for all $\zeta\in\R_+^2$.}
	\end{align}
Consider the function $f(t)=\fe^\beta(\xi+t e_1)$, $t\ge -\xi\cdot e_1$. This is a concave function and its right-derivative $f'(0+)$ at $t=0$ exists and equals $\nabla\fe^\beta(\xi+)$. Concavity then implies that for all $t\ge-\xi\cdot e_1$
	\[f(t)-f(0)\le f'(0+)t.\]
This means that for $\zeta\in(\xi+\R e_1)\cap\R_+^2$ 
	\[\fe^\beta(\zeta)-\fe(\xi)\le(\zeta-\xi)\cdot\nabla\fe^\beta(\xi+).\]
Applying Lemma \ref{euler} this becomes
	\[\fe^\beta(\zeta)\le\zeta\cdot\nabla\fe^\beta(\xi_+)\quad\text{for }\zeta\in(\xi+\R e_1)\cap\R_+^2.\]
Homogeneity of $\fe^\beta$ extends this inequality to all of $\R_+^2$. This shows that $v=\nabla\fe^\beta(\xi+)$ satisfies
\eqref{superdiff2} and is hence in $\partial\fe^\beta(\xi)$. A similar proof works for $\nabla\fe^\beta(\xi-)$.

The definition of $\partial\fe^\beta(\xi)$ now implies that 
$[\nabla\fe^\beta(\xi-),\nabla\fe^\beta(\xi+)]\subset\partial\fe^\beta(\xi)$. On the other hand, if $v\in\partial\fe^\beta(\xi)$, then
for any $\e>0$ we have
	\[-\e v\cdot e_1\le\fe^\beta(\xi)-\fe^\beta(\xi+\e e_1).\]
This implies that $\nabla\fe^\beta(\xi+)\cdot e_1\le v\cdot e_1$. Similar inequalities work for $e_2$ in place of $e_1$ and also for $\nabla\fe^\beta(\xi-)$ and give us that $v\in[\nabla\fe^\beta(\xi+),\nabla\fe^\beta(\xi-)]$. The first claim in \eqref{item-c} is proved.  By concavity, existence of two-sided $e_1$ and $e_2$ directional derivatives at a point is equivalent to differentiability at that point; see \cite[Theorem 25.2]{Roc-70}. For $\alpha>0$, by homogeneity $\partial \fe^{\beta}(\alpha\xi) = \partial \fe^{\beta}(\xi)$; the second claim in \eqref{item-c} follows from the fact that (by concavity) $f$ defined above and the corresponding function for $e_2$ have  countably many points $t$ where the left and right derivatives disagree.
\end{proof}

\subsection{Existence and regularity of shift-covariant correctors}\label{sec:cocycles}

%First, let us describe the extended space on which we will construct the correctors. 
Fix a probability space $(\Omega, \sF,\bbP)$ as in Section \ref{sub:notation}.  Let $\sB_0$ be the union of $\{\infty\}$ and a dense countable subset of $(0,\infty)$. For $\beta\in(0,\infty]$ 
recall the limiting free energy $\fe^\beta$ from Section \ref{sec:fe} and
let $\sH^\beta=-\partial\fe^\beta(\Uset)$. 
Let $\sH^\beta_0$ be  a countable dense subset of $\sH^\beta$.
%with the property that if for some $\xi\in\ri\Uset$, $\partial \fe^\beta(\xi)$ does not consist of a single point, then (minus) both extreme points of $\partial \fe^\beta(\xi)$ are included in $\sH^\beta_0$. This is possible because there are at most countably many such points. 
Let $\sB_0 \times \sH^\bbullet_0 = \{(\beta,h) : \beta \in \sB_0 , h \in \sH^{\beta}_0\}$ and define $\sB_0\times\sH^\bbullet$ similarly.
Let $\widehat{\Omega} = \Omega \times \bbR^{\bbZ^2 \times \{1,2\}\times (\sB_0\times\sH^\bbullet_0) }$ be equipped with the product topology and product Borel $\sigma$-algebra, $\widehat{\sG}$. This space satisfies the conditions in Section \ref{sub:notation} if $\Omega$ does. 
Let $\That=\{\That_z:z\in\Z^2\}$  be the $\wh \sG$-measurable group of transformations 
that map $(\w, \{t_{x,i,\beta,h}:(x,i,\beta,h)\in\Z^2\times\{1,2\}\times(\cB_0\times\sH^\bbullet_0)\})$ to 
$(T_z\w, \{t_{x+z,i,\beta,h}:(x,i,\beta,h)\in\Z^2\times\{1,2\}\times(\cB_0\times\sH^\bbullet_0)\} )$.
Denote by $\pi_{\Omega}$ the projection map to the $\Omega$ coordinate. We will write $\w$ for $\pi_\Omega(\what)$ and 
the usual $\w_x$ for $\w_x(\w)$.

The next theorem furnishes the covariant, recovering cocycles used in \cite{Geo-Ras-Sep-17-ptrf-1,Geo-Ras-Sep-17-ptrf-2} without the condition $\P(\w_0\ge c)=1$ which was inherited from queueing theory; see \cite[(2.1)]{Geo-Ras-Sep-17-ptrf-1}. In \cite{Geo-Ras-Sep-17-ptrf-1} the authors also prove ergodicity of these cocycles.
As one can see from the proofs in this paper, ergodicity can be replaced by stationarity without losing the conclusions of \cite{Geo-Ras-Sep-17-ptrf-1}. We do not need ergodicity in the present project and so do not prove it here. These questions are addressed in our companion paper \cite{Jan-Ras-18-ecp-}.
%under the assumptions in this paper. 

Our construction of cocycles follows ideas from \cite{Dam-Han-14}. 
However, there is a novel technical difficulty stemming from the directedness of the paths, boiling down to a lack of uniform integrability of pre-limit Busemann functions. Essentially the same issue is resolved in the zero temperature queueing literature by an argument which relies on Prabhakar's \cite{Pra-03} rather involved result showing that ergodic fixed points of the corresponding 
$\acdot/G/1/\infty$ queue are attractive. Instead, we handle this problem by appealing to the variational formulas for the 
free energy derived in \cite{Geo-Ras-Sep-16}.

For a subset $I\subset\Z^2$ let	$I^{<}=\{x\in\Z^2:x\not\ge z\ \forall z\in I\}$.    

\begin{theorem}\label{thm:cocyexist}
There are a $\That$-invariant probability measure $\Phat$ on $(\wh \Omega, \wh \sG)$ and real-valued
measurable functions $B^{\beta,h+}(x,y,\what)$ and $B^{\beta,h-}(x,y,\what)$ of $(\beta,h,x,y,\what)\in(\sB_0\times\sH^\bbullet)\times\Z^2\times\Z^2\times\Omhat$
 such that:
\begin{enumerate}[label={\rm(\alph*)}, ref={\rm\alph*}]
\item\label{thm:cocyexist:a} For any event $A \in \sF$,
$\Phat(\pi_{\Omega}(\wh \w)\in A) = \bbP(A)$.
\item\label{thm:cocyexist:b} For any $I \subset \bbZ^2,$ the variables
\begin{align*}
\{(\w_x, B^{\beta,h\pm}(\wh\w,x,y)): x \in I,y\ge x,\beta\in\sB_0,h\in \sH^\beta\}	
\end{align*}
are independent of $\{\w_x : x \in I^{<}\}$.
\item\label{thm:cocyexist:c} For each $\beta\in\sB_0$, $h\in\sH^\beta$, and  $x,y\in\Z^2$,
$B^{\beta,h\pm}(x,y)$ are  integrable and 
\begin{align}\label{E[h(B)]}
\Ehat\bigl[ B^{\beta,h\pm}(x,x+e_i)\bigr] &= - h \cdot e_i.
\end{align}
\item\label{thm:cocyexist:d} There exists a $\That$-invariant event  $\Ommonohat$ with $\Phat(\Ommonohat)=1$ such that for each $\what\in\Ommonohat$,  $x,y,z\in\Z^2$, $\beta\in\sB_0$,   
$h \in \sH^\beta$, and $\epsilon\in\{-,+\}$ 
	\begin{align}
	&B^{\beta,h\epsilon}(x+z,y+z,\what)=B^{\beta,h\epsilon}(x,y,\That_z\what),\label{cov-prop}\\
	&B^{\beta,h\epsilon}(x,y,\what)+B^{\beta,h\epsilon}(y,z,\what)=B^{\beta,h\epsilon}(x,z,\what),
	\quad\text{and}\label{coc-prop}\\
	&e^{-\beta B^{\beta,h\epsilon}(x,x+e_1,\what)}+e^{-\beta B^{\beta,h\epsilon}(x,x+e_2,\what)}=e^{-\beta\w_x},\text{ if $\beta<\infty$,}\label{rec-prop1}\\
	&\min\bigl\{B^{\beta,h\epsilon}(x,x+e_1,\what),B^{\beta,h\epsilon}(x,x+e_2,\what)\bigr\}=\w_x,\text{ if $\beta=\infty$.}\label{rec-prop2}
	\end{align}
\item\label{thm:cocyexist:e} For each $\what\in\Ommonohat$, $x\in\Z^2$, $\beta\in\sB_0$, and $h,h' \in \sH^\beta$ with $h \cdot e_1 \leq h' \cdot e_1$, 
\begin{align}\label{mono}
\begin{split}
&B^{\beta,h-}(x,x+e_1,\what) \ge B^{\beta,h+}(x,x+e_1,\what)\\
&\qquad\ge B^{\beta,h'-}(x,x+e_1,\what)\ge B^{\beta,h'+}(x,x+e_1,\what)\quad\text{and} \\
&B^{\beta,h-}(x,x+e_2,\what) \le B^{\beta,h+}(x,x+e_2,\what)\\
&\qquad\le B^{\beta,h'-}(x,x+e_2,\what)\le B^{\beta,h'+}(x,x+e_2,\what).
\end{split}
\end{align} 
\item\label{thm:cocyexist:f} For each $\what\in\Ommonohat$, $\beta\in\sB_0$, $h\in\sH^\beta$, and $x,y\in\Z^2$, 
	\begin{align}\label{Busemann-limits}
	\begin{split}
	&B^{\beta,h-}(x,y,\what)=\lim_{\substack{\sH^\beta\ni h' \to h\\ h' \cdot e_1 \nearrow h \cdot e_1}}B^{\beta,h'\pm}(x,y,\what)\quad\text{and}\\	
	&B^{\beta,h+}(x,y,\what)=\lim_{\substack{\sH^\beta\ni h' \to h\\ h' \cdot e_1 \searrow h \cdot e_1}}B^{\beta,h'\pm}(x,y,\what).
	\end{split}
	\end{align}
When $B^{\beta,h+}(x,y,\what)=B^{\beta,h-}(x,y,\what)$ we drop the $+/-$ and 
write $B^{\beta,h}(x,y,\what)$ and then for any $\epsilon\in\{-,+\}$
	\begin{align}\label{B:cont}
	\lim_{\sH^\beta\ni h' \to h} B^{\beta,h'\epsilon}(x,y,\what) = B^{\beta,h}(x,y,\what).
	\end{align}
\item\label{thm:cocyexist:g} For each $\beta\in\sB_0$ and $h \in \sH^\beta$ 
there exists an event \label{Omconthat}$\Omconthat{\beta,h}\subset\Ommonohat$ with $\Phat(\Omconthat{\beta,h})=1$ and 
%{\rm(}and hence \eqref{B:cont} holds{\rm)} 
for all $\what\in\Omconthat{\beta,h}$ and all $x,y \in \bbZ^2$
	\[B^{\beta,h+}(x,y,\what)=B^{\beta,h-}(x,y,\what)=B^{\beta,h}(x,y,\what).\] 
\end{enumerate}	
\end{theorem}%\med

\begin{remark}
The proofs of parts \eqref{thm:cocyexist:a} and (\ref{thm:cocyexist:c}-\ref{thm:cocyexist:f}) work word-for-word if 
the distribution of $\{\w_x(\w):x\in\Z^2\}$ induced by $\P$ is $T$-ergodic and $\w_0(\w)$ belongs to class $\mathcal L$, defined in \cite[Definition 2.1]{Geo-Ras-Sep-16}. 
\end{remark}

\begin{remark}
In the rest of the paper we will construct various full-measure events. By shift-invariance of $\P$ and $\Phat$, replacing any such event with the intersection of all its shifts we can assume these full-measure events to also be shift-invariant. This will be implicit in the proofs that follow.
\end{remark}

\begin{proof}[Proof of Theorem \ref{thm:cocyexist}]
For $\beta\in(0,\infty]$, $h\in\R^2$, $n\in\N$, $x\in\Z^2$, and $i\in\{1,2\}$  define 
	\[B_n^{\beta,h}(x,x+e_i)=F^{\beta,h}_{x,(n)}-F^{\beta,h}_{x+e_i,(n)}-h\cdot e_i\]
if $x\cdot\ehat<n$ and $B_n^{\beta,h}(x,x+e_i)= 0$ otherwise. A direct computation shows that if $x \cdot \ehat<n$ then
\begin{align}
&B_n^{\beta,h}(x,x+e_i)	= B_{n-x \cdot\ehat}^{\beta,h}(0,e_i) \circ T_x,\label{aux-cov}\\
&e^{-\beta \w_x} = 	e^{-\beta B_n^{\beta,h}(x,x+e_1)} + e^{-\beta B_n^{\beta,h}(x,x+e_2)},\quad\text{if }\beta<\infty,\text{ and}\label{aux-rec<infty}\\
&\w_x =\min\bigl(B_n^{\beta,h}(x,x+e_1), B_n^{\beta,h}(x,x+e_2)\bigr),\quad\text{if }\beta=\infty.\label{aux-rec=infty}
\end{align}
Moreover, if $n>x \cdot \ehat+1$, then
\begin{align}\label{aux-coc}
\begin{split}
&B_n^{\beta,h}(x,x+e_1) + B_n^{\beta,h}(x+e_1,x+e_1+e_2) \\
&\qquad = B_n^{\beta,h}(x,x+e_2) + B_n^{\beta,h}(x+e_2,x+e_1+e_2).
\end{split}
\end{align}

We also prove the following in Appendix \ref{sub:coupling}.

\begin{lemma}\label{lem:mono}
Suppose $n > x \cdot \ehat$ and that $h \cdot e_1 \le h' \cdot e_1$ and $h \cdot e_2 \ge h' \cdot e_2$. Then for each $\beta \in (0,\infty]$, each $n$, each $x \in \bbZ^2$, $\bbP$-almost surely
\begin{align}\label{eq:mono}
\begin{split}
&B_n^{\beta,h}(x,x+e_1) \geq B_n^{\beta,h'}(x,x+e_1)\quad\text{and}\\ 
&B_n^{\beta,h}(x,x+e_2) \leq B_n^{\beta,h'}(x,x+e_2).
\end{split}
\end{align}
\end{lemma}

Next, we employ an averaging procedure previously used by \cite{Dam-Han-14,Hof-08,Gar-Mar-05,Lig-85}, among others.
For each $n \in \bbN$, let $N_n$ be uniformly distributed on $\{1,\dots,n\}$ and independent of everything else. 
Let $\bfP_n$ be its distribution and abbreviate $\widetilde\P_n=\P\otimes\bfP_n$ with expectation $\wt\E_n$.
Define
\begin{align*}
\wh B_n^{\beta,h}(x,x+e_i) &= B_{N_n}^{\beta,h}(x,x+e_i).
\end{align*}
Then whenever $n > x \cdot \ehat$, 
\begin{align}
%&\bbE[\wh B_n^{\beta,h}(x,x+e_i)] =\frac{1}{n}\sum_{j=x\cdot(e_1+e_2)}^n \bbE[B_j^{\beta,h}(x,x+e_i)]-\frac{x\cdot(e_1+e_2)-1}n\, h\cdot e_i\\
\widetilde{\E}_n\bigl[\wh B_n^{\beta,h}(x,x+e_i)\bigr] 
%=\frac{1}{n}\sum_{j=x\cdot\ehat+1}^n \bbE\bigl[F^{\beta,h}_{x,(j)}-F^{\beta,h}_{x+e_i,(j)}\bigr]-h\cdot e_i\notag\\
%&\qquad
&= \frac{1}{n}\sum_{j=x\cdot\ehat+1}^n \bigl(\bbE\bigl[F_{0,(j-x\cdot\ehat)}^{\beta,h} - F_{0,(j-x\cdot\ehat-1)}^{\beta,h}\bigr] - h\cdot e_i\bigr)\notag\\
&= \frac{1}{n}\bbE\bigl[F_{0,(n-x\cdot\ehat)}^{\beta,h}\bigr] -  \left(\frac{n - x \cdot \ehat}{n}\right) h\cdot e_i.\label{for fatou}
\end{align}
By (\ref{aux-rec<infty}-\ref{aux-rec=infty}) we have $\wh B_n^{\beta,h}(x,x+e_i) \geq \w_x$ on the event 
$\{N_n > x \cdot \ehat\}$. On the complementary event we have $\wh B_n^{\beta,h}(x,x+e_i)=0$.
Whenever $n> x \cdot\ehat$,
\begin{align*}
&\widetilde{\E}_n\bigl[\abs{\wh B_n^{\beta,h}(x,x+e_i)}\bigr]
=\widetilde{\E}_n\bigl[\wh B_n^{\beta,h}(x,x+e_i)\bigr]-2\widetilde{\E}_n\bigl[\min\bigl(0,\wh B_n^{\beta,h}(x,x+e_i)\bigr)\bigr] \\
&\qquad\leq \frac{1}{n}\bbE\bigl[F_{0,(n-x\cdot\ehat)}^{\beta,h}\bigr]  - \left(\frac{n - x \cdot \ehat}{n}\right) h\cdot e_i + 2 \bbE[\abs{\w_0}].
\end{align*}
The first term converges to $\fepl^{\beta}(h)$, which equals zero if $h\in\sH^\beta$ by Lemma \ref{lem:insub}. Then the right-hand side is bounded by a finite constant $c(x,\beta,h)$. 
If we denote by $\bbP_n$ the law of
	\[\Bigl(\w,\bigl\{\wh B_n^{\beta,h}(x,x+e_i):  x\in \bbZ^2, i \in \{1,2\}, \beta\in\sB_0, h \in \sH^\beta_0\bigr\}\Bigr)\] 
induced by $\widetilde{\P}_n$ on $(\Omhat,\wh\sG)$, then the family $\{\bbP_n : n \in \bbN\}$ is tight.    Let $\Phat$ denote any weak subsequential limit point of this family of measures. 
$\Phat$ is then $\That$-invariant because of \eqref{aux-cov} and the $T$-invariance of $\P$. 
We prove next that such a measure satisfies all of the conclusions of the theorem. 

Let $B^{\beta,h}(x,x+e_i,\what)$ be the $(x,i,\beta,h)$-coordinate of $\what\in\Omhat$. Since inequalities \eqref{eq:mono} hold for every $n$ 
there exists an event $\Omtemphat'$ (which can be assumed to be $\That$-invariant) with $\Phat(\Omtemphat')=1$ such that 
for any $\beta\in\sB_0$, $h,h'\in\sH^\beta_0$ with $h \cdot e_1 \le h' \cdot e_1$, $x \in \bbZ^2$,
and $\what\in\Omtemphat'$,
\begin{align}\label{eq:mono2}
\begin{split}
&B^{\beta,h}(x,x+e_1,\what) \ge B^{\beta,h'}(x,x+e_1,\what)\quad\text{and}\\ 
&B^{\beta,h}(x,x+e_2,\what) \le B^{\beta,h'}x,x+e_2,\what).
\end{split}
\end{align}
Due to this monotonicity we can define
	\begin{align*}
	&B^{\beta,h-}(x,x+e_i,\what)=\lim_{h'\in\sH_0^\beta, h'\cdot e_1\nearrow h\cdot e_1}B^{\beta,h}(x,x+e_i)\quad\text{and}\\
	&B^{\beta,h+}(x,x+e_i,\what)=\lim_{h'\in\sH_0^\beta, h'\cdot e_1\searrow h\cdot e_1}B^{\beta,h}(x,x+e_i).
	\end{align*}
Then parts \eqref{thm:cocyexist:e} and \eqref{thm:cocyexist:f} come immediately.
	
Since \eqref{aux-coc} holds for every $n$ we get the existence of a $\That$-invariant event $\Omtemphat''\subset\Omtemphat'$ 
with $\Phat(\Omtemphat'')=1$ and
\begin{align}\label{aux-coc2}
\begin{split}
&B^{\beta,h}(x,x+e_1,\what) + B^{\beta,h}(x+e_1,x+e_1+e_2,\what) \\
&\qquad = B^{\beta,h}(x,x+e_2,\what) + B^{\beta,h}(x+e_2,x+e_1+e_2,\what)
\end{split}
\end{align}
for all $x\in\Z^2$, $\beta\in\sB_0$, $h\in\sH^\beta_0$, and $\what\in\Omtemphat''$.
This equality transfers to $B^{\beta,h\pm}$. 
Set $B^{\beta,h\pm}(x+e_i,x,\what)=-B^{\beta,h\pm}(x,x+e_i,\what)$ and for $x,y\in\Z^2$ and $\what\in\Omtemphat''$ 
	\[B^{\beta,h\pm}(x,y,\what)=\sum_{k=0}^{m-1} B^{\beta,h\pm}(x_k,x_{k+1},\what),\]
where $x_{0,m}$ is any path from $x$ to $y$ with $\abs{x_{k+1}-x_k}_1=1$. The sum does not depend on the path we choose,
due to \eqref{aux-coc2}. Property \eqref{coc-prop} follows.

For each $n$ and each $A \in \sF$, $\bbP_n(\pi_{\Omega}(\wh \w) \in A) = \bbP(\w \in A)$. Moreover, for each $n$ and each $I \subset \bbZ^2$, the family $\bigl\{\w_x, \wh B_n^{\beta,h}(x,x+e_i) : x \in I, \beta\in\sB_0,h \in \sH^\beta_0, i \in \{1,2\}\bigr\}$ is independent of 
$\{\w_x:x \in I^{<}\}$. These properties  transfer to $\Phat$ and parts \eqref{thm:cocyexist:a} and \eqref{thm:cocyexist:b} follow. 
%\P_n(O) converges to \Phat(O) for open sets and so \Phat=\P on open sets and that's enough

Again, since (\ref{aux-cov}-\ref{aux-rec=infty}) hold for each $n$, there exists a $\That$-invariant full $\Phat$-measure event 
\label{Ommonohat}$\Ommonohat\subset\Omtemphat''$ on which \eqref{cov-prop} and  (\ref{rec-prop1}-\ref{rec-prop2}) hold.
 \eqref{thm:cocyexist:d} is proved.

Recall \eqref{for fatou} and that the right-hand side converges to $\fe^\beta(h)-h\cdot e_i=-h\cdot e_i$.
We have also seen that
 	\begin{align*}
	\wh B_n^{\beta,h}(x,x+e_i) 
	\ge \w_x\one\{N_n>x\cdot\ehat\}.
	\end{align*}
Fatou's lemma then implies that $B^{\beta,h}(x,x+e_i)$ is integrable under $\Phat$ and
\begin{align}
\Ehat\bigl[B^{\beta,h}(x,x+e_i)\bigr] \leq - h \cdot e_i\quad\text{for }\beta\in\sB_0,h\in\sH^\beta_0.\label{eq:fatou}
\end{align}

The reverse inequality is the nontrivial step in this construction. We spell out the argument in the case $\beta < \infty$, with the $\beta = \infty$ case being similar. 

Let $\hbar=-\Ehat[B^{\beta,h}(x,x+e_1)]e_1-\Ehat[B^{\beta,h}(x,x+e_2)]e_2$,
$\mfS = \sigma(\w_x : x \in \bbZ^2)$, and define
$\wt{B}^{\beta,h}(x,x+e_i) = \Ehat[ B^{\beta,h}(x,x+e_i)\,|\mfS]$. Then $\wt B^{\beta,h}$ satisfies an equation like \eqref{aux-coc2} which we can use to define a cocycle $\wt B^{\beta,h}(x,y)$, $x,y\in\Z^2$.
Note that in general this cocycle will not recover the potential, even if $B^{\beta,h}$ does; it does however have the same mean vector $\hbar$ as $B^{\beta,h}$. By Jensen's inequality and recovery,
\begin{align}\label{for Timo}
\begin{split}
e^{-\beta \wt{B}^{\beta,h}(0,e_1)} + e^{-\beta \wt{B}^{\beta,h}(0,e_2)}
&\le\Ehat\bigl[e^{-\beta B^{\beta,h}(0,e_1)} + e^{-\beta B^{\beta,h}(0,e_2)}\,|\mfS\bigr]\\
&= \Ehat[ e^{-\beta \w_0}\, | \mfS] = e^{-\beta \w_0}.
\end{split}
\end{align}
For $h\in\sH^\beta$ let $\xi \in \ri \Uset$ be such that $-h \in \partial \fe^{\beta}(\xi)$. Having conditioned on $\mfS$, we are back in the canonical setting where 
$\wt B^{\beta,h}$ can be viewed as defined on the product space $\R^{\Z^2}$ with its Borel $\sigma$-algebra and an i.i.d.\ probability measure $\P_0^{\otimes\Z^2}$, where $\P_0$ is the distribution of $\w_0$ under $\P$.
This setting is ergodic. Apply the duality of $\xi$ and $h$, the variational formula of \cite[Theorem 4.4]{Geo-Ras-Sep-16}, and \eqref{for Timo} to obtain
\begin{align*}
-h \cdot \xi = \fe^{\beta}(\xi) \leq \P\text{-}\esssup_\w \frac{1}{\beta} \log \sum_{i=1,2} e^{\beta \w_0 - \beta \tilde{B}^{\beta,h}(0,e_i) - \beta \hbar \cdot \xi}
\le -\hbar \cdot \xi.
\end{align*}
This, inequality \eqref{eq:fatou}, and the fact that $\xi$ has positive coordinates imply $\hbar \cdot e_i = h \cdot e_i$ for $i =1,2$. In other words,
\begin{align*}
\Ehat\bigl[ B^{\beta,h}(x,x+e_i)\bigr] &= - h \cdot e_i\quad\text{for }\beta\in\sB_0,h\in\sH^\beta_0.
\end{align*}
Part \eqref{thm:cocyexist:c} follows from this, monotonicity \eqref{eq:mono2}, and the monotone convergence theorem.
Then part \eqref{thm:cocyexist:g} follows from monotonicity \eqref{mono} and the fact that for $i\in\{1,2\}$, $B^{\beta,h\pm}(x,x+e_i)$ have the same mean $\hbar$.
\end{proof}

It will be convenient to also define the process indexed by $\xi\in\ri\Uset$:
	\[B^{\beta,\xi\pm}(x,y)=B^{\beta,-\nabla\fe^\beta(\xi\pm)\pm}(x,y).\]

\begin{remark}\label{B-xi}
Parts (\ref{thm:cocyexist:b}-\ref{thm:cocyexist:f}) of Theorem \ref{thm:cocyexist} transfer to this process in the obvious way.
For example the first set of inequalities in \eqref{mono} becomes
\begin{align*}%\label{mono}
\begin{split}
&B^{\beta,\xi-}(x,x+e_1) \ge B^{\beta,\xi+}(x,x+e_1)
\ge B^{\beta,\zeta-}(x,x+e_1)\ge B^{\beta,\zeta+}(x,x+e_1)%\quad\text{and} \\
%&B^{\beta,\xi-}(x,x+e_2) \le B^{\beta,\xi+}(x,x+e_2)\le B^{\beta,\zeta-}(x,x+e_2)\le B^{\beta,\zeta+}(x,x+e_2)
\end{split}
\end{align*} 
for $\xi,\zeta\in\ri\Uset$ with $\xi\cdot e_1\le\zeta\cdot e_1$. 
\eqref{thm:cocyexist:g} becomes the following: for each $\beta\in\sB_0$ and $\xi\in\Diff^\beta$ there exists an event 
$\Omconthat{\beta,\xi}=\Omconthat{\beta,-\nabla\fe^\beta(\xi)}$ with $B^{\beta,\xi+}(x,y,\what)=B^{\beta,\xi-}(x,y,\what)=B^{\beta,\xi}(x,y,\what)$ for all $\what\in\Omconthat{\beta,\xi}$ and all $x,y\in\Z^2$.
\end{remark}

We will need two lemmas in what follows. 

\begin{lemma}\label{lm:tilt}
For each $\xi\in\ri\Uset$, there exists  an event \label{Omconsthat}$\Omconsthat{\xi+}$ such that $\Phat(\Omconsthat{\xi+})=1$ and $h(B^{\beta,\xi+})=-\nabla\fe^\beta(\xi+)$ on $\Omconsthat{\xi+}$ for all $\beta\in\sB_0$. A similar statement holds for $\xi-$.
\end{lemma}

\begin{proof}
By \eqref{E[h(B)]} we have $-\Ehat[h(B^{\beta,\xi\pm})]=\nabla\fe^\beta(\xi\pm)$. The claim then follows from Lemma \ref{lm:h-xi}\eqref{lm:h-xi:c}. 
\end{proof}

\begin{lemma}\label{lm:B-ei}
There exists a $\That$-invariant event \label{Omeihat}$\Omeihat$ so that for $\what\in \Omeihat$, $x \in \Z^2$, $\beta\in \sB_0\backslash\{\infty\}$, and $i \in \{1,2\}$,
\begin{align*}
\lim_{\ri\Uset\ni\xi \to e_i} B^{\beta,\xi \pm,\what}(x,x+e_i) = \w_x \quad\text{and}\quad\lim_{\ri\Uset\ni\xi \to e_i} B^{\beta,\xi \pm}(x,x+e_{3-i},\what) = \infty.
\end{align*}
\end{lemma}

\begin{proof}
Suppose $\beta < \infty$ and take $\what\in\Ommono$. Then the claimed limits exist due to the above monotonicity. The second limit follows from the first by 
recovery (\ref{rec-prop1}-\ref{rec-prop2}). Recovery also implies that $B^{\beta,\xi\pm}(x,x+e_i,\what)-\w_x\ge0$. 
 But then
 	\begin{align*}
	0&\le\Ehat\Bigl[\lim_{\xi\to e_i} B^{\beta,\xi\pm}(x,x+e_i)-\w_x\Bigr]
	=\Ehat\Bigl[\inf_{\xi\in\ri\Uset} B^{\beta,\xi\pm}(x,x+e_i)-\w_x\Bigr]\\
	&\le\inf_{\xi\in\ri\Uset}\Ehat[B^{\beta,\xi\pm}(x,x+e_i)]-\E[\w_x]
	=\inf_{\xi\in\ri\Uset}\nabla\fe^\beta(\xi\pm)\cdot e_i-\E[\w_0]
	=0,
	\end{align*}
where the last equality follows from Lemma \ref{lem:polymermartin}.
\end{proof}

\begin{remark}
Lemma \ref{lm:B-ei} also holds when $\beta = \infty$. The proof in \cite[Lemma 5.1]{Geo-Ras-Sep-17-ptrf-2} does not use any of the additional hypotheses in that paper.
\end{remark}

As mentioned earlier, in the rest of the paper 
we assume $\beta=1$ and omit it from the notation. In particular, we write $\fe$ and $\sH$ instead of $\fe^1$ and $\sH^1$.

\subsection{Ratios of partition functions}
%Following similar steps to the proofs of (4.3) in \cite{Geo-etal-15} and Theorem 6.1 in \cite{Geo-Ras-Sep-17-ptrf-1} we obtain the following.
Following similar steps to the proofs of (4.3) of \cite{Geo-etal-15} and Theorem 6.1 in \cite{Geo-Ras-Sep-17-ptrf-1} we obtain the next theorem. 
Our more natural definition of the $B^{\xi\pm}$ processes makes the claim hold on one full-measure event, in contrast
with \cite{Geo-etal-15,Geo-Ras-Sep-17-ptrf-1} where the events depend on $\xi$.

\begin{theorem}\label{thm:Buselim1}
%Fix $\xi\in\ri\Uset$ such that $\Uset_{\xi-}=\Uset_{\xi+}=[\ximin,\ximax]$. 
%There exists a shift-invariant  event $\OmBushat$ such that $\Phat(\OmBushat)=1$ and for all $\what\in\OmBushat$, any  $x\in\Z^2$, and any {\rm(}possibly $\what$-dependent{\rm)} $\Uset_\xi$-directed sequence $x_n\in\V_n$:
There exists a shift-invariant event $\OmBusallhat$ such that $\Phat(\OmBusallhat)=1$ and  
for all $\what\in\OmBusallhat$, any {\rm(}possibly $\what$-dependent{\rm)} $\xi\in\ri\Uset$, $x\in\Z^2$, and $\Uset_\xi$-directed sequence $x_n\in\V_n$:
	\begin{align}\label{fixed-xi-lim}
	\begin{split}
	&e^{-B^{\ximin-}(x,x+e_1,\what)}\le \varliminf_{n\to\infty}\frac{Z_{x+e_1,x_n}}{Z_{x,x_n}} \le \varlimsup_{n\to\infty}\frac{Z_{x+e_1,x_n}}{Z_{x,x_n}} \le e^{-B^{\ximax+}(x,x+e_1,\what)}\ \text{and}\\
	&e^{-B^{\ximax+}(x,x+e_2,\what)}\le \varliminf_{n\to\infty}\frac{Z_{x+e_2,x_n}}{Z_{x,x_n}} \le \varlimsup_{n\to\infty}\frac{Z_{x+e_2,x_n}}{Z_{x,x_n}} \le e^{-B^{\ximin-}(x,x+e_2,\what)}.
	\end{split}
	\end{align}
\end{theorem}

\begin{proof}
Let $\Ddense$ be a countable dense subset of $\Diff$. Let $\xi \in \ri \Uset$ and
	\[\what\in\OmBusallhat=\Ommonohat\cap\bigcap_{\zeta\in\Ddense}\Omconthat{\zeta}.\label{OmBusallhat}\]
First, consider $\bar x_n=\bar x_n(\xi)$ that is the (leftmost) closest point in $\V_n$ to $n\xi$. 
Then $\bar x_n/n\to\xi$ as $n\to\infty$.
Let $\zeta\in\Ddense$ be such that $\zeta\cdot e_1>\ximax\cdot e_1$. Since $\what\in\Omconthat{\zeta}$ we have $B^{\zeta\pm}=B^\zeta$.
For $x\in\V_k$, $y\in\V_\ell$, $k,\ell\in\Z$, and $x\le y$, define the point-to-point partition function
	\begin{align}\label{Zne}
	\Zne_{x,y}=\sum_{x_{k,\ell}\in\bbX_x^y} e^{\sum_{i=k}^{\ell-1}\wbar_{x_i}},
	\end{align}
where $\wbar_u=B^{\zeta}(u,u+e_i,\what)$ if $y-u\in\N e_i$, $i\in\{1,2\}$ and $\wbar_u= \w_u$ otherwise. 
%Recovery and the cocycle property of $B^{\zeta}$ imply that  $Z_{x,y}^{\beta,NE} =\text{exp}\{B^{\zeta}(x,y)\}$ for all $x \leq y$ because they satisfy the same recursion with the same boundary conditions on $y-x\in\N e_i$, $i\in\{1,2\}$.
%
For $x \leq y - e_1 - e_2$, use the cocycle property to rewrite the recovery property as
\begin{align*}
e^{- \omega_x}&=  e^{- B^{\zeta}(x,x+e_1)} + e^{- B^{\zeta}(x,x+e_2)} = e^{-\bigl(B^{\zeta}(x,y) - B^{\zeta}(x+e_1,y)\bigr)} + e^{-\bigl(B^{\zeta}(x,y) - B^{\zeta}(x+e_2,y)\bigr)}.
\end{align*}
This implies the recursion
\begin{align*}
e^{ B^{\zeta}(x,y)} &= e^{\omega_x} \Bigl[e^{ B^{\zeta}(x+e_1,y)} + e^{ B^{\zeta}(x+e_2,y)}\Bigr].
\end{align*}
$\Zne_{x,y}$, $x \leq y-e_1-e_2$, solves the same recursion with the same boundary conditions on  $y-x\in\N e_i$, $i\in\{1,2\}$ and therefore $\Zne_{x,y}=e^{B^{\zeta}(x,y)}$ for all $x \leq y$. 
Then
	\[\frac{\Zne_{x,\bar x_n+e_1+e_2}}{\Zne_{x+e_1,\bar x_n+e_1+e_2}}=e^{B^{\zeta}(x,x+e_1,\what)}.\]

For $v$ with $x\le v\le y$ let $\Zne_{x,y}(v)$ be defined as in \eqref{Zne} but with the sum being only over admissible paths that go through $v$.
Apply the first inequality in \eqref{comparison} with $\wtil$  such that 
$\w_y(\wtil)=\w_y$ for $y\le \bar x_n$,  $\w_y(\wtil)=B^{\zeta}(y,y+e_i,\what)$ 
for $y$ with $\bar x_n+e_1+e_2-y\in\N e_i$, $i\in\{1,2\}$, $v=\bar x_n$, and $u=\bar x_n+e_1$ to get 
	\begin{align}
	\frac{Z_{x,\bar x_n}}{Z_{x+e_1,\bar x_n}}
	&\ge
	\frac{\Zne_{x,\bar x_n+e_1+e_2}(\bar x_n+e_1)}{\Zne_{x+e_1,\bar x_n+e_1+e_2}(\bar x_n+e_1)}\notag\\
	&\ge
	\frac{\Zne_{x,\bar x_n+e_1+e_2}(\bar x_n+e_1)}{\Zne_{x,\bar x_n+e_1+e_2}}\cdot
	\frac{\Zne_{x,\bar x_n+e_1+e_2}}{\Zne_{x+e_1,\bar x_n+e_1+e_2}}\notag\\
	&=
	\frac{\Zne_{x,\bar x_n+e_1+e_2}(\bar x_n+e_1)}{\Zne_{x,\bar x_n+e_1+e_2}}\cdot
	e^{B^{\zeta}(x,x+e_1)}.\label{frac-aux}	
%     &=	
%	\Bigl(1-\frac{\Zne_{x,\bar x_n+e_1+e_2}(\bar x_n+e_2)}{\Zne_{x,\bar x_n+e_1+e_2}}\Bigr)\,
%	e^{B^{\hu(\zeta)}(x,x+e_1)}.	
	\end{align}

Using the shape theorems \eqref{shape} and \eqref{shape-B} and a standard argument, given for example in the proof 
of \cite[Lemma 6.4]{Geo-Ras-Sep-17-ptrf-1}, we have
	\begin{align*}
	\lim_{n\to\infty} n^{-1}\log\Zne_{x,\bar x_n+e_1+e_2}(\bar x_n+e_1)
	%&=\sup\{\fe(\xi_1,t)+(\xi_2-t)\nabla\fe(\zeta)\cdot e_2:0\le t\le\xi_2\}\\
	&=\sup\bigl\{\fe(\eta)+(\xi-\eta)\cdot\nabla\fe(\zeta):\eta\in[(\xi\cdot e_1)e_1,\xi]\bigr\}\quad\text{and}\\
	\lim_{n\to\infty} n^{-1}\log\Zne_{x,\bar x_n+e_1+e_2}(\bar x_n+e_2)
	%&=\sup\{\fe(s,\xi_2)+(\xi_1-s)\nabla\fe(\zeta)\cdot e_1:0\le s\le\xi_1\},
	&=\sup\bigl\{\fe(\eta)+(\xi-\eta)\cdot\nabla\fe(\zeta):\eta\in[(\xi\cdot e_2)e_2,\xi]\bigr\}.
	\end{align*}
%where we write $\xi=\xi_1 e_1+\xi_2 e_2$ and view $\fe$ as a function of the two coordinates.

By Lemma \ref{lem:superdiffprop}\eqref{item-b} $\nabla\fe(\zeta)\cdot e_1\le\nabla\fe(\xi)\cdot e_1$ and hence for $\eta\in[(\xi\cdot e_2)e_2,\xi]$, 
	\[(\xi-\eta)\cdot\nabla\fe(\zeta)\le(\xi-\eta)\cdot\nabla\fe(\xi)\le\fe(\xi)-\fe(\eta).\]
Thus, the second supremum in the above is achieved at $\eta=\xi$ and the limit is equal to $\fe(\xi)$. Set $\eta_0=(\xi\cdot e_1/\zeta\cdot e_1)\zeta\in[(\xi\cdot e_1)e_1,\xi]$. For $\eta\in[(\xi\cdot e_1)e_1,\xi]$
	\begin{align*}
	(\eta_0-\eta)\cdot\nabla\fe(\zeta)
	%&=(\eta_0\cdot e_2-\eta\cdot e_2)\nabla\fe(\zeta)\cdot e_2\\
	&=\frac{\xi\cdot e_1}{\zeta\cdot e_1}\Bigl(\zeta-\frac{\zeta\cdot e_1}{\xi\cdot e_1}\eta\Bigr)\cdot\nabla\fe(\zeta)\\
	&\le\frac{\xi\cdot e_1}{\zeta\cdot e_1}\Bigl(\fe(\zeta)-\frac{\zeta\cdot e_1}{\xi\cdot e_1}\fe(\eta)\Bigr)
		=\fe(\eta_0)-\fe(\eta).
	\end{align*}
Rearranging, we have $\fe(\eta)+(\xi-\eta)\cdot\nabla\fe(\zeta)\le\fe(\eta_0)+(\xi-\eta_0)\cdot\nabla\fe(\zeta).$ Hence, the first supremum is achieved at $\eta_0$. But if equality also held for $\eta=\xi$, then concavity of $\fe$ would imply that $\fe$ is linear on $[\eta_0,\xi]$ and hence on $[\zeta,\xi]$. 
%Pf: Recall that $\partial \fe(\eta) = \partial \fe(\alpha \eta)$ for $\alpha>0$. $\fe$ is linear on $[\eta_0,\xi]$ then there is a vector $-h$ which lies in the superdifferential of all of the vectors $\eta \in [\eta_0,\xi]$. But then for any such $\eta$, we have $-h \in \partial \fe(\eta/\eta \cdot \ehat)$ where $\eta/\eta\cdot \ehat \in \Uset$ and $\eta_0/\eta_0 \cdot \ehat = \zeta$. 
This cannot be the case since $\zeta\not\in\Uset_\xi$.  We therefore have
	\[\fe(\eta_0)+(\xi-\eta_0)\cdot\nabla\fe(\zeta)>\fe(\xi).\]
This implies that $\Zne_{x,\bar x_n+e_1+e_2}(\bar x_n+e_2)/\Zne_{x,\bar x_n+e_1+e_2}(\bar x_n+e_1)\to0$ 
as $n\to\infty$. Since $\Zne_{x,\bar x_n+e_1+e_2}=\Zne_{x,\bar x_n+e_1+e_2}(\bar x_n+e_1)+\Zne_{x,\bar x_n+e_1+e_2}(\bar x_n+e_2)$ we conclude that the fraction in \eqref{frac-aux} converges to $1$. Consequently,
	\[\varlimsup_{n\to\infty}\frac{Z_{x+e_1,\bar x_n}}{Z_{x,\bar x_n}}\le e^{-B^\zeta(x,x+e_1)}.\]
Taking $\zeta\to\ximax$ we get the right-most inequality in the first line of \eqref{fixed-xi-lim}. The other inequalities come similarly.

Next, we prove the full statement of the theorem, namely that  \eqref{fixed-xi-lim} holds for all sequences $x_n\in\V_n$, directed into $\Uset_\xi$.  To this end, take such a sequence and let $\eta_\ell,\zeta_\ell\in\ri\Uset$ be two sequences such that $\eta_\ell\cdot e_1<\ximin\cdot e_1\le\ximax\cdot e_1<\zeta_\ell\cdot e_1$, $\eta_\ell\to\ximin$, and $\zeta_\ell\to\ximax$. For a fixed $\ell$ and a large $n$ we have
	\[\bar x_n(\eta_\ell)\cdot e_1<x_n\cdot e_1<\bar x_n(\zeta_\ell)\cdot e_1
	\quad\text{and}\quad
	\bar x_n(\eta_\ell)\cdot e_2>x_n\cdot e_2>\bar x_n(\zeta_\ell)\cdot e_2.\]
Applying \eqref{comparison} we have
	\[\frac{Z_{x+e_1,\bar x_n(\eta_\ell)}}{Z_{x,\bar x_n(\eta_\ell)}}\le\frac{Z_{x+e_1,x_n}}{Z_{x,x_n}}\le\frac{Z_{x+e_1,\bar x_n(\zeta_\ell)}}{Z_{x,\bar x_n(\zeta_\ell)}}\,.\]
Take $n\to\infty$ and apply the already proved version of \eqref{fixed-xi-lim} for the sequences $\bar x_n(\eta_\ell)$ and $\bar x_n(\zeta_\ell)$ to get for each $\ell$
	\[e^{-B^{\etamin_\ell-}(x,x+e_1,\what)}\le \varliminf_{n\to\infty}\frac{Z_{x+e_1,x_n}}{Z_{x,x_n}} \le \varlimsup_{n\to\infty}\frac{Z_{x+e_1,x_n}}{Z_{x,x_n}} \le e^{-B^{\zetamax_\ell+}(x,x+e_1,\what)}.\]
Send $\ell\to\infty$ to get the first line of \eqref{fixed-xi-lim}. The second line is similar.
\end{proof}

\section{Semi-infinite polymer measures}\label{sec:semiDLR}
In this section we prove general versions of our main results on rooted solutions, starting with Lemma \ref{lm:nondeg}.

\begin{proof}[Proof of Lemma \ref{lm:nondeg}]
Fix $x\in\V_m$, $m\in\Z$. Suppose $\Pi_x$ is degenerate. By \eqref{Pi-cons}
there exist $y\ge x$ and $n\ge m$ with $y\in\V_n$ and $\Pi_x(X_n=y)=0$. 
Then for $v\ge y$ with $v\cdot\ehat=k$
	\[0=\Pi_x(X_n=y)\ge\Pi_x(X_n=y,X_k=v)=\Pi_x(X_k=v)\,Q^\w_{x,v}(X_n=y).\]
Hence, $\Pi_x(X_k=v)=0$. This means that 
	\[\Pi_x\{\forall n\ge m: X_n\cdot e_1\le y\cdot e_1\}+\Pi_x\{\forall n\ge m:X_n\cdot e_2\le y\cdot e_2\}=1.\]
Denote the first probability by $\alpha$. We will show that 
		\begin{align}\label{alpha}
		\begin{split}
		&\Pi_x\{\forall n\ge m:X_n=x+(n-m)e_2\}=\alpha\quad\text{and}\\ 
		&\Pi_x\{\forall n\ge m:X_n=x+(n-m)e_1\}=1-\alpha.
		\end{split}
		\end{align} 
If \eqref{alpha} holds, then $\Pi_x=\alpha\Pi_x^{e_2}+(1-\alpha)\Pi_x^{e_1}$.  
Let us now prove \eqref{alpha}.

If $\alpha=0$, then also $\Pi_x\{\forall n\ge m:X_n=x+(n-m) e_2\}=0=\alpha$. If, on the other hand, $\alpha>0$, then either again $\Pi_x\{\forall n\ge m:X_n=x+(n-m) e_2\}=\alpha$ or 
there exist $k\ge m$ and $v\ge x$ with $v\cdot e_1\in(0, y\cdot e_1]$ such that
	\begin{align}\label{contra}
	\Pi_x\{\forall n\ge k:X_n=v+(n-k)e_2\}=\delta\in(0,\alpha].
	\end{align}
Let $\ell=(v-x)\cdot e_2$. Then $v-x=(k-m-\ell)e_1+\ell e_2$. For any $n\ge k$
	\begin{align*}
	\delta
	&\le\Pi_x\{X_i=v+(i-k)e_2,\,k\le i\le n\}\\
	&=\Pi_x\{X_n=v+(n-k)e_2\}\,Q^\w_{x,v+(n-k)e_2}\{X_i=v+(i-k)e_2,\,k\le i\le n\}\\
	&\le\frac{Z_{x,v} e^{\sum_{i=k}^{n-1}\w_{v+(i-k)e_2}}}{Z_{x,v+(n-k)e_2}}\\
	&\le\frac{Z_{x,v} \exp\Bigl\{\sum_{i=k}^{n-1}\w_{v+(i-k)e_2}\Bigr\}}{\exp\Bigl\{\sum_{i=0}^{\ell+n-k-1}\w_{x+ie_2}\Bigr\}\exp\Bigl\{\sum_{i=0}^{k-m-\ell-1}\w_{x+(\ell+n-k)e_2+ie_1}\Bigr\}}\\
	&=\frac{Z_{x,v} \exp\Bigl\{\sum_{i=k}^{n-1}(\w_{v+(i-k)e_2}-\w_{x+(i+\ell-k)e_2})\Bigr\}}{\exp\Bigl\{\sum_{i=0}^{\ell-1}\w_{x+ie_2}\Bigr\}\exp\Bigl\{\sum_{i=0}^{k-m-\ell-1}\w_{x+(\ell+n-k)e_2+ie_1}\Bigr\}}\,.
	\end{align*}
	
Let \label{Omnondeg}$\Omnondeg$ be the intersection of the events 
	\[\biggl\{\exists n\ge k:\frac{\exp\Bigl\{\sum_{i=k}^{n-1}(\w_{v+(i-k)e_j}-\w_{x+(i+\ell-k)e_j})\Bigr\}}{\exp\Bigl\{\sum_{i=0}^{k-m-\ell-1}\w_{x+(\ell+n-k)e_j+ ie_{3-j}}\Bigr\}}\le e^{-r}\biggr\}\]
over all $x,v\in\Z^2$ such that $v\ge x$, $r\in\N$, $j\in\{1,2\}$, and integers $k\ge\ell=(v-x)\cdot e_j$ and $m=(v-x)\cdot(e_1+e_2)$.
The event $\Omnondeg$ has full $\P$-probability. Indeed, for each $r\in\N$
	\begin{align*}
%	&\P\biggl(\exists n\ge k:\frac{\exp\Bigl\{\sum_{i=k}^{n-1}(\w_{v+\epsilon(i-k)e_2}-\w_{x+\epsilon(i+\ell-k)e_2})\Bigr\}}{\exp\Bigl\{\sum_{i=0}^{k-\ell-1}\w_{x+\epsilon(\ell+n-k)e_2+\epsilon ie_1}\Bigr\}}\le e^{-r}\biggr)\\ %\label{aux0101}\\
%	&\qquad=\P\biggl(\exists n\ge0:\frac{\exp\Bigl\{\sum_{i=0}^{n-1}(\w_{e_1+ie_2}-\w_{ie_2})\Bigr\}}{\exp\Bigl\{\sum_{i=0}^{k-\ell-1}\w_{(i+2)e_1}\Bigr\}}\le e^{-r}\biggr)=1.%\notag
	&\P\biggl(\exists n\ge k:\frac{\exp\Bigl\{\sum_{i=k}^{n-1}(\w_{v+(i-k)e_j}-\w_{x+(i+\ell-k)e_j})\Bigr\}}{\exp\Bigl\{\sum_{i=0}^{k-m-\ell-1}\w_{x+(\ell+n-k)e_j+ ie_{3-j}}\Bigr\}}\le e^{-r}\biggr)\\ %\label{aux0101}\\
	&\qquad=\P\biggl(\exists n\ge0:\frac{\exp\Bigl\{\sum_{i=0}^{n-1}(\w_{e_1+ie_2}-\w_{ie_2})\Bigr\}}{\exp\Bigl\{\sum_{i=0}^{k-m-\ell-1}\w_{(i+2)e_1}\Bigr\}}\le e^{-r}\biggr)=1.%\notag
	\end{align*}
The first equality is because weights are i.i.d.\ and hence the distribution of the two ratios is the same. 
The second equality holds
because $\sum_{i=0}^{n-1}(\w_{e_1+ie_2}-\w_{ie_2})$ is a sum of i.i.d.\ centered 
nondegenerate random variables and hence has liminf $-\infty$.  
	
For $\w\in\Omnondeg$ we have
	\[\delta\le \frac{Z_{x,v} e^{-r}}{e^{\sum_{i=0}^{\ell-1}\w_{x+ie_2}}}\]
for all $r\in\N$. Taking $r\to\infty$ gives a contradiction. Therefore, \eqref{contra} cannot hold. The first equality in \eqref{alpha} is proved. The other one is similar.
\end{proof}

Since $\{\forall n\ge m:X_n=x+(n-m)e_{3-i}\}\subset\{\forall n\ge m: X_n\cdot e_i\le y\cdot e_i\}$, \eqref{alpha} implies 
that for any $\w\in\Omnondeg$, $x\in\V_m$, $m\in\Z$, $\Pi_x\in\DLR_x^\w$, $y\in x+\Z^2_+$, and $i\in\{1,2\}$ we have
	\begin{align}\label{cross}
	\Pi_x\Bigl\{\{\forall n\ge m: X_n\cdot e_i\le y\cdot e_i\}\setminus\{\forall n\ge m:X_n=x+(n-m)e_{3-i}\}\Bigr\}=0.
	\end{align}

\begin{lemma}\label{lm:DLR-MC}
Fix $\w\in\Omega$ and $x\in\Z^2$. Let $\Pi_x\in\DLR_x^\w$ be a nondegenerate solution.
Then $\Pi_x$ is a Markov chain with transition probabilities
	\begin{align}\label{pi-Pi}
	\pi^x_{y,y+e_i}(\w)=\frac{\Pi_x(y+e_i)\,Z_{x,y}\,e^{\w_y}}{\Pi_x(y)\,Z_{x,y+e_i}}\,,\quad y\ge x,i\in\{1,2\}.
	\end{align}
\end{lemma}

\begin{proof}%[Proof of Lemma \ref{lm:DLR-MC}]
Let $x\in\V_m$ and $y\in\V_n$, $n\ge m$. Fix an admissible path $x_{m,n}$ with $x_m=x$ and $x_n=y$. Compute for $i\in\{1,2\}$
	\begin{align*}
	\Pi_x(X_{n+1}=y+e_i\,|\,X_{m,n}=x_{m,n})
	%=\frac{\Pi_x(X_{m,n}=x_{m,n},X_{n+1}=y+e_i)}{\Pi_x(X_{m,n}=x_{m,n})}\\
	&=\frac{\Pi_x(X_{n+1}=y+e_i)\,Z_{x,y}\,e^{\sum_{i=m}^{n}\w_{x_i}}}{\Pi_x(X_n=y)\,Z_{x,y+e_i}\,e^{\sum_{i=m}^{n-1}\w_{x_i}}}\\
	=\frac{\Pi_x(X_{n+1}=y+e_i)\,Z_{x,y}\,e^{\w_y}}{\Pi_x(X_n=y)\,Z_{x,y+e_i}}
%	=\frac{\Pi_x(X_{n+1}=y+e_i)Q_{x,y+e_i}^\w(X_n=y)}{\Pi_x(X_n=y)}\\
%	%=\frac{\Pi_x(X_n=y,X_{n+1}=y+e_i)}{\Pi_x(X_n=y)}\\
	&=\Pi_x(X_{n+1}=y+e_i\,|\,X_n=y).\qedhere
%	=\pi^x_{y,y+e_i}(\w).\qedhere
%%	=\frac{\ddd\sum_{x_{0,n}:x_n=y}\Pi_x(X_{0,n}=x_{0,n},X_{n+1}=y+e_i)}{\ddd\sum_{x_{0,n}:x_n=y}\Pi_x(X_{0,n}=x_{0,n})}\\
%%	&\quad\ =\frac{\Pi_x(X_{n+1}=y+e_i)\,Z_{x,y}\ddd\sum_{x_{0,n}:x_n=y}e^{\sum_{i=0}^n\w_{x_i}}}{\Pi_x(X_n=y)\,Z_{x,y+e_i}\ddd\sum_{x_{0,n}:x_n=y}e^{\sum_{i=0}^{n-1}\w_{x_i}}}
%%	=\frac{\Pi_x(X_{n+1}=y+e_i)\,Z_{x,y}\,e^{\w_y}}{\Pi_x(X_n=y)\,Z_{x,y+e_i}}\\
	\end{align*}
\end{proof}

\begin{remark}
The above makes sense even for degenerate solutions. Transitions $\pi^x_{y,y+e_i}$ are then only defined at points $y\ge x$ that are reachable from $x$ with positive $\Pi_x$-probability, i.e.\ such that 
$\Pi_x(y)>0$. One can check that these transitions keep the chain within this class of %reachable 
points.
\end{remark}

Next, we relate nondegenerate DLR solutions in environment $\w$ and cocycles that recover the potential $\{\w_x(\w)\}$.

\begin{theorem}\label{th:B-DLR}
%\normalmarginpar\note{corresponds to Lm 4.1(a) of \cite{Geo-Ras-Sep-17-ptrf-2}, but also gives reverse direction!}
Fix $\w\in\Omega$ and $x\in\V_m$, $m\in\Z$. Then $\Pi_x$ is a nondegenerate DLR solution in environment $\w$ if, and only if, there exists a cocycle $\{B(u,v):u,v\ge x\}$ that satisfies recovery \eqref{beta-rec} 
and
	\begin{align}\label{Pi-B-cons}
	\Pi_x(X_{m,n}=x_{m,n})=e^{\sum_{k=m}^{n-1}\w_{x_k}-B(x,x_n)}
	\end{align}
for every admissible path $x_{m,n}$ starting at $x_m=x$.   This cocycle is uniquely determined by the formula
	\begin{align}\label{Pi-B}
	e^{-B(u,v)}=\frac{\Pi_x(v)}{\Pi_x(u)}\cdot\frac{Z_{x,u}}{Z_{x,v}},\quad u,v\ge x.
	\end{align}
It  satisfies
	\begin{align}\label{B-Z}
	e^{-B(x,y)}= E^{\Pi_x}\Bigl[\frac{Z_{y,X_n}}{Z_{x,X_n}}\Bigr]\quad\text{for all }y\ge x\text{ and }n\ge y\cdot\ehat\,.
	\end{align}
The transition probabilities of $\Pi_x$ are then given by
	\begin{align}\label{pi-B}
	\pi^x_{y,y+e_i}(\w)=e^{\w_y-B(y,y+e_i)},\quad y\ge x,i\in\{1,2\}.
	\end{align}
\end{theorem}

When $\Pi_x$ is given we denote the corresponding cocycle by $B^{\Pi_x}(u,v)$. Conversely, when $B$ is given, we denote the corresponding DLR solution in environment $\w$ (that $B$ recovers) by $\Pi_x^{B}$. 

\begin{proof}[Proof of Theorem \ref{th:B-DLR}]
Given a nondegenerate solution $\Pi_x\in\DLR_x^\w$ define $B$ via \eqref{Pi-B}. Telescoping products check that this is a cocycle.  
To check the recovery property write $Q^\w_{x,u+e_i}(u) = Z_{x,u}e^{\w_u} (Z_{x,u+e_i})^{-1}$. Hence, 
	\begin{align*}
	&e^{-B(u,u+e_1)}+e^{-B(u,u+e_2)}\\
	&\quad=\frac{e^{-\w_u}}{\Pi_x(u)}\Bigl(\Pi_x(u+e_1)Q^\w_{x,u+e_1}(u)+\Pi_x(u+e_2)Q^\w_{0,u+e_2}(u)\Bigr)\\
	&\quad=\frac{e^{-\w_u}}{\Pi_x(u)}\Bigl(\Pi_x(u,u+e_1\in X_\bbullet)+\Pi_x(u,u+e_2\in X_\bbullet)\Bigr)=e^{-\w_u}.
	\end{align*}
\eqref{pi-B} follows from \eqref{pi-Pi} and \eqref{Pi-B} and then \eqref{Pi-B-cons} follows from \eqref{pi-B}, the Markov property of $\Pi_x$, and the cocycle property of $B$.

Conversely, given a cocycle $B$ that recovers the potential, define $\pi^x$ via \eqref{pi-B}. Recovery implies that $\pi^x$ are transition probabilities. Let $\Pi_x$ be the distribution of the Markov chain
with these transition probabilities. Again, \eqref{pi-B} and the cocycle property imply \eqref{Pi-B-cons}. 
In particular, $\Pi_x$ is not degenerate. 
For $y\ge x$ adding \eqref{Pi-B-cons} over all admissible paths from $x$ to $y$ gives
	\begin{align}\label{Pi-Z-B}
	\Pi_x(y)=Z_{x,y}e^{-B(x,y)}.
	\end{align}   
This and the cocycle property of $B$ imply \eqref{Pi-B}. Using $y=x_n$ 
and solving for $e^{-B(x,y)}$ in \eqref{Pi-Z-B} then plugging back into \eqref{Pi-B-cons} 
gives 
	\[\Pi_x(x_{m,n})=e^{\sum_{k=m}^{n-1}\w_{x_k}-B(x,x_n)}=\Pi_x(y)\frac{e^{\sum_{k=m}^{n-1}\w_{x_k}}}{Z_{x,y}}
 	=\Pi_x(y)\,Q^\w_{x,y}(x_{m,n}),\]
which says $\Pi_x$ is a DLR solution in environment $\w$.

Lastly, we prove \eqref{B-Z}. Let $k=y\cdot\ehat\ge m$. Then
\begin{align*}%\label{aux0001}
Z_{x,y}E^{\Pi_x}\Bigl[\frac{Z_{y,X_n}}{Z_{x,X_n}}\Bigr]
&=\sum_{\substack{v\ge y\\v\in\V_n}}\Pi_x(X_n=v)\frac{Z_{x,y}Z_{y,v}}{Z_{x,v}}\\
&=\sum_{\substack{v\ge y\\v\in\V_n}}\Pi_x(X_n=v)Q_{x,v}^\w(X_k=y)\\
&=\sum_{\substack{v\ge y\\v\in\V_n}}\Pi_x(X_k=y,X_n=v)
=\Pi_x(X_k=y).
\end{align*}

%	\begin{align}
%%	\pi^x_{y,y+e_i}(\w)
%%	&=\frac{\Pi_x(X_{m+1}=y+e_i)\,Z_{x,y}\,e^{\w_y}}{\Pi_x(X_m=y)\,Z_{x,y+e_i}}\notag\\
%%	&=\frac{Z_{x,y}\,e^{\w_y}\sum_{v\in x+D_n}\Pi_x(X_n=v)Q^\w_{x,v}(X_{m+1}=y+e_i)}{Z_{x,y+e_i}\sum_{v\in x+D_n} \Pi_x(X_n=v)Q^\w_{x,v}(X_m=y)}\notag\\
%%	&=\frac{Z_{x,y}\,e^{\w_y}\sum_{v\in x+D_n}\Pi_x(X_n=v)Z_{x,y+e_i}Z_{y+e_i,v}/Z_{x,v}}{Z_{x,y+e_i}\sum_{v\in x+D_n} \Pi_x(X_n=v)Z_{x,y}Z_{y,v}/Z_{x,v}}\notag\\
%%	&=\frac{e^{\w_y}\sum_{v\in x+D_n}\Pi_x(X_n=v)Z_{y+e_i,v}/Z_{x,v}}{\sum_{v\in x+D_n} \Pi_x(X_n=v)Z_{y,v}/Z_{x,v}}\notag\\
%	\Pi_x(X_{m,k}=x_{m,k})
%	&=\sum_{v\in x+D_n}\Pi_x(X_n=v)\,Q_{x,v}^\w(X_{m,k}=x_{m,k})\notag\\
%	&=\sum_{v\in x+D_n}\Pi_x(X_n=v)\frac{e^{\sum_{i=m}^{k-1}\w_{x_i}}\,Z_{y,v}}{Z_{x,v}}\notag\\
%	&=e^{\sum_{i=m}^{k-1}\w_{x_i}}E^{\Pi_x}\Bigl[\frac{Z_{y,X_n}}{Z_{x,X_n}}\Bigr].\label{aux0001}
%%	&=e^{\sum_{i=0}^{m-1}\w_{x_i}}E^{\Pi_x}\Bigl[\prod_{i=0}^{m-1}\frac{Z_{x_{i+1},X_n}}{Z_{x_i,X_n}}\Bigr]\notag\\
%	\end{align}
Then \eqref{B-Z} follows from this and \eqref{Pi-Z-B}.
The theorem is proved.
\end{proof}  

\begin{remark}
We can make sense of the above theorem even for degenerate solutions if we allow cocycle $B$ to take the value $\infty$. %See for example the proof of Theorem \ref{th:B-lim}.
\end{remark}

\begin{remark}
The above theorem gives the following interesting fact.  If $B_1(u,v)$ and $B_2(u,v)$, $u,v\ge x$, are two cocycles that recover the potential,  then for any $s\in[0,1]$
	\[B(x,y)=-\log\bigl( s e^{-B_1(x,y)} + (1-s) e^{-B_2(x,y)} \bigr)\quad\text{and}\quad B(u,v)=B(x,v)-B(x,u),\quad u,v\ge x\]
is also a cocycle that recovers the potential. This is in fact not limited to a convex combination of two recovering cocycles and works for any convex mixture of them.
\end{remark}    

%These correspond to 
%degenerate cocycles $B^{e_i}(x+ke_i,x+(k+1)e_i)=\w_{x+ke_i}$, $k\in\Z_+$, and $B^{e_i}(u,u+e_j)=\infty$ if $u\not\in x+\Z_+ e_i$ or $j=3-i$.
%Q^\w_{x,x+ke_i}(x_{0,n})=1
The DLR solutions that correspond to cocycles $B^{\xi\pm}$, $\xi\in\ri\Uset$, will play a key role in what follows. We will denote these by $\Pi_x^{\xi\pm,\what}$ and the corresponding transition probabilities by 
$\pi^{\xi\pm,\what}$. These transition probabilities do not depend on the starting point $x$. When $B^{\xi-}=B^{\xi+}=B^\xi$  we also write $\Pi_x^{\xi,\what}$ and $\pi^{\xi,\what}$. In addition to recovering the potential, the $B^{\xi\pm}$ cocycles are also $\That$-covariant when $\xi$ is deterministic.
We next show how these observations relate to the law of large numbers for the corresponding DLR solution.

%Given a configuration $\w$, a cocycle $B$ that recovers the weights $\w$, and a starting point $x\in\Z^2$ let $\Pi^{B,\what}_x$ be the Markov chain that starts at $x$ and has transition probabilities
%	\[\pi^B_{y,y+e_i}(\w)=e^{\w_y-B(y,y+e_i)},\quad y\in\Z^2,i\in\{1,2\}.\]
%%Then, 
%%	\[\Pi^{B,\what}_x(X_{0,n}=x_{0,n})=e^{\sum_{i=0}^{n-1}\w_{x_i}-B(x,x_n)}=Q_{x,x_n}^\w(X_{0,n}=x_{0,n})\,e^{-B(x,x_n)}.\]
%
%\begin{lemma}\label{B-DLR}
%\addmath{corresponds of Lm 4.1(a) of \cite{Geo-Ras-Sep-17-ptrf-2}}
%$\Pi^{B,\what}_x$ solves the DLR equations in environment $\w$. 
%\end{lemma}
%
%\begin{proof}
%Apply the Markov and cocycle properties to get
%%\Pi_x(X_{0,n}=x_{0,n})=\Pi_x(X_n=x_n)Q^\w_{x,x_n}(X_{0,n}=x_{0,n})=\frac{\Pi_x(X_n=x_n)\,e^{\sum_{i=0}^{n-1}\w_{x_i}}}{Z_{x,x_n}}.
%	\[\Pi^{B,\what}_x(X_{0,n}=x_{0,n})=\prod_{i=0}^{n-1}\pi^B_{x_i,x_{i+1}}(\w)=e^{\sum_{i=0}^{n-1}\w_{x_i}-B(x,x_n)}.\]
%Adding over all admissible paths $x_{0,n}$ from $x_0=x$ to $x_n=y$ gives
%	\begin{align}\label{Pi-B}
%	\Pi^{B,\what}_x(X_n=y)=Z_{x,y}\,e^{-B(x,y)}.
%	\end{align}
%Therefore
%	\[\frac{\Pi^{B,\what}_x(X_n=x_n)\,e^{\sum_{i=0}^{n-1}\w_{x_i}}}{Z_{x,x_n}}=e^{\sum_{i=0}^{n-1}\w_{x_i}-B(x,x_n)}=\Pi^{B,\what}_x(X_{0,n}=x_{0,n}).\qedhere\]
%%and we have shown \eqref{Pi-cons}.
%\end{proof}

\begin{theorem}\label{th:PiB-lln}
Let $B$ be an  $L^1(\Omhat,\Phat)$ $\That$-covariant cocycle that recovers the weights $(\w_x)$. There exists an event \label{OmBhat}$\OmBhat{B}\subset\Omhat$ such that $\Phat(\OmBhat{B})=1$ and for every $\what\in\OmBhat{B}$ and $x\in\Z^2$ the 
distribution of $X_n/n$ under $\Pi^{B(\what)}_x$ satisfies a large deviation principle with convex rate function $I_B(\xi)=-h(B)\cdot\xi-\fe(\xi)$, $\xi\in\Uset$.
Consequently, $\Pi_x^{B(\what)}$ is strongly directed into $\Uset_{h(B)}$.
%the interval of directions dual to $h(B)$:
%	\fixtext{define the supergradient $\partial\Lambda(\xi)$}
%	\[\{\xi\in\ri\Uset:h(B)\cdot\xi+\fe(\xi)=0\}=\{\xi\in\ri\Uset:-h(B)\in\partial\Lambda(\xi)\}.\] 
\end{theorem} 

\begin{proof}
From equation \eqref{Pi-Z-B} and the shape theorems \eqref{shape} and \eqref{shape-B} for the free energy and shift-covariant cocycles we get that $\Phat$-almost surely, 
for all $x\in\Z^2$, all $\xi\in\Uset$, and any sequence $x_n\ge x$ with $x_n\in\V_n$ and $x_n/n\to\xi$
	\[n^{-1}\log \Pi^{B}_x(X_n=x_n) = n^{-1}\log Z_{x,x_n} - n^{-1} B(x,x_n)\mathop{\longrightarrow}_{n\to\infty}\fe(\xi)+h(B)\cdot\xi.\]
The large deviation principle follows.   
Then Borel-Cantelli and strict positivity of $I_B$ off of $\Uset_{h(B)}$ imply the directedness claimed in the theorem.
%Directedness also follows.  Indeed, if we denote the interval of directions dual to $h(B)$ by $[\eta,\zeta]$, then we see that $\Pi^{B(\what)}_x\{X_n\cdot e_1\le n(\eta\cdot e_1-\e)\}$ decays exponentially fast for every $\e>0$ and the Borel-Cantelli lemma implies
%that $\varliminf n^{-1}X_n\cdot e_1\ge\eta\cdot e_1$, $\Pi_x^{B(\what)}$-almost surely. Similarly, we get $\varlimsup n^{-1}X_n\cdot e_1\le\zeta\cdot e_1$, $\Pi_x^{B(\what)}$-almost surely. 
\end{proof}

Next, couple $\Pi_x^{\xi\pm,\what}$, $x\in\Z^2$, $\xi\in\ri\Uset$, pathwise, as described in Section \ref{sub:coupling}. Denote the coupled up-right paths by $X^{x,\xi\pm,\what}_{m,\infty}$, $x\in\V_m$, $m\in\Z$, $\xi\in\ri\Uset$.
%Let $\Omdir$ be the full measure event from Theorem \ref{th:DLR-dir1}. \addmath{may need to intersect this with the event where monotonicity of $B^\xi_\pm$ holds}
%To this end, let $\{\unif(y):y\in\Z^2\}$ be i.i.d.\ Uniform$(0,1)$ random variables, with joint probability distribution $\bfP$ and corresponding expectation operator $\bfE$.
%Define admissible paths $X^{x,\xi\pm,\what}_{m,\infty}$ as follows.  For all $\xi\in\Uset$ and $x\in\V_m$, $m\in\Z$, set $X_m^{x,\xi\pm,\what}=x$. 
%For $k\ge m$ and $\xi\in\ri\Uset$ let 
%	\begin{align}\label{couple}
%	X^{x,\xi\pm,\what}_{k+1}=X^{x,\xi\pm,\what}_k+\begin{cases}e_1&\text{if }\unif(X^{x,\xi\pm,\what}_k)<\pi^{\xi\pm,\what}(X^{x,\xi\pm,\what}_k,X^{x,\xi\pm,\what}_k+e_1),\\e_2&\text{otherwise.}\end{cases}
%	\end{align}
When $B^{\xi-}=B^{\xi_+}=B^\xi$  we write $X^{x,\xi,\what}$.
For $i\in\{1,2\}$ set $X_k^{x,e_i\pm,\what}=X_k^{x,e_i,\what}=x+(k-m)e_i$, $k\ge m$.  
%For $\what\in\Omhat$, $x\in\Z^2$, and $\xi\in\Uset$, $\Pi_x^{\xi\pm,\what}$ is the distribution of $X^{x,\xi\pm,\what}$ under $\bfP$. 

When $\what\in\Ommonohat$, the event from Theorem \ref{thm:cocyexist} on which \eqref{mono} holds, and $x\in\V_m$, $m\in\Z$, paths $X^{x,\xi\pm,\what}$ are ordered: For any $\xi,\zeta\in\Uset$ with $\xi\cdot e_1<\zeta\cdot e_1$ and any $k\ge m$, 
	\begin{align}\label{path-order}
	X_k^{x,\xi-,\what}\cdot e_1\le X_k^{x,\xi+,\what}\cdot e_1\le X_k^{x,\zeta-,\what}\cdot e_1 \le X_k^{x,\zeta+,\what}\cdot e_1.
	\end{align}
%To see this, note that the above paths all start at the same point $x$, and then their moves are ordered due to 
%monotonicity \eqref{??} of $B^\xi_\pm$. 
%For example, say $X_k^{x,\xi-,\what}=X_k^{x,\xi+,\what}=y$ and $X_{k+1}^{x,\xi+,\what}=y+e_2$. This means $\unif(y)\ge\pi^{\xi+,\what}_{y,y+e_1}$.
%Then $B^{\xi-}(y,y+e_1)\ge B^{\xi+}(x,x+e_1)$ %and $B^\xi_-(y,y+e_2)\ge B^\xi_+(x,x+e_2)$ imply 
%implies that $\pi^{\xi-,\what}_{y,y+e_1}\le\pi^{\xi+,\what}_{y,y+e_1}\le \unif(y)$ and
%therefore $X_{k+1}^{x,\xi-,\what}=y+e_2$ as well.

\begin{theorem}\label{th:DLR-dir1}
There exists an event $\Omexisthat\subset\Omhat$ such that $\Phat(\Omexisthat)=1$ and 
for every $\what\in\Omexisthat$, $x\in\Z^2$,  and $\xi\in\ri\Uset$, 
$\Pi_x^{\xi\pm,\what}\in\DLR_x^\w$ and are, respectively, strongly $\Uset_{\xi\pm}$-directed.
%there exist $\Uset_{\xi-}$- and $\Uset_{\xi+}$-directed 
For $i\in\{1,2\}$, the trivial polymer measure $\Pi_x^{e_i}$ gives a DLR solution that is strongly $\Uset_{e_i}$-directed.
\end{theorem}

\begin{proof}
Let $\Udense$ be a countable subset of $\ri\Uset$ that contains all of $(\ri\Uset)\setminus\Diff$ and a countable dense subset of $\Diff$.
Let 
	\[\Omexisthat=\Ommonohat\cap\bigcap_{\xi\in\Udense\cap\Diff}\Omconthat{\xi}\cap\bigcap_{\xi\in\Udense}\bigl(\OmBhat{B^{\xi+}}\cap\Omconsthat{\xi+}\cap\OmBhat{B^{\xi-}}\cap\Omconsthat{\xi-}\bigr).\label{Omexisthat}\]
When $\xi\in\Udense$ and $\what\in\OmBhat{B^{\xi+}}\cap\Omconsthat{\xi+}$ Lemma \ref{lm:tilt} says $h(B^{\xi+})=-\nabla\fe(\xi+)$ and then
Theorem \ref{th:PiB-lln} says that $\Pi_x^{\xi+,\what}$ is strongly $\Uset_{\xi+}$-directed, for all $x\in\Z^2$. A similar argument works for $\Pi_x^{\xi-,\what}$.

Now fix $\xi\in(\ri\Uset)\setminus\Udense$ and $\what\in\Omexisthat$. 
If $\xi\cdot e_1<\ximax\cdot e_1$ then pick $\zeta\in\Udense\cap\Diff$ such that $\xi\cdot e_1<\zeta\cdot e_1<\ximax\cdot e_1$.
Then $\ximax=\zetamax$. 
The ordering of paths \eqref{path-order} implies that $X_k^{x,\xi+,\what}\cdot e_1\le X_k^{x,\zeta,\what}\cdot e_1$ for all $k\ge m$ (there is no need for the $\pm$ distinction for $\zeta\in\Udense\cap\Diff$).
Since the distribution of the latter path is $\Pi_x^{\zeta,\what}$ and it is strongly $\Uset_\zeta$-directed, we deduce that 
	\[\varlimsup_{n\to\infty}n^{-1}X_n\cdot e_1\le\zetamax\cdot e_1=\ximax\cdot e_1,\quad\Pi_x^{\xi+,\what}\text{-almost surely}.\]

If $\xi\in(\ri\Uset)\setminus\Udense$ is such that $\xi=\ximax$, then let $\e>0$ and pick $\zeta\in\Udense\cap\Diff$ such that 
$\xi\cdot e_1<\zeta\cdot e_1\le\zetamax\cdot e_1<\xi\cdot e_1+\e=\ximax\cdot e_1+\e$.
This is possible because $\nabla\fe(\zeta)$ converges to but never equals $\nabla\fe(\xi)$ as $\zeta\cdot e_1\searrow\xi\cdot e_1$.
(Note that $\xi\in\Diff$.)
The same ordering argument as above implies
	\[\varlimsup_{n\to\infty}n^{-1}X_n\cdot e_1\le\zetamax\cdot e_1\le\ximax\cdot e_1+\e,\quad\Pi_x^{\xi+,\what}\text{-almost surely}.\]
Take $\e\to0$. Similarly, 
	\[\varliminf_{n\to\infty}n^{-1}X_n\cdot e_1\ge\ximin\cdot e_1,\quad\Pi_x^{\xi-,\what}\text{-almost surely}.\]
Appealing once again to the path ordering, we see now that both $\Pi_x^{\xi\pm,\what}$ are strongly directed into $\Uset_\xi$.  
Since $\xi\in\Diff$ we have $\Uset_{\xi+}=\Uset_{\xi-}=\Uset_\xi$. The theorem is proved.
\end{proof}

\begin{proof}[Proof of Theorem \ref{th:main1}]
Recall the set $\Udense$ from the proof of Theorem \ref{th:DLR-dir1}. %that $(\ri\Uset)\setminus\Diff\subset\Udense$. Hence, 
When $\xi\in(\ri\Uset)\setminus\Diff$ and $\what\in\Omexisthat$, $\what\in\Omconsthat{\xi+}\cap\Omconsthat{\xi-}$ and we have by Lemma \ref{lm:tilt}
	\begin{align}\label{diff-mean}
	\Ehat[B^{\xi-}(0,e_1)\,|\,\sI]=e_1\cdot\nabla\fe(\xi-)>e_1\cdot\nabla\fe(\xi+)=\Ehat[B^{\xi+}(0,e_1)\,|\,\sI].
	\end{align}
By the ergodic theorem there exists a full $\Phat$-measure event $\Omtemphat'''$ %\subset\Omexisthat\cap\bigcap_{\xi\in(\ri\Uset)\setminus\Diff}(\Omconsthat{\xi+}\cap\Omconsthat{\xi-})$ 
such that 
for each $\what\in\Omtemphat'''$, $\xi\in(\ri\Uset)\setminus\Diff$, and $x\in\Z^2$ there is a $y\ge x$ such that $B^{\xi-}(y,y+e_1)\ne B^{\xi+}(y,y+e_1)$. This implies $\Pi_x^{\xi-,\what}\ne\Pi_x^{\xi+,\what}$. 

Recall the projection $\pi_\Omega$ from $\Omhat$ onto $\Omega$. 
There exists a family of regular conditional distributions $\mu_\w(\acdot)=\Phat(\acdot\,|\,\pi_\Omega^{-1}(\w))$ and a Borel set 
\label{Omreg}$\Omreg\subset\Omega$ such that $\P(\Omreg)=1$ and for every $\w\in\Omreg$, 
$\mu_\w(\pi_\Omega^{-1}(\w))=1$. See Example 10.4.11 in \cite{Bog-07}.
Since
	%\begin{align}\label{Omexisthat}
	\[\int\mu_\w(\Omtemphat''')\,\P(d\w)=\Phat(\Omtemphat''')=1\]
	%\end{align}
we see that $\mu_\w(\Omtemphat''')=1$, $\P$-almost surely. 
Set
	\begin{align}\label{eq:Omexist}
	\Omexist=\Omreg\cap\bigl\{\w\in\Omega:\mu_\w(\Omtemphat''')=1\bigr\}.
	\end{align}\label{Omexist}
Then $\P(\Omexist)=1$. We take $\w\in\Omexist$ so that $\mu_\w(\pi_\Omega^{-1}(\w)\cap\Omtemphat''')=\mu_\w(\Omtemphat''')=1$.
%An event with positive probability cannot be empty 
There exists $\what\in\Omtemphat'''$ with $\pi_\Omega(\what)=\w$. 
For $\xi\in\Uset$, the $\Uset_{\xi\pm}$-directed solutions in the claim are $\Pi_x^{\xi\pm,\what}$. 
\end{proof}

\begin{theorem}\label{th:DLR=B}
Fix $\xi\in\Diff$. Assume $\ximin,\ximax\in\Diff$. There exists a $\That$-invariant event
$\OmBusPihat\subset\Omhat$ such that $\Phat(\OmBusPihat)=1$ and for every $\what\in\OmBusPihat$ and $x\in\Z^2$, 
$\Pi_x^{\xi\pm,\what}=\Pi_x^{\xi,\what}$ is the unique weakly $\Uset_\xi$-directed solution in $\DLR_x^\w$.
It is also strongly directed into $\Uset_\xi$ and for any $\Uset_\xi$-directed sequence $(x_n)$ 
%$x_n\in x+D_n$ such that
%	\[\ximin\cdot e_1\le\varliminf_{n\to\infty} n^{-1}x_n\cdot e_1 \le \varlimsup_{n\to\infty} n^{-1}x_n\cdot e_1\le\ximax\cdot e_1\]
the sequence of quenched point-to-point polymer measures $Q^\w_{x,x_n}$ converges weakly to $\Pi_x^{\xi,\what}$.
The family $\{\Pi_x^{\xi,\what}:x\in\Z^2,\what\in\Omhat\}$ is consistent and $\That$-covariant.
\end{theorem}

\begin{proof}
Let $\eta_k,\zeta_k\in\ri\Uset$ be such that $\eta_k\cdot e_1$ strictly increases to $\ximin\cdot e_1$ and $\zeta_k\cdot e_1$ strictly decreases to $\ximax\cdot e_1$.
Let 
	\[\OmBusPihat=\pi_\Omega^{-1}(\Omnondeg)\cap\Ommonohat\cap\Omconthat{\ximin}\cap\Omconthat{\ximax}\cap\OmBusallhat\cap\OmBhat{B^{\xi+}}.\label{OmBusPihat}\] 
	%\cap\bigcap_{k\in\N}\bigl(\Omhat_{[\etamin_k,\etamax_k]}\cap\Omhat_{[\zetamin_k,\zetamax_k]}\bigr).\label{OmBusPihat}\]
%	$, $k\in\N$, and  the event on which 
%$B^{\etamin_k}_-(y,y+e_i,\what)\to B^{\ximin}(y,y+e_i,\what)$, $B^{\zetamax_k}_+(y,y+e_i,\what)\to B^{\ximax}(y,y+e_i,\what)$, 
%and
%	\begin{align}\label{B-lim}
%	\begin{split}
%	B^{\ximin}(y,y+e_i,\what)
%	&=B^{\ximax}(y,y+e_i,\what)=B^\xi(y,y+e_i,\what)\\
%	&=\lim_{n\to\infty}(\log Z_{y,x_n}-\log Z_{y+e_i,x_n}),
%	\end{split}
%	\end{align}
%for all $y\in \Z^2$, $i\in\{1,2\}$, and $\Uset_\xi$-directed sequences $(x_n)$.
%Note that since $\ximin,\xi,\ximax\in\Diff$, there is no $\pm$ distinction in $B^{\ximin}$, $B^{\ximax}$, and $B^\xi$. 
%Also, the limit in \eqref{B-lim} shows that we can consider these cocycles as functions of $\w$ instead of $\what$.

Take $\what\in\OmBusPihat$. Since $\what\in\OmBusallhat$, Theorem \ref{thm:Buselim1} implies that for all $y\in\Z^2$ and $k\in\N$
	\[\varliminf_{n\to\infty}\frac{Z_{y+e_1,\fl{n\eta_k}}}{Z_{y,\fl{n\eta_k}}}\ge e^{-B^{\etamin_k-}(y,y+e_1,\what)}
	\ \ \text{and}\ \ \varlimsup_{n\to\infty}\frac{Z_{y+e_2,\fl{n\eta_k}}}{Z_{y,\fl{n\eta_k}}}\le e^{-B^{\etamin_k-}(y,y+e_2,\what)}.\]

Fix $x\in\V_m$, $m\in\Z$, and $y\ge x$.
Fix $\e>0$. Since the choice of $\what$ guarantees continuity as $\etamin_k\to\ximin$ we can choose $k$ large then $n$ large so that 
%	\[e^{-B^{\etamin_k-}(y,y+e_2,\what)}\le e^{-B^{\ximin}(y,y+e_2,\what)}+\e/2.\]
%Then for $n$ large 
	\[\frac{Z_{y+e_2,\fl{n\eta_k}}}{Z_{y,\fl{n\eta_k}}}\le e^{-B^{\etamin_k-}(y,y+e_2,\what)}+\e/2\le e^{-B^{\ximin}(y,y+e_2,\what)}+\e.\]
%Note that in the last term we switched to $\w$ since $B^{\ximin}(\what)=B^{\ximin}(\w)$.

Let $\Pi_x\in\DLR_x^\w$ be weakly $\Uset_\xi$-directed.  Since $\eta_k,\zeta_k\in\ri\Uset$, both $n\eta_k\ge y$ and $n\zeta_k\ge y$ for large $n$. 
Applying \eqref{comparison} in the first inequality we have
	\begin{align*}
	&\Pi_x\bigl\{X_n\cdot e_1>\fl{n\eta_k\cdot e_1},\,X_n\ge y\bigr\}
	\le\Pi_x\Bigl\{\frac{Z_{y+e_2,X_n}}{Z_{y,X_n}}\le \frac{Z_{y+e_2,\fl{n\eta_k}}}{Z_{y,\fl{n\eta_k}}},X_n\ge y\Bigr\}\\
	&\qquad\qquad\le\Pi_x\Bigl\{\frac{Z_{y+e_2,X_n}}{Z_{y,X_n}}\le e^{-B^{\ximin}(y,y+e_2,\what)}+\e,X_n\ge y\Bigr\}\le1.
	\end{align*}
%But if  $X_n\cdot e_1>\fl{n\eta_k\cdot e_1}$, $X_n\ge y$, and $n\eta_k\ge y$, then \eqref{comparison} implies
%	\[\frac{Z_{y+e_1,X_n}}{Z_{y,X_n}}\ge \frac{Z_{y+e_1,\fl{n\eta_k}}}{Z_{y,\fl{n\eta_k}}}
%	\quad\text{and}\quad \frac{Z_{y+e_2,X_n}}{Z_{y,X_n}}\le \frac{Z_{y+e_2,\fl{n\eta_k}}}{Z_{y,\fl{n\eta_k}}}\,.\]
The weak directedness implies  the first probability converges to one.  Hence,
%This and the weak directedness imply
%	\[\Pi_x\bigl\{X_n\cdot e_1>\fl{n\eta_k\cdot e_1},\,X_n\ge y\bigr\}\mathop{\longrightarrow}_{n\to\infty}1.\]
%we get
	\[\lim_{n\to\infty}\Pi_x\Bigl\{\frac{Z_{y+e_2,X_n}}{Z_{y,X_n}}\le e^{-B^{\ximin}(y,y+e_2,\what)}+\e,X_n\ge y\Bigr\}=1.\]
Similarly, 
	\[\lim_{n\to\infty}\Pi_x\Bigl\{\frac{Z_{y+e_1,X_n}}{Z_{y,X_n}}\ge e^{-B^{\ximin}(y,y+e_1,\what)}-\e,X_n\ge y\Bigr\}=1.\]
Using a similar argument with the sequence $\zeta_k$ we also get
	\[\lim_{n\to\infty}\Pi_x\Bigl\{\frac{Z_{y+e_2,X_n}}{Z_{y,X_n}}\ge e^{-B^{\ximax}(y,y+e_2,\what)}-\e,X_n\ge y\Bigr\}=1\]
and
	\[\lim_{n\to\infty}\Pi_x\Bigl\{\frac{Z_{y+e_1,X_n}}{Z_{y,X_n}}\le e^{-B^{\ximax}(y,y+e_1,\what)}+\e,X_n\ge y\Bigr\}=1.\]
Since $\xi,\ximin,\ximax\in\Diff$ we have $\nabla\fe(\ximin-)=\nabla\fe(\ximax+)=\nabla\fe(\xi)$  
and by our choice of $\what$, $B^{\ximax}=B^{\ximin}=B^{\xi\pm}=B^\xi$. We have shown that  
$Z_{y+e_i,X_n}/Z_{y,X_n}$ converges in $\Pi_x$-probability 
to $e^{-B^\xi(y,y+e_i,\what)}$ for every $y\ge x$ and $i\in\{1,2\}$. 
%By going to a subsequence $n_j$ we  turn this into  $\Pi_x$-almost sure convergence.

Using any fixed admissible path from $x$ to $y$ and applying the cocycle property of $B^\xi$ and the above limit (to the increments of the path) we see that $Z_{y,X_{n}}/Z_{x,X_{n}}\to e^{-B^\xi(x,y,\what)}$, 
in $\Pi_x$-probability.
But if $\ell=y\cdot\ehat$ then
	\begin{align}\label{Z-bd}
	0\le\frac{Z_{y,X_{n}}}{Z_{x,X_{n}}}=\frac{Z_{y,X_{n}}}{\sum_{\substack{v\ge x\\v\in\V_\ell}}Z_{x,v}Z_{v,X_{n}}}\le\frac{1}{Z_{x,y}}<\infty.
	\end{align}
Since $\w=\pi_\Omega(\what)\in\Omnondeg$, Lemma \ref{lm:nondeg} says $\Pi_x$ is nondegenerate. Thus, bounded convergence and \eqref{B-Z} imply that $B^\xi$ is the cocycle that corresponds to $\Pi_x$.  In other words, $\Pi_x=\Pi_x^{\xi,\what}$. Since $B^\xi=B^{\xi+}$ and $\what\in\OmBhat{B^{\xi+}}$ we conclude that $\Pi_x$ is  strongly $\Uset_\xi$-directed.
%But then \eqref{Pi-B-cons}
% implies that for any up-right path $x_{m,n}$ starting from $x$ 
% 	\[\Pi_x(X_{m,n}=x_{m,n})=e^{\sum_{i=m}^{n-1}\w_{x_i}-B^\xi(x,x_n,\what)}=\Pi_x^{\xi,\what}(X_{m,n}=x_{m,n}).\]
	
For the weak convergence claim, apply Theorem \ref{thm:Buselim1} to get that for any up-right path $x_{m,k}$ out of $x$
	\begin{align}\label{Q-cv}
	\begin{split}
	Q^\w_{x,x_n}(X_{m,k}=x_{m,k})
	&=\frac{e^{\sum_{i=m}^{k-1}\w_{x_i}}Z_{x_k,x_n}}{Z_{x,x_n}}\\
	&\mathop{\longrightarrow}_{n\to\infty} e^{\sum_{i=m}^{k-1}\w_{x_i}-B^\xi(x,x_k,\what)}=\Pi_x^{\xi,\what}(X_{m,k}=x_{m,k}).
	\end{split}
	\end{align}
	
The covariance and consistency claims follow from the covariance of $B^\xi=B^{\xi+}$ and the fact that $\Pi_x^{\xi,\what}$ all use the same 
transition probabilities $\pi^{\xi,\what}$, regardless of the starting point $x$, as noted right before the statement of Theorem \ref{th:PiB-lln}.
The theorem is proved.
\end{proof}

\begin{proof}[Proof of Theorem \ref{th:main3}]
Define \label{OmBusPi}$\OmBusPi$ out of $\OmBusPihat$, similarly to \eqref{eq:Omexist}.
Then $\P(\OmBusPi)=1$ and for each $\w\in\OmBusPi$ there exists $\what\in\OmBusPihat$ with $\pi_\Omega(\what)=\w$.
The claim now follows directly from Theorem \ref{th:DLR=B}.
\end{proof}

\begin{proof}[Proof of Theorem \ref{thm:main:Bus}]
Let \label{OmBushat}$\OmBushat=\OmBusallhat\cap\OmBusPihat$.
Define \label{OmBus}$\OmBus$ out of $\OmBushat$, similarly to \eqref{eq:Omexist}.
Then $\P(\OmBus)=1$ and for each $\w\in\OmBus$ there exists $\what\in\OmBushat$ with $\pi_\Omega(\what)=\w$.
When $\ximin,\xi,\ximax\in\Diff$, $\nabla\fe(\ximin\pm)=\nabla\fe(\ximax\pm)=\nabla\fe(\xi\pm)=\nabla\fe(\xi)$. Since $\what\in\OmBusallhat\subset\Omconthat{\ximin}\cap\Omconthat{\ximax}$, $B^{\ximin-}=B^{\xi-}=B^{\xi+}=B^{\ximax+}$. Theorem \ref{thm:Buselim1} then implies the limit in \eqref{Bus-lim-nice} exists and equals the value of the cocycle 
$B^\xi(x,y,\what)$. Then \eqref{mono-nice} follows from \eqref{mono}. 

Take $x \in \bbV_m$ and consider $n>m$. By \cite[Theorem 4.1]{Ras-Sep-14} the distributions of $X_n/n$ under $Q_{x,(n)}^{\w,h}$ satisfy a large deviation principle with rate function 
	\begin{align*}
	J(\zeta)&=- h\cdot\zeta -\fe(\zeta)+\fepl(h),\quad\zeta\in\Uset.
	\end{align*}
By duality, $J(\cdot)$ vanishes exactly on $\Uset_h = [\ximin,\ximax]$. Borel-Cantelli and strict positivity of $J$ off of $[\ximin,\ximax]$ imply that  $\nu^\w=\otimes_{n>m} Q_{x,(n)}^{\w,h}$ is strongly $[\ximin,\ximax]$-directed.  For $i\in\{1,2\}$ use the weak convergence in Theorem \ref{th:main3} to find
	\begin{align*}
	\frac{Z_{x+e_i,(n)}^h}{Z_{x,(n)}^h}
	&=e^{-\w_x-h \cdot e_i} E^{\nu^{\omega}}[Q_{x,X_{n}}^\w(x+e_i)]\\
	&\mathop{\longrightarrow}_{n\to\infty}e^{-\w_x-h\cdot e_i}\Pi_x^{\xi,\w}(x+e_i)=e^{-B^\xi(x,x+e_i;\w)-h\cdot e_i}.
	\end{align*}
\eqref{Bus-lim-p2l-nice} follows from the above, telescoping products, and \eqref{coc-prop}.
\end{proof}

\begin{proof}[Proof of Corollary \ref{cor:main3+Bus}]
The claims follow from the observation that the limit in \eqref{Bus-lim-nice} is exactly the cocycle $B^\xi$.
\end{proof}

\begin{lemma}\label{lm:mgale1}
Fix $x,y\in\Z^2$, $\w\in\Omega$, and $\Pi_x\in\DLR_x^\w$. Then 
$Z_{y,X_n}/Z_{x,X_n}$ is a $\Pi_x$-backward martingale relative to the filtration $\cX_{[n,\infty)}$. %=\sigma(X_{n,\infty})$.
\end{lemma}

\begin{proof}
Fix $N>n$ and an up-right path $x_{n+1,N}$ with $\Pi_x(x_{n+1,N})>0$. Abbreviate $A=\{x_{n+1}-e_1,x_{n+1}-e_2\}$. Write
	\begin{align*}
	&E^{\Pi_x}\Bigl[\frac{Z_{y,X_n}}{Z_{x,X_n}}\,\Big|\,X_{n+1,N}=x_{n+1,N}\Bigr]
	=\sum_{x_n\in A} \frac{Z_{y,x_n}}{Z_{x,x_n}}\cdot\frac{\Pi_x(x_{n,N})}{\Pi_x(x_{n+1,N})}\\
	%&= \sum_{x_n\in A}\frac{Z_{y,x_n}}{Z_{x,x_n}}\cdot\frac{\Pi_x(x_N)Z_{x,x_n}e^{\sum_{j=n}^{N-1}\w_{x_j}}/Z_{x,x_N}}{\Pi_x(x_N)Z_{x,x_{n+1}}e^{\sum_{j=n+1}^{N-1}\w_{x_j}}/Z_{x,x_N}}\\
	&= \sum_{x_n\in A}\frac{Z_{y,x_n}e^{\w_{x_n}}}{Z_{x,x_{n+1}}}=\frac{Z_{y,x_{n+1}}}{Z_{x,x_{n+1}}}\,.\qedhere
	\end{align*}
\end{proof}

\begin{theorem}\label{th:B-lim}
Fix $\w\in\Omega$ and $x\in\V_m$, $m\in\Z$. 
Let $\Pi_x$ be a nondegenerate extreme point of $\DLR_x^\w$. 
Then for all $u,v\ge x$, 
	\begin{align}\label{cv-as}
	\frac{Z_{v,X_n}}{Z_{u,X_n}}\mathop{\longrightarrow}_{n\to\infty} e^{-B^{\Pi_x}(u,v)}\quad\text{$\Pi_x$-almost surely.}
	\end{align}
\end{theorem}

\begin{proof}
By the backward-martingale convergence theorem \cite[Theorem 5.6.1]{Dur-10}
%, bound \eqref{Z-bd}, and bounded convergence, we know that 
$Z_{y,X_n}/Z_{x,X_n}$ converges $\Pi_x$-almost surely and in $L^1(\Pi_x)$ 
to a limit $\eB_{x,y}=\eB_{x,y}(x_{m,\infty})$. 
{\sl A priori} $\eB_{x,y}$ is $\bigcap_n\cX_{[n,\infty)}$-measurable.
%We omit the $X_{m,\infty}$ from the notation.
Define
	\[\eB_{y,y+e_i}=\frac{\eB_{x,y+e_i}}{\eB_{x,y}},\quad\text{for $i\in\{1,2\}$ and  $y\in x+\Z^2_+$ with $\eB_{x,y}>0$}.\]
Note that %if $Z_{y,X_n}>0$ then we have
	\begin{align}\label{Z-rec}
	\frac{Z_{y+e_1,X_n}}{Z_{y,X_n}}+\frac{Z_{y+e_2,X_n}}{Z_{y,X_n}}=e^{-\w_y}.
	\end{align}
%Then \eqref{eB-rec} implies 
This implies
	\begin{align}\label{rec-aux}
	\Pi_x\Bigl\{\forall y\ge x:\eB_{x,y}=0\text{ or }\eB_{y,y+e_1}+\eB_{y,y+e_2}=e^{-\w_y}\Bigr\}=1.
	\end{align}

Next, note that %if $Z_{y+e_1+e_2,X_n}>0$, then $Z_{y,X_n}>0$, $Z_{y+e_1,X_n}>0$, and $Z_{y+e_2,X_n}>0$ and we have
	\[\frac{Z_{y+e_1,X_n}}{Z_{y,X_n}}\cdot\frac{Z_{y+e_1+e_2,X_n}}{Z_{y+e_1,X_n}}=\frac{Z_{y+e_2,X_n}}{Z_{y,X_n}}\cdot\frac{Z_{y+e_1+e_2,X_n}}{Z_{y+e_2,X_n}}.\]
Thus, %similarly to the above, this property transfers with the weak limit and we get
	\begin{align}\label{coc-aux}
	\begin{split}
	\Pi_x\Bigl\{\forall y\ge x:\,&\eB_{x,y+e_1+e_2}=0\text{ or }\\
	&\eB_{y,y+e_1}\eB_{y+e_1,y+e_1+e_2}
	=\eB_{y,y+e_2}\eB_{y+e_2,y+e_1+e_2}\Bigr\}=1.
	\end{split}
	\end{align}
%Let $\EB_1$ and $\EB_2$ denote  the two events in \eqref{rec-aux} and \eqref{coc-aux}, respectively. Let $\EB_0=\EB_1\cap\EB_2$. Then $\Pi_x(\EB_i)=1$, $i\in\{0,1,2\}$.
	
On the event in \eqref{rec-aux}
	\[\pi_{y,y+e_i}=\begin{cases}\eB_{y,y+e_i} \,e^{\w_y}&i\in\{1,2\}\text{ and }y\in x+\Z^2_+\text{ such that }\eB_{x,y}>0\\ 1/2&i\in\{1,2\}\text{ and }y\in x+\Z^2_+\text{ such that }\eB_{x,y}=0\end{cases}\]
define transition probabilities. Let $\Pi^\eB_x$ be the distribution of the Markov chain $X_{m,\infty}$ starting at $X_m=x$ and using these transition probabilities.

Note that $\eB_{x,x}=1$ and %\eqref{rec-aux} implies that 
if $\eB_{x,y}>0$ and $\pi_{y,y+e_i}>0$, then $\eB_{y,y+e_i}>0$ and $\eB_{x,y+e_i}=\eB_{x,y}\eB_{y,y+e_i}>0$.
This means that the Markov chain stays $\Pi^\eB_x$-almost surely within the set $\{y\ge x:\eB_{x,y}>0\}$.

On the intersection of the two events in \eqref{rec-aux} and \eqref{coc-aux}, if $x_{m,k}$ is an admissible path starting at $x$ and $\Pi^\eB_x(x_{m,k})>0$, then the above paragraph says $\eB_{x,x_i}>0$ for each $i\in\{m,\dotsc,k\}$  and then
	\begin{align}\label{PieB-Z}
	\Pi^\eB_x(x_{m,k})=\prod_{i=m}^{k-1}\pi_{x_i,x_{i+1}}=\prod_{i=m}^{k-1}\eB_{x_i,x_{i+1}}\,e^{\w_{x_i}}=\eB_{x,x_k}\,e^{\sum_{i=m}^{k-1}\w_{x_i}}.
	\end{align}
%For the last equality we used the fact that $\eB\in\EB_2$.
Adding over all admissible paths from $x$ to $y\in\V_k$ gives
	\begin{align}\label{PieB-Z2}
	\Pi^\eB_x(y)=\eB_{x,y}\,Z_{x,y}.
	\end{align}
Putting the two displays together gives
	\[\Pi^\eB_x(x_{m,k})=\Pi^\eB_x(y)\cdot\frac{e^{\sum_{i=m}^{k-1}\w_{x_i}}}{Z_{x,y}}=\Pi^\eB_x(y)\,Q^\w_{x,y}(x_{m,k}).\]
In other words, $\Pi^\eB_x\in\DLR_x^\w$, $\Pi_x$-almost surely.
The $L^1$-convergence implies 
	\[E^{\Pi_x}[\eB_{x,y}]=\lim_{n\to\infty}E^{\Pi_x}\Bigl[\frac{Z_{y,X_n}}{Z_{x,X_n}}\Bigr]=e^{-B^{\Pi_x}(x,y)},\]
where we used \eqref{B-Z} for the last equality (since $\Pi_x$ is assumed to be nondegenerate).  The above, \eqref{PieB-Z}, and \eqref{Pi-B-cons}  give
	\[E^{\Pi_x}[\Pi^\eB_x(x_{m,k})]=e^{\sum_{i=m}^{k-1}\w_{x_i}-B^{\Pi_x}(x,y)}=\Pi_x(x_{m,k}).\]
%for any admissible path $x_{0,m}$ starting at $x$.
In other words, $\Pi_x=\int \Pi^{\eB(x_{m,\infty})}_x\,\Pi_x(dx_{m,\infty})$. Since $\Pi_x$ was assumed to be an extreme point in $\DLR_x^\w$, we conclude that 
$\Pi_x(\Pi^\eB_x=\Pi_x)=1$. Since $\Pi^\eB_x$ determines $\eB$, this says that $\eB_{x,y}(x_{m,\infty})=e^{-B^{\Pi_x}(x,y)}$ for all $y\ge x$ and $\Pi_x$-almost every $x_{m,\infty}$. 
%In particular, $\eB(x,y)>0$ and Since $e^{-B^{\Pi_x}(x,y)}>0$ for all $y\ge x$, we have that 
%	\[\Pi_x(X_m\ge y)=\Pi_x(Z_{y,X_m}>0)\mathop{\longrightarrow}_{m\to\infty}1.\] 
%Since $X_n\ge X_m$ for all $n\ge m$  the events $\{X_m\ge y\}=\{\forall n\ge m:X_n\ge y\}$ are non-decreasing and \eqref{Xn>y} follows.
%
%Now that we know that $\Pi_x$-almost surely, $Z_{u,X_n}>0$ for large enough $n$, 
Now \eqref{cv-as} follows from  writing 
	\[\frac{Z_{v,X_n}}{Z_{u,X_n}}=\frac{Z_{v,X_n}/Z_{x,X_n}}{Z_{u,X_n}/Z_{x,X_n}},\]
taking $n\to\infty$ and applying the cocycle property of $B^{\Pi_x}$.
\end{proof}

We now turn to the proof of Theorem \ref{th:main2}. The full proof requires handling some technical issues, so we begin with a brief sketch of the main idea in the case where $\fe$ is strictly concave to give a sense of how the argument works. By \eqref{cv-as}, $\log Z_{y,X_n}-\log Z_{y+e_1,X_n}$ converges $\Pi_x$-almost surely to $B^{\Pi_x}(y,y+e_1)$.
On the other hand, \eqref{Bus-lim-nice}  implies that for nice directions $\xi$, 
$\log Z_{y,\fl{n\xi}}-\log Z_{y+e_2,\fl{n\xi}}$ converges $\P$-almost surely to $B^\xi(y,y+e_1)$.
This, and the monotonicity from \eqref{comparison} imply that if $X_n\cdot e_1>n\xi\cdot e_1$ happens infinitely often, then 
$B^{\Pi_x}(y,y+e_1)\le B^\xi(y,y+e_1)$ for all $y\ge x$. But then coupling $\Pi_x^{\xi,\w}$ and $\Pi_x$ pathwise, as described in Section \ref{sub:coupling}, implies that almost surely the $\Pi_x$-path must stay to the right of the $\Pi_x^{\xi,\w}$-path. 
A similar argument holds if $X_n\cdot e_1<n\xi\cdot e_1$ happens infinitely often. In short, this argument shows that if a subsequential limit point of $X_n$ goes to the right of a nice direction $\zeta$ with positive probability, then every subsequential limit point must stay to the right of $\zeta$. Similarly, if any subsequential limit goes to the left of a nice direction $\eta$, then every subsequential limit point must stay to the left of $\eta$. These two statements are only consistent if the path satisfies the strong law of large numbers for some direction $\xi \in \ri \Uset$. The technicalities in the proof arise because we do not assume strict concavity.

\begin{proof}[Proof of Theorem \ref{th:main2}]
%Let $\Ddense$ be a countable dense subset of $\Diff$.
%\normalmarginpar\fixtext{$\Ommonohat$ is the event where $B^\xi_\pm$ are monotone and $\OmBusPihatzeta$ is the event on which liminf and limsup bounds hold for the given $\zeta$}
Let \label{Omdirhat}$\Omdirhat=\Omexisthat\cap\OmBusallhat\cap\Omeihat$ and
%\note{even though $\Ddense\subset\Diff$, we do not necessarily have Busemann limits. for that we also need $\zetamin,\zetamax\in\Diff$}
similarly to \eqref{eq:Omexist} let 
	\label{Omdir}$\Omdir=\Omnondeg\cap\Omreg\cap\bigl\{\w\in\Omega:\mu_\w(\Omdirhat)=1\bigr\}.$
Then $\P(\Omdir)=1$. Fix $\w\in\Omdir$. There exists $\what\in\Omdirhat$ such that $\pi_\Omega(\what)=\w$.

%Since $\w\in\Omnondeg$, if $\Pi_x$ is a degenerate extreme solution,  it must equal $\Pi_x^{e_i}$ for $i=1$ or $i=2$. In this case, $\Pi_x$ is strongly $\Uset_{e_i}$-directed.
%In the rest of the proof we assume $\Pi_x$ is a nondegenerate extreme solution.

Take $\zeta\in\ri\Uset$. For any $y\ge x$ we have $n\zeta\ge y$ when $n$ is large enough.
Then when  $X_n\cdot e_1>\fl{n\zeta\cdot e_1}$  
inequality \eqref{comparison} implies
	\[\frac{Z_{y+e_1,X_n}}{Z_{y,X_n}}\ge \frac{Z_{y+e_1,\fl{n\zeta}}}{Z_{y,\fl{n\zeta}}}
	\quad\text{and}\quad \frac{Z_{y+e_2,X_n}}{Z_{y,X_n}}\le \frac{Z_{y+e_2,\fl{n\zeta}}}{Z_{y,\fl{n\zeta}}}\,.\]
Since $\what\in\OmBusallhat$
	\[\varliminf_{n\to\infty}\frac{Z_{y+e_1,\fl{n\zeta}}}{Z_{y,\fl{n\zeta}}}\ge e^{-B^{\zetamin-}(y,y+e_1,\what)}
	\quad\text{and}\quad \varlimsup_{n\to\infty}\frac{Z_{y+e_2,\fl{n\zeta}}}{Z_{y,\fl{n\zeta}}}\le e^{-B^{\zetamin-}(y,y+e_2,\what)}.\]
%Also, \eqref{cross} implies $\Pi_x(X_n\ge y)\to1$, for all $y\in x+\Z^2$. 
Putting these facts together with \eqref{cross} we get for $\e>0$
	\begin{align*}
	&\varlimsup_{n\to\infty}\Pi_x\bigl\{X_n\cdot e_1>\fl{n\zeta\cdot e_1}\bigr\}
	\le \varlimsup_{n\to\infty}\Pi_x\Bigl\{\frac{Z_{y+e_2,X_n}}{Z_{y,X_n}}\le \frac{Z_{y+e_2,\fl{n\zeta}}}{Z_{y,\fl{n\zeta}}},X_n\ge y\Bigr\}\\
	&\qquad\qquad\le \varlimsup_{n\to\infty}\Pi_x\Bigl\{\frac{Z_{y+e_2,X_n}}{Z_{y,X_n}}<e^{-B^{\zetamin-}(y,y+e_2,\what)}+\e,X_n\ge y\Bigr\}.
%	&\le \one\bigl\{e^{-B^{\Pi_x}(y,y+e_2)}<e^{-B^{\zetamin-}(y,y+e_2,\what)}+\e\Bigr\}.
	\end{align*}
If the limsup on the left is positive then using \eqref{cv-as} implies $e^{-B^{\Pi_x}(y,y+e_2)}\le e^{-B^{\zetamin-}(y,y+e_2,\what)}+\e$. The case of $e_1$ is similar. Taking $\e\to0$ we get
%Taking $\e\to0$ we get that $B^{\Pi_x}(y,y+e_2)\ge B^{\zetamin-}(y,y+e_2,\what)$. The case of $e_1$ increments is dealt with in the same way and we get that
	 \begin{align}\label{B<Bzeta}
	 \begin{split}
	 &B^{\Pi_x}(y,y+e_1)\le B^{\zetamin-}(y,y+e_1,\what)\quad\text{and}\\
	 &B^{\Pi_x}(y,y+e_2)\ge B^{\zetamin-}(y,y+e_2,\what)
	 \end{split}
	 \end{align}
 for each $y\in x+\Z^2_+$ and $\zeta\in\ri\Uset$ such that
	\begin{align}\label{dir-aux1}
	\varlimsup_{n\to\infty}\Pi_x\bigl\{X_n\cdot e_1>n\zeta\cdot e_1\bigr\}>0.
	\end{align}
%Similarly, we can show that 
%	 \begin{align}\label{B>Bzeta}
%	 B^{\Pi_x}(y,y+e_1)\ge B^\zetamax_+(y,y+e_1,\what)\quad\text{and}\quad B^{\Pi_x}(y,y+e_2)\le B^\zetamax_+(y,y+e_2,\what)
%	 \end{align}
% for each $y\in x+\Z^2_+$ and $\zeta\in\Ddense$ such that
%	\begin{align}\label{dir-aux2}
%	\varlimsup_{n\to\infty}\Pi_x\bigl\{X_n\cdot e_1<n\zeta\cdot e_1\bigr\}>0.
%	\end{align}

%Recall formula \eqref{pi-B} for the transition probabilities of $\Pi_x$ and the coupling of solutions $\Pi_x^{\zeta\pm,\what}$ via paths  $X^{x,\zeta\pm,\what}_{m,\infty}$. 
%Recall also that $x\in\V_m$. 
Couple $\{\Pi_x,\Pi_x^{\zeta\pm,\what}:\zeta\in\ri\Uset\}$ as described in Section \ref{sub:coupling} and denote the coupled paths by
$\XDLR_{m,\infty}$ (distribution $\Pi_x$) and $X^{x,\zeta\pm,\what}_{m,\infty}$ (distribution $\Pi_x^{\zeta\pm,\what}$).
%by $\XDLR_m=x$ and 
%	\[\XDLR_{k+1}=\XDLR_k+\begin{cases}e_1&\text{if }\unif(\XDLR_k)<\pi^x(\XDLR_k,\XDLR_k+e_1),\\e_2&\text{otherwise.}\end{cases}\]
%%Here, $\pi^x$ is the transition probability for $\Pi_x$, as given by \eqref{pi-Pi}.
%Then the distribution of $\XDLR$, induced by $\bfP$, is exactly $\Pi_x$.

We have already seen that paths $X^{x,\zeta\pm,\what}$ are monotone in $\zeta$. Similarly, \eqref{B<Bzeta} implies that for  $\zeta\in\ri\Uset$ satisfying \eqref{dir-aux1}, we have
	\[\XDLR_k\cdot e_1\ge X_k^{x,\zetamin-,\what}\cdot e_1\quad\text{for all $k\in\Z_+$}.\]
%and \eqref{B>Bzeta} implies that for each $\zeta\in\Ddense$ satisfying \eqref{dir-aux2}
%	\[\XDLR_k\cdot e_1\le X_k^{x,\zetamax,+,\what}\cdot e_1\quad\text{for all $k\in\Z_+$.}\]
Since the distribution of $X_k^{x,\zetamin-,\what}$ %and $X_k^{x,\zetamax,+,\what}$ are 
is $\Pi_x^{\zetamin-,\what}$ %and $\Pi_x^{\zetamax,+,\what}$ 
and %are 
is strongly directed into $\Uset_{\zetamin-}$ (because $\what\in\Omexisthat$) %and $\Uset_{\zetamax+}$, respectively, 
we see that for $\zeta\in\ri\Uset$ satisfying \eqref{dir-aux1}
	\begin{align}\label{zetaminmin}
	\Pi_x\Bigl\{\varliminf_{n\to\infty} n^{-1}X_n\cdot e_1\ge\zetaminmin\cdot e_1\Bigr\}=1.
	\end{align}
%and
%	\[\Pi_x\Bigl\{\varlimsup_{n\to\infty} n^{-1}\XDLR_k\cdot e_1\le\zetamaxmax\cdot e_1\Bigr\}=1\]
%for $\zeta\in\Ddense$ satisfying \eqref{dir-aux2}. 
Here, $\zetaminmin\cdot e_1=\inf\{\zeta\cdot e_1:\zeta\in\Uset_{\zetamin-}\}=1-\zetaminmin\cdot e_2$. % with a similar definition for $\zetamaxmax$.
Let $\xi'\in\Uset$ be such that 
	\[\xi'\cdot e_1=\sup\{\zeta\cdot e_1:\zeta\in\ri\Uset\text{ and \eqref{dir-aux1} holds for $\zeta$}\}.\]
If the above set is empty, then we set $\xi'=e_2$. 
%$X_n/n\to e_2$ in $\Pi_x$-probability and the case is closed.
%Assume $\xi'\ne e_2$.  
Let $\xi_1=\ximin'$. If $\xi'=e_2$, then $\xi_1=\ximin_1=e_2$ as well and we trivially have
	\begin{align}\label{liminf}
	\Pi_x\Bigl\{\varliminf_{n\to\infty} n^{-1}X_n\cdot e_1\ge\ximin_1\cdot e_1\Bigr\}=1.
	\end{align}

Assume $\xi'\ne e_2$ and take $\zeta\in\ri\Uset$ with $\zeta\cdot e_1<\xi'\cdot e_1$. 
Observe that we can take $\zetaminmin$ arbitrarily close to $\ximin_1$. Indeed, if $\xi_1\cdot e_1<\xi'\cdot e_1$, then take $\xi_1\cdot e_1<\zeta\cdot e_1<\xi'\cdot e_1$
to get $\zetamin=\ximin'=\xi_1$ and $\zetaminmin=\ximin_1$. If instead $\xi_1=\xi'$, then also $\ximin_1=\ximin'=\xi_1$. Now, as
$\zeta\to\xi_1$, $\nabla\fe(\zeta\pm)$ approach but never equal $\nabla\fe(\xi_1-)$ because there is no linear segment of $\fe$ adjacent to $\xi_1$ on the left. This forces $\zetamin$ and $\zetaminmin$
to converge to $\xi_1$.

Fix $\e>0$ and take $\zeta\in\ri\Uset$ with $\zeta\cdot e_1<\xi'\cdot e_1$ and $\zetaminmin\cdot e_1>\ximin_1\cdot e_1-\e$. Then \eqref{dir-aux1} holds and therefore \eqref{zetaminmin} holds too and we have
	\begin{align*}
%	\varliminf_{n\to\infty}\Pi_x\{X_n\cdot e_1>n\ximin\cdot e_1-n\e\}
%	&\ge E^{\Pi_x}\Bigl[\varliminf_{n\to\infty}\one\Bigl\{n^{-1}X_n\cdot e_1>\ximin\cdot e_1-\e\Bigr\}\Bigr]\\
%	&=
	\Pi_x\Bigl\{\varliminf_{n\to\infty} n^{-1}X_n\cdot e_1>\ximin_1\cdot e_1-\e\Bigr\}
	\ge\Pi_x\Bigl\{\varliminf_{n\to\infty} n^{-1}X_n\cdot e_1\ge\zetaminmin\cdot e_1\Bigr\}
	=1.
	\end{align*}
Take $\e\to0$ to get \eqref{liminf} when $\xi'\ne e_2$.

%Then $\xi'\in\Uset_{\xi+}$ and thus $\xi'\cdot e_1\le\ximax\cdot e_1$. Fix $\e>0$.
%Take $\zeta\in\Ddense$ such that $\xi'\cdot e_1<\zeta\cdot e_1<\ximax\cdot e_1+\e$. Then the definition of $\xi'$ implies that \eqref{dir-aux1} fails and
%	\begin{align}\label{foo}
%	\lim_{n\to\infty}\Pi_x\{n^{-1}X_n\cdot e_1\ge\ximax\cdot e_1+\e\}\le \lim_{n\to\infty}\Pi_x\{X_n\cdot e_1>n\zeta\cdot e_1\}=0.
%	\end{align}
A symmetric argument (e.g.\ exchanging the roles of $e_1$ and $e_2$) gives
	\begin{align}\label{limsup}
	\Pi_x\Bigl\{\varlimsup_{n\to\infty} n^{-1}X_n\cdot e_1\le\ximax_2\cdot e_1\Bigr\}=1,
	\end{align}
where $\xi_2=\ximax''$ and $\xi''\in\Uset$ is such that 
	\[\xi''\cdot e_1=\inf\Bigl\{\zeta\cdot e_1:\zeta\in\ri\Uset\text{ and $\varlimsup_{n\to\infty}\Pi_x\bigl\{X_n\cdot e_1<n\zeta\cdot e_1\bigr\}>0$}\Bigr\},\]
with $\xi''=e_1$ if the set is empty.

\eqref{liminf} and \eqref{limsup} imply that $\ximin_1\cdot e_1\le\xi''\cdot e_1$ and $\xi'\cdot e_1\le\ximax_2\cdot e_1$. 
%case \ximax_2=e_1 is trivial
For example if $\zeta\in\ri\Uset$ is such that $\zeta\cdot e_1>\ximax_2\cdot e_1$ then
Fatou's lemma gives
	\begin{align*}
	1
	&=\Pi_x\Bigl\{\varlimsup_{n\to\infty} n^{-1}X_n\cdot e_1\le\ximax_2\cdot e_1\Bigr\}
	\le\Pi_x\Bigl\{\varlimsup_{n\to\infty} n^{-1}X_n\cdot e_1<\zeta\cdot e_1\Bigr\}\\
	&\le E^{\Pi_x}\Bigl[\varliminf_{n\to\infty}\one\bigl\{X_n\cdot e_1<n\zeta\cdot e_1\bigr\}\Bigr]
	\le \varliminf_{n\to\infty} \Pi_x\bigl\{X_n\cdot e_1<n\zeta\cdot e_1\bigr\},
	\end{align*}
and then $\zeta\cdot e_1\ge\xi'\cdot e_1$.  

Also, $\xi'\cdot e_1\ge\xi''\cdot e_1$. To see this take $\zeta,\zeta'\in\ri\Uset$ with $\zeta\cdot e_1<\zeta'\cdot e_1<\xi''\cdot e_1$. Then $\Pi_x(X_n\cdot e_1\ge n\zeta'\cdot e_1)\to1$ and hence
\eqref{dir-aux1} holds and $\zeta\cdot e_1\le \xi'\cdot e_1$. Take $\zeta\to\xi''$. We now consider three cases.  

Case \eqref{case-a}: If $\xi'=\xi_1$, then $\xi'=\xi_1=\ximin_1$,  forcing $\xi''=\xi'=\xi_1$.  Let $\xi=\xi'$.
Weak $\{\xi\}$-directedness holds by the definitions of $\xi'$ and $\xi''$, since they equal $\xi$. 
Note that $\ximax=\xi_2$ and $\Uset_{\ximax}=[\xi_1,\ximax_2]=[\ximin_1,\ximax_2]$. Then strong directedness
into $\Uset_{\ximax}$ follows from \eqref{liminf} and \eqref{limsup}. The case $\xi''=\xi_2$ is similar.

Case \eqref{case-b}: Assume $\xi'\ne\xi_1$ and $\xi''\ne\xi_2$ but $\ximin_1\cdot e_1\le\xi''\cdot e_1\le\xi_1\cdot e_1$. Then set $\xi=\xi_1$. We have $\ximax''=\xi$ and thus $\ximax_2=\ximax$. 
%\ximax'=\xi because \xi=\xi_1 is the leftmost point of a linear segment and \ximax' is inside \Uset_{\xi-}
We also have $\ximin_1=\ximin$ and again 
strong directedness into $\Uset_\xi$ follows from \eqref{liminf} and \eqref{limsup}. The case $\xi_2\cdot e_1\le\xi'\cdot e_1\le\ximax_2\cdot e_1$ is similar. 

Case \eqref{case-c}: In the remaining case, $\xi_1\cdot e_1<\xi''\cdot e_1\le\xi'\cdot e_1<\xi_2\cdot e_1$ we have $[\xi_1,\xi_2]=\Uset_{\xi'}=\Uset_{\xi''}$. In this case, $\fe$ is linear on $[\xi_1,\xi_2]$ and therefore
$\xi',\xi''\in\Diff$.
Let $\xi=\xi'$. The definitions of $\xi''$ and $\xi'$ give weak directedness into $[\xi'',\xi']\subset[\xi_1,\xi_2]=\Uset_\xi$. Strong directedness into $\Uset_{\ximin}\cup\Uset_{\ximax}=[\ximinmin,\ximaxmax]=[\ximin_1,\ximin_2]$ follows 
from \eqref{liminf} and \eqref{limsup}.

To finish, note that in all three cases $\xi\in\ri\Uset$. Indeed, strong directedness into $\Uset_{e_1}$ would imply \eqref{dir-aux1} and thus \eqref{B<Bzeta} hold for all $\zeta\in\ri\Uset$. Then Lemma \ref{lm:B-ei} would imply $B^{\Pi_x}(y,y+e_2)=\infty$, contradicting nondegeneracy. Strong directedness into $\Uset_{e_2}$ is argued similarly.
\end{proof}

For the rest of the section we assume that \eqref{La-reg} holds. Then, in Theorem \ref{thm:cocyexist}, 
we can ask that $1\in\sB_0$ and take $\sH_0^1$ to be
$\{-\nabla\fe(\xi):\xi\in\Ddense\}$, where $\Ddense$ is the countable dense subset of $\ri\Uset$ from the paragraph following \eqref{La-reg}.
Theorem \ref{thm:Buselim1} then implies that for $\xi\in\Ddense$ and 
$\what\in\Ommonohat\cap\OmBusallhat\cap\Omconthat{\ximin}\cap\Omconthat{\ximax}$, $B^{\ximin-}=B^{\ximax+}=B^\xi$ is a function of $\{\w_x(\what):x\in\Z^2\}$.
This and \eqref{Busemann-limits} imply that the whole process $\{B^{h\pm}:h\in\sB\}$ is 
measurable with respect to $\mfS = \sigma(\w_x : x \in \bbZ^2)\subset\sF$.
In other words we do not need the extended space $\Omhat$. For the rest of the section we write $\w$ instead of $\what$ and more generally drop the hats from our notation.%\smallskip

Recall the definition of the countable random set $\Usetnonuniq\subset\ri\Uset$ in  \eqref{Usetnonuniq}.

\begin{lemma}\label{lm:B-disc}
Assume \eqref{La-reg}. Fix $x\in\Z^2$. The following hold.
\begin{enumerate}[label={\rm(\alph*)}, ref={\rm\alph*}]
\item\label{lm:B-disc:a} For any $\eta,\zeta\in\ri\Uset$,
$\{\w\in\Ommono:A\cap\Usetnonuniq\neq\varnothing\}$ is measurable, where $A$ is any of the four intervals $[\eta,\zeta]$, $[\eta,\zeta[$, $]\eta,\zeta]$, or $]\eta,\zeta[$.
Also, $\{\w\in\Ommono:\Diff\cap\Usetnonuniq\neq\varnothing\}$ and $\{\w\in\Ommono:\abs{\Diff\cap\Usetnonuniq}=\infty\}$ are measurable as are  %HERE
 %$\{\w\in\Ommono:\Usetnonuniq\text{ is dense in }[\eta,\zeta]\}$.
 $\{\w\in\Ommono:\eta\text{ is a right-accumulation point in }\Usetnonuniq\}$
and  $\{\w\in\Ommono:\zeta\text{ is a left-accumulation point in }\Usetnonuniq\}$.
\item\label{lm:B-disc:b} For any $\eta,\zeta\in\Uset$, $\P([\eta,\zeta]\cap\Usetnonuniqz\neq\varnothing)\in\{0,1\}$
and $\P([\eta,\zeta]\cap\Usetnonuniqz\neq\varnothing)=1$ if and only if
	\begin{align}\label{B-discont}
	\P\bigl\{\w\in\Ommono:\exists\xi\in[\eta,\zeta]\cap\ri\Uset:B^{\xi+}(0,e_1,\w)\ne B^{\xi-}(0,e_1,\w)\bigr\}>0.
	\end{align}
%\item\label{lm:B-disc:c} $\Usetnonuniq$ is supported outside the linear segments of $\fe$:
%For any $\xi\in\Diff$, $\P\{[\ximin,\ximax]\cap\Usetnonuniq\neq\varnothing\}=0$. Its atoms are precisely the points of non-differentiability of $\fe$:
%If $\xi\in\Diff$, then $\P\{\xi\in\Usetnonuniq\}=0$ and if $\xi\in\ri\Uset\smallsetminus\Diff$, then $\P\{\xi\in\Usetnonuniq\}=1$.
%\item\label{lm:B-disc:d} For any $\eta,\zeta\in\Uset$ such that $[\eta,\zeta]\cap\ri\Uset\subset\Diff$ and $\P\{[\eta,\zeta]\cap\,\Usetnonuniq\neq\varnothing\}>0$,
%we have \[\P\bigl\{[\eta,\zeta]\cap\,\Usetnonuniq\text{ is dense in }[\eta,\zeta]\bigr\}=1.\]
\end{enumerate}
\end{lemma}

\begin{proof}
Let $\Ddense$ be a dense set of points in $\Diff$.
%or could use $\Usetnice$, defined below \eqref{La-reg}.
For $\eta,\zeta\in\ri\Uset$ the event $\{\w\in\Ommono:]\eta,\zeta[\,\cap\,\Usetnonuniq\neq\varnothing\}$ can be rewritten as
	\begin{align*}
	&\Bigl\{\w\in\Ommono:\exists y\in x+\Z_+^2\ \exists i\in\{1,2\}\ \exists\ell\in\N\ \forall k\in\N\ \exists \xi_0,\xi_1\in \Ddense\cap]\eta,\zeta[:\\
	&\qquad\qquad\qquad\qquad\abs{\xi_1-\xi_0}_1<1/k,\ \abs{B^{\xi_1}(y,y+e_i)-B^{\xi_0}(y,y+e_i)}_1>1/\ell\Bigr\}.
	\end{align*}
%if \w has a point of discontinuity at \xi\in ]\eta,\zeta[, then it is clearly in the above event (just take \x_0 and \xi_1 surrounding this \xi)
% if on the other hand \w is in the above event, then we can go along a subsequence in k to have \xi_0 and \xi_1 converge to some \xi.
%it must be that \xi_0 goes up to \xi and \xi_1 goes down to \xi. otherwise, if e.g. \xi_0 went down to \xi then so does \xi_1 and
%then intervals [\xi_0,\xi_1] can be taken to be disjoint and we'd have infinitely many jumps of size 1/\ell.
%now we see that \xi must be inside ]\eta,\zeta[. but also that it is a jump point.
It is therefore measurable. The other cases of the set $A$ can be obtained as decreasing intersections of the one above and are hence measurable too.

Recall that $\nabla \fe(\xi+)=\nabla\fe(\xi-)$ if and only if $\xi \in \Diff$. By Lemma \ref{lem:superdiffprop}\eqref{item-b}, \eqref{item-a}, and \eqref{item-c} this holds if and only if there exist sequences $\xi_{1}^j,\xi_{2}^j \in \Ddense$ with $\xi_{1}^j \cdot e_1 < \xi \cdot e_1 < \xi_2^j \cdot e_1$, $\lim_j \xi_{1}^j = \lim_j \xi_{2}^j = \xi$ and $\lim_j \nabla \fe(\xi_{1}^j) = \lim_j \nabla\fe(\xi_{2}^j)$. 
%if such sequences exist, then \partial\fe(\xi)$ is trapped between $\nabla\fe(\xi_1^j)$ and
%\nabla\fe(\xi_2^j)$, which converge to the same thing. hence \partial\fe(\xi) is a singleton
%and $\xi\in\Diff$.    conversely, if \xi\in\Diff$ then \partial\fe(\xi) is a signleton.
%taking any sequences \xi_1^j and \xi_2^j and going to a subsequence to get the 
%gradients to converge, we must get something in \partial\fe(\xi). hence the limits
%exist and are equal
The event $\{\w\in\Ommono:\abs{\Diff\cap\Usetnonuniq}\ge m\}$ can then be rewritten as
	\begin{align*}
	\Bigl\{\w\in\Ommono:\,&\exists\ell\in\N\ \exists y_j\in x+\Z_+^2\ \exists i_j\in\{1,2\},\ j\in\{1,\dotsc,m\},\ \forall k\in\N\\ 
	&\exists \xi_0^1<\xi_1^1<\xi_0^2<\xi_1^2<\cdots<\xi_0^m<\xi_1^m\in \Ddense:\ \abs{\xi_1^j-\xi_0^j}_1<1/k,\\
	 &\!\!\!\!\abs{\nabla\fe(\xi_1^j)-\nabla\fe(\xi_0^j)}_1<1/k,\ \abs{B^{\xi_1^j}(y_j,y_j+e_{i_j})-B^{\xi_2^j}(y_j,y_j+e_{i_j})}_1>1/\ell,\ j\in\{1,\dotsc,m\}\Bigr\},
	\end{align*}
so it is measurable. $m=1$ gives $\{\w\in\Ommono:\Diff\cap\Usetnonuniq\ne\varnothing\}$. Intersection over all $m$ gives $\{\w\in\Ommono:\abs{\Diff\cap\Usetnonuniq}=\infty\}$.

%HERE
%The event $\{\Usetnonuniq\text{ is dense in }[\eta,\zeta]\}$ is the intersection of the events 
%$\{\Usetnonuniq\cap[\eta',\zeta']\neq\varnothing\}$ over $\eta',\zeta'\in]\eta,\zeta[\cap\Ddense$.
The event $\{\w\in\Ommono:\eta\text{ is a right-accumulation point in }\Usetnonuniq\}$ is the intersection of the events
$\{\Usetnonuniq\cap]\eta,\eta']\neq\varnothing\}$ over $\eta'\in\Ddense$ with $\eta'\cdot e_1>\eta\cdot e_1$. Similarly for left-accumulation points.
Part \eqref{lm:B-disc:a} is proved.

Fix $\eta,\zeta\in\Uset$. Similarly to $\{\w\in\Ommono:[\eta,\zeta]\cap\Usetnonuniq\ne\varnothing\}$, the event
%	\[\cE=\bigl\{\w\in\Ommono:\exists y\in\Z^2,\exists i\in\{1,2\},\exists\xi\in\ri\Uset:B^{\xi+}(y,y+e_i,\w)\ne B^{\xi-}(y,y+e_i,\w)\bigr\}\]
	\[\cE=\bigl\{\w\in\Ommono:\exists y\in\Z^2,\exists i\in\{1,2\},\exists\xi\in[\eta,\zeta]\cap\ri\Uset:B^{\xi+}(y,y+e_i,\w)\ne B^{\xi-}(y,y+e_i,\w)\bigr\}\]
is measurable. 
It is also shift-invariant and the ergodicity of the distribution of $\{\w_x:x\in\Z^2\}$ induced by $\P$ 
implies that this event has probability either 
$0$ or $1$. It has probability $1$ if and only if 
	\begin{align}\label{B-discont2}
	\P\bigl\{\exists i\in\{1,2\},\exists\xi\in[\eta,\zeta]\cap\ri\Uset:B^{\xi+}(0,e_i)\ne B^{\xi-}(0,e_i)\bigr\}>0.
	\end{align}
%Recovery implies that $B^{\xi+}(0,e_1)=B^{\xi-}(0,e_1)$ is equivalent to $B^{\xi+}(0,e_2)=B^{\xi-}(0,e_2)$.
%Hence, 
%	\begin{align*}
%	&\P\bigl\{\exists\xi\in[\eta,\zeta]\cap\ri\Uset:B^{\xi+}(0,e_1)\ne B^{\xi-}(0,e_1)\bigr\}\\
%	&\quad\le\P\bigl\{\exists i\in\{1,2\},\exists\xi\in[\eta,\zeta]\cap\ri\Uset:B^{\xi+}(0,e_i)\ne B^{\xi-}(0,e_i)\bigr\}\\
%	&\quad\le\sum_{i=1}^2 \P\bigl\{\exists\xi\in[\eta,\zeta]\cap\ri\Uset:B^{\xi+}(0,e_i)\ne B^{\xi-}(0,e_i)\bigr\}\\
%	&\quad= 2\P\bigl\{\exists\xi\in[\eta,\zeta]\cap\ri\Uset:B^{\xi+}(0,e_1)\ne B^{\xi-}(0,e_1)\bigr\}.
%	\end{align*}
But recovery \eqref{rec-prop1} implies that $B^{\xi+}(0,e_1)\ne B^{\xi-}(0,e_1)$ is equivalent to $B^{\xi+}(0,e_2)\ne B^{\xi-}(0,e_2)$.
Therefore, \eqref{B-discont2} holds if and only if \eqref{B-discont} holds.

If $\P(\cE)=0$ then $\P\{\w:[\eta,\zeta]\cap\Usetnonuniqz\neq\varnothing\}=0$, since the latter is a smaller event.
On the other hand, if $\P(\cE)=1$ then \eqref{B-discont2} holds and ergodicity implies that with $\P$-probability one there is a positive density of sites $y$ such that 
there exist $i\in\{1,2\}$ and $\xi\in[\eta,\zeta]\cap\ri\Uset$ with $B^{\xi+}(y,y+e_i)\ne B^{\xi-}(y,y+e_i)$. In particular, there exist such sites in $\Z^2_+$ and so $[\eta,\zeta]\cap\Usetnonuniqz\neq\varnothing$. Part \eqref{lm:B-disc:b} is proved.
\end{proof}

\begin{proof}[Proof of Theorem \ref{th:main4}]
For $\xi\in(\ri\Uset)\setminus\Diff$ let $\eta=\zeta=\xi$. Then \eqref{E[h(B)]} implies \eqref{B-discont} holds.
The first claim in part \eqref{th:main4:b} follows from applying Lemma \ref{lm:B-disc},
since there are countably many directions of non-differentiability.
The second claim, about $\xi\in\Diff$, comes from the continuity in Remark \ref{B-xi}.

When  $\ximin\ne\ximax$ condition \eqref{La-reg} implies that $[\ximin,\ximax]\subset\Diff$ and hence $\nabla\fe(\zeta\pm)=\nabla\fe(\xi)$
and $B^{\zeta-}=B^{\zeta+}=B^\xi$ for all $\zeta\in[\ximin,\ximax]$. Part \eqref{th:main4:b'} now follows  from Lemma \ref{lm:B-disc}
with $\eta=\ximin$ and $\zeta=\ximax$. (There are countably many $\xi$ with $\ximin\ne\ximax$.)

%HERE
The first claim in part \eqref{th:main4:b''} is the same as the first claim in Lemma \ref{lm:B-disc}.  
Fix $\eta$ and $\zeta$ as in the second claim. Define
	\[A=\Bigl\{\xi\in[\eta,\zeta[:\P(]\xi,\xi']\cap\,\Usetnonuniq\neq\varnothing)=1\quad\forall \xi'\in]\xi,\zeta]\Bigr\}\subset[\eta,\zeta].\]
Note that any point in $A$ is an almost sure (right) accumulation point of $\Usetnonuniqz$. Let $\xi_0\in[\eta,\zeta]$ be such that 
	\begin{align*}
	\xi_0\cdot e_1=\sup\bigl\{\xi'\cdot e_1:\xi'\in[\eta,\zeta]\text{ and }\P([\eta,\xi']\cap\Usetnonuniq\neq\varnothing)=0\bigr\}.
	\end{align*}
%By \eqref{th:main4:b}, $\P(\eta\in\Usetnonuniq)=0$ and so the above set is not empty.
We have $\P([\eta,\xi']\cap\,\Usetnonuniq=\varnothing)=1$ for all $\xi'\in[\eta,\xi_0[$. Taking $\xi'\to\xi_0$ 
implies the same claim for $[\eta,\xi_0[$. Since $\xi_0\in\Diff$, part \eqref{th:main4:b} implies the same holds for $[\eta,\xi_0]$; therefore $\xi_0\ne\zeta$. 
The definition of $\xi_0$, the (already proven) first claim in \eqref{th:main4:b''}, and $\P(\xi_0 \not\in\Usetnonuniq)=1$ now imply 
%$\P([\eta,\xi']\cap\,\Usetnonuniq\neq\varnothing)=1$ for  $\xi'\in]\xi,\zeta]\cap\Ddense$. Recalling the conclusion from the previous paragraph we can replace $\eta$ by $\xi$.
that $\xi_0 \in A$ and so $A$ is not empty. 

For any $\xi\in A$ and $\xi'\in]\xi,\zeta[$ there exists $\xi''\in]\xi,\xi']$ such that 
$\P([\xi'',\xi']\cap\,\Usetnonuniq\neq\varnothing)=1$. Otherwise, taking $\xi''\to\xi$ and using $\P(\xi\in\Usetnonuniq)=0$ we get a contradiction with $\xi\in A$. The previous paragraph  shows that there exists $\xi''' \in ]\xi'',\xi'[ \, \cap \, A$. It follows that $\xi$ is an accumulation point of $A$.  \eqref{th:main4:b''} is proved.

Let \label{OmBusPiall}$\OmBusPiall$ be the intersection of $\Ommono\cap\OmBusall\cap\Omnondeg\cap\Omexist\cap\Omdir$ with 
the full-measure event from the already proven parts \eqref{th:main4:b} and \eqref{th:main4:b'} and with
 $\Omcont{\ximin}\cap\Omcont{\ximax}$ for all of $\fe$'s 
linear segments $[\ximin,\ximax]$, $\ximin\ne\ximax$ (if any). Take $\w\in\OmBusPiall$.

Since $\w\in\Omnondeg$, uniqueness of degenerate extreme solutions comes from Lemma \ref{lm:nondeg}. 
Assumption \eqref{La-reg} implies that 
	\begin{align}\label{all-equal}
	\Uset_{\ximin}=\Uset_{\ximax}=\Uset_{\xi-}=\Uset_{\xi+}=\Uset_\xi\quad\text{for all }\xi\in\Uset.
	\end{align}
Then strong directedness of nondegenerate extreme solutions follows from Theorem \ref{th:main2} (since $\w\in\Omdir$). This proves part \eqref{th:main4:c}.

Consider a solution  $\Pi_x\in\DLR_x^\w$ that is weakly $\Uset_\xi$-directed for some $\xi\in\Uset$. If $\xi=e_i$ for some $i\in\{1,2\}$ then the paragraph following Theorem \ref{th:main2} explains why it must be that $\Pi_x=\Pi_x^{e_i}$.
Assume therefore that $\xi\in\ri\Uset$. 
As explained at the end of Section \ref{sub:DLR}, applying Choquet's theorem gives
%\addmath{mention necessary conditions and a reference}
the existence of a probability measure $\nu$ on $\DLR_x^\w$ such that
	\[\Pi_x=\int_{\ext\DLR_x^\w}  \Pibar_x\, \nu(d\Pibar_x).\]
Fix $\eta,\zeta\in\ri\Uset$ such that $\etamin\cdot e_1<\ximin\cdot e_1$ and $\zetamax\cdot e_1>\ximax\cdot e_1$. Then
	\begin{align*}
	&\Pi_x\bigl\{\etamin\cdot e_1\le n^{-1}X_n\cdot e_1\le\zetamax\cdot e_1\bigr\}\\
	&\qquad\qquad=\int_{\ext\DLR_x^\w} \!\!\! \Pibar_x\bigl\{\etamin\cdot e_1\le n^{-1}X_n\cdot e_1\le\zetamax\cdot e_1\bigr\}\, \nu(d\Pibar_x).
	\end{align*}
The weak $\Uset_\xi$-directedness of $\Pi_x$ implies the left-hand side goes to $1$ as $n\to\infty$.
On the other hand, the strong directedness of extreme DLR solutions, proved in part \eqref{th:main4:c}, implies that the probability being integrated on the right-hand side converges to either $0$ or $1$. 
It converges to $1$ exactly for those $\Pibar_x$ that are strongly $[\etamin,\zetamax]$-directed.  Applying bounded convergence we then get
	\[\nu\bigl\{\Pibar_x\text{ is strongly $[\etamin,\zetamax]$-directed}\bigr\}=1.\]
Taking $\eta$ and $\zeta$ to $\xi$ we conclude that 
	\begin{align}\label{ex-directed}
	\nu\bigl\{\Pibar_x\text{ is strongly $\Uset_\xi$-directed}\bigr\}=1.
	\end{align}
But then this implies that $\Pi_x$ is strongly $\Uset_\xi$-directed and part \eqref{th:main4:d} is proved.

Now fix $\xi\in\Uset\setminus\Usetnonuniq$.
Since $\w\in\Omexist$ and $\Uset_{\xi-}=\Uset_{\xi+}$, we already know from Theorem \ref{th:DLR-dir1} that $\Pi^{\xi,\w}_x$ is a strongly $\Uset_\xi$-directed DLR solution.  
Let $\Pi_x$ be (possibly another) strongly $\Uset_\xi$-directed DLR solution.
If $\ximin\ne\xi$, then assumption \eqref{La-reg} implies $\fe$ is linear on $[\ximin,\ximax]\subset\Diff$ 
and $\w\in\Omcont{\ximin}$ implies  $B^{\ximin-}=B^{\ximin}=B^\xi=B^{\xi-}$. 
Either way, we have $B^{\ximin-}=B^{\xi-}$. Similarly, $B^{\ximax+}=B^{\xi+}$.
By Theorem \ref{thm:Buselim1}  we have $\Pi_x$-almost surely, for all $y\in x+\Z^2_+$,
	\begin{align}\label{foofoo1}
	e^{-B^{\xi-}(y,y+e_1,\w)}\le \varliminf_{n\to\infty}\frac{Z_{y+e_1,X_n}}{Z_{y,X_n}} \le \varlimsup_{n\to\infty}\frac{Z_{y+e_1,X_n}}{Z_{y,X_n}} \le e^{-B^{\xi+}(y,y+e_1,\w)}
	\end{align}
and
	\begin{align}\label{foofoo2}
	e^{-B^{\xi+}(y,y+e_2,\w)}\le \varliminf_{n\to\infty}\frac{Z_{y+e_2,X_n}}{Z_{y,X_n}} \le \varlimsup_{n\to\infty}\frac{Z_{y+e_2,X_n}}{Z_{y,X_n}} \le e^{-B^{\xi-}(y,y+e_2,\w)}.
	\end{align}
%with similar inequalities for $e_2$-increments.	
Consequently, if  $\xi\not\in\Usetnonuniq$, then for $i\in\{1,2\}$
%we get that 
	\[\lim_{n\to\infty}\frac{Z_{y+e_i,X_n}}{Z_{y,X_n}}=e^{-B^\xi(y,y+e_i,\w)}\]
and hence %consequently
	\[\lim_{n\to\infty}\frac{Z_{y,X_n}}{Z_{x,X_n}}=e^{-B^\xi(x,y,\w)}\]
$\Pi_x$-almost surely and for all $y\in x+\Z^2_+$.
Then due to  \eqref{Z-bd} and \eqref{B-Z} applying bounded convergence we deduce that $\Pi_x=\Pi^{\xi,\w}_x$.  
The existence and uniqueness claimed in part \eqref{th:main4:e} have been verified.

As explained above Lemma \ref{lm:cond-bi-to-root}, one can write $\Pi_x$ as a convex integral mixture of extreme measures from $\DLR_x^\w$. This mixture will then have to be supported on DLR solutions that are all strongly $\Uset_\xi$-directed. Uniqueness then implies that they are all equal to $\Pi_x$ and therefore $\Pi_x$ is extreme.

The weak convergence claim comes similarly to \eqref{Q-cv}. The argument for consistency is similar
to the one below \eqref{Q-cv}. \eqref{th:main4:e} is proved.

When $\xi\in\Usetnonuniq$, $\Pi^{\xi\pm,\w}_x$ are two DLR solutions which, by Theorem \ref{th:DLR-dir1} and \eqref{all-equal}, are both strongly $\Uset_\xi$-directed.
The two are different because they are nondegenerate and so if $y\in x+\Z^2_+$ and $i\in\{1,2\}$ are such that $B^{\xi-}(y,y+e_i,\w)\ne B^{\xi+}(y,y+e_i,\w)$, then 
passing through $y$ has a positive probability under both $\Pi^{\xi\pm,\w}_x$, and the transitions out of $y$ are different. 

Since $\Pi_x^{\xi\pm,\w}$ are two different $\Uset_\xi$-directed solutions, there must exist at least two different extreme ones. Part \eqref{th:main4:f} is proved and we are done.
\end{proof}

We can in fact prove a little bit more than the claim in Theorem \ref{th:main4}\eqref{th:main4:f}. 

\begin{lemma}
Under the assumptions of Theorem \ref{th:main4}\,\eqref{th:main4:f} we have that $\Pi_x^{\xi\pm,\w}$ are extreme.
\end{lemma}

\begin{proof}
Applying \eqref{ex-directed} to $\Pi_x=\Pi_x^{\xi-,\w}$ says that this measure is a convex mixture of extreme DLR solutions that are all strongly $\Uset_\xi$-directed.
Then $\nu$-almost surely $\Pibar_x$ is nondegenerate and \eqref{foofoo1} and \eqref{foofoo2} hold $\Pibar_x$-almost surely. By \eqref{cv-as}, the ratios of partition functions converge
 $\Pibar_x$-almost surely and we have 
% \reversemarginpar\addmath{why do we need bounded cv or Fatou??}
% Equation \eqref{Z-rec} implies $0\le Z_{y+e_i,X_n}/Z_{y,X_n}\le e^{-\w_y}$ and combining bounded convergence with Fatou's lemma  implies that 
	\begin{align}\label{temp-ineq}
	e^{-B^{\xi-}(y,y+e_1,\w)}\le e^{-B^{\Pibar_x}(y,y+e_1)} \quad\text{and}\quad e^{-B^{\xi-}(y,y+e_2,\w)}\ge e^{-B^{\Pibar_x}(y,y+e_2)}.
	\end{align}
By the cocycle property of $B^{\Pibar_x}$ we can rewrite the above as
	\[e^{-B^{\Pibar_x}(x,y)}\,e^{-B^{\xi-}(y,y+e_1,\w)}\le e^{-B^{\Pibar_x}(x,y+e_1)} \quad\text{and}\quad e^{-B^{\Pibar_x}(x,y)}\,e^{-B^{\xi-}(y,y+e_2,\w)}\ge e^{-B^{\Pibar_x}(x,y+e_2)}.\]
Integrating any of \eqref{Pi-B-cons}, \eqref{B-Z}, or \eqref{Pi-Z-B} shows that $\int e^{-B^{\Pibar_x}(x,v)}\,\nu(d\Pibar_x)=e^{-B^{\xi-}(x,y,\w)}$. Therefore
	\[e^{-B^{\xi-}(x,y,\w)}\,e^{-B^{\xi-}(y,y+e_1,\w)}\le e^{-B^{\xi-}(x,y+e_1,\w)} \quad\text{and}\quad e^{-B^{\xi-}(x,y,\w)}\,e^{-B^{\xi-}(y,y+e_2,\w)}\ge e^{-B^{\xi-}(x,y+e_2,\w)}.\]
But the cocycle property of $B^{\xi-}$ says the above are in fact equalities. Hence, it must be the case that \eqref{temp-ineq} were in fact equalities and therefore $B^{\Pibar_x}=B^{\xi-}$, $\nu$-almost surely.
In other words, we have shown that $\nu=\delta_{\Pi_x^{\xi-}}$ and therefore $\Pi_x^{\xi-}$ is extreme.  A similar reasoning holds for $\Pi_x^{\xi+}$.
\end{proof}

\begin{proof}[Proof of Theorem \ref{th:cif}]
Let $\Ddense$ be a countable dense subset of $\Diff$ containing the endpoints of all linear segments of $\fe$ (if any).
We define a coupling of certain paths on the tree $\cT^\w_0$. 
Set \label{Omcif}$\Omcif=\bigcap_{\zeta\in\Ddense}\OmBusPizeta$ and take $\w\in\Omcif$.
For $n\in\N$ and $\zeta\in\Ddense$ let $\Xtreen_{0,\infty}$ be the up-right path on $\cT^\w_0$ that goes from $0$ to $\fl{n\zeta}$ and then continues by taking, say, $e_1$ steps.
Let $\wh Q^\w_{0,(n)}$ be the joint distribution of $\cT_0^\w$ and $\{\Xtreen_{0,\infty}:\zeta\in\Ddense\}$, induced by $Q_0^\w$. 
By compactness, the sequence $\wh Q^\w_{0,(n)}$ has a subsequence that converges weakly to 
a probability measure.
Let $\wh Q^\w_0$ be a weak limit. This is a probability measure on trees spanning $\Z^2_+$ and infinite up-right paths on these trees, rooted at $0$ and indexed by $\zeta\in\Ddense$. 
% for any n, the paths stay on the tree up to level n, and these conditions define closed sets, so the limit measure will give full
%probability to the event where every path stays on the tree up to all levels.  
 We denote the tree by $\wh\cT^\w_0$ and the paths by $\Xtree_{0,\infty}$.
The distribution of $\wh\cT^\w_0$ under $\wh Q^\w_0$ is the same as that of $\cT^\w_0$ under $Q_0^\w$.
Furthermore,  since by Lemma \ref{lm:p2pQ-pi} for each $n\in\N$ and $\xi\in\Ddense$ the distribution of $\Xtreen_{0,n}$ under $Q_0^\w$ is exactly $Q^\w_{0,\fl{n\zeta}}$, 
Theorem \ref{th:main3} implies that 
the distribution of $\Xtree_{0,\infty}$ under $\wh Q^\w_0$ is exactly $\Pi_0^{\zeta,\w}$. 
One consequence is that $\Xtree$ is $\Uset_\zeta$-directed, $\wh Q^\w_0$-almost surely and for all $\zeta\in\Ddense$.

We can define a competition interface $\wh\CI_n^\w$ between the subtrees of $\wh\cT^\w_0$ rooted at $e_1$ and $e_2$, and its distribution under $\wh Q^\w_0$ is then the same as the distribution of the original competition interface $\CI_n^\w$ under $Q_0^\w$.
Since $\Xtree$ is a path on the spanning tree $\wh\cT^\w_0$, $\{\Xtree_1=e_2\}$ implies that $\wh\CI_n^\w\cdot e_1\ge\Xtree_n\cdot e_1$ for all $n\in\Z_+$. This in turn implies the event
$\{\varliminf \wh\CI_n^\w\cdot e_1/n\ge\zetamin\cdot e_1\}$. Consequently, for all $\zeta\in\Ddense$,
	\[Q_0^\w\Bigl\{\varliminf_{n\to\infty}\CI^\w_n\cdot e_1/n<\zetamin\cdot e_1\Bigr\}\le\Pi_0^{\zeta,\w}(X_1=e_1)
	=e^{\w_0-B^\zeta(0,e_1,\w)}.\]
A similar argument gives
	\begin{align}\label{fafa2}
	Q_0^\w\Bigl\{\varlimsup_{n\to\infty}\CI^\w_n\cdot e_1/n\le\zetamax\cdot e_1\Bigr\}\ge 
	e^{\w_0-B^\zeta(0,e_1,\w)}.
	\end{align}
For $\xi\in\ri\Uset$ with $\ximax\in\Ddense$ taking $\Ddense\ni\zeta\to\ximax$ with $\zeta\cdot e_1$ strictly decreasing 
makes $\zetamin\to\ximax$. Recall that $\w\in\Omcont{\ximax}$.
Applying the above we get
	\[Q_0^\w\Bigl\{\varliminf_{n\to\infty}\CI^\w_n\cdot e_1/n\le\ximax\cdot e_1\Bigr\}\le e^{\w_0-B^{\ximax}(0,e_1,\w)}.\]
Applying \eqref{fafa2} with $\zeta=\ximax$ we get 
	\[Q_0^\w\Bigl\{\varlimsup_{n\to\infty}\CI^\w_n\cdot e_1/n\le\ximax\cdot e_1\Bigr\}\ge 
	e^{\w_0-B^{\ximax}(0,e_1,\w)}.\]
Since  the liminf is always bounded above by the limsup we get 
	\[Q_0^\w\Bigl\{\varliminf_{n\to\infty}\CI^\w_n\cdot e_1/n\le\ximax\cdot e_1\Bigr\}
	=Q_0^\w\Bigl\{\varlimsup_{n\to\infty}\CI^\w_n\cdot e_1/n\le\ximax\cdot e_1\Bigr\}
	=e^{\w_0-B^{\ximax}(0,e_1,\w)}.\]
A similar argument, starting by taking $\zeta\to\ximin$ and applying \eqref{fafa2}, gives 
	\[Q_0^\w\Bigl\{\varliminf_{n\to\infty}\CI^\w_n\cdot e_1/n<\ximin\cdot e_1\Bigr\}
	=Q_0^\w\Bigl\{\varlimsup_{n\to\infty}\CI^\w_n\cdot e_1/n<\ximin\cdot e_1\Bigr\}
	=e^{\w_0-B^{\ximin}(0,e_1,\w)}\]
for all $\xi\in\ri\Uset$ with $\ximin\in\Ddense$. But for $\xi\in\Ddense$
%whether $\ximin=\ximax$ or $\ximin\ne\ximax$
we  have $B^\xi(\w)=B^{\ximin}(\w)=B^{\ximax}(\w)$. Hence, all four probabilities in the above two displays  equal
$e^{\w_0-B^\xi(0,e_1,\w)}$. We conclude that for any $\xi\in\Ddense$
	\[Q_0^\w\Bigl\{\varliminf_{n\to\infty}\CI^\w_n\cdot e_1/n\le\xi\cdot e_1\Bigr\}
	=Q_0^\w\Bigl\{\varlimsup_{n\to\infty}\CI^\w_n\cdot e_1/n\le\xi\cdot e_1\Bigr\}
	=e^{\w_0-B^{\xi}(0,e_1,\w)}.\]
This implies that  $\cid=\lim_{n\to\infty}\CI^\w_n\cdot e_1/n$ exists $Q_0^\w$-almost surely and its cumulative distribution function is given by \eqref{cid-CDF}.
Parts \eqref{th:cif:a} and \eqref{th:cif:c} are proved.
Part \eqref{th:cif:b} follows because $B^{\xi+}$ is constant on the linear segments of $\fe$.
 For  \eqref{th:cif:d} observe that
	\begin{align*}
	\E Q_0^\w\{\cid=\xi\}=\E\bigl[e^{\w_0}(e^{-B^{\xi+}(0,e_1,\w)}-e^{-B^{\xi-}(0,e_1,\w)})\bigr],
	\end{align*}
which vanishes if and only if $\P\{B^{\xi+}(0,e_1)=B^{\xi-}(0,e_2)\}=1$, which holds if and only if $\xi\in\Diff$.
\end{proof}

\section{Bi-infinite polymer measures} \label{sec:biDLR}
We now prove Theorem \ref{thm:nodirbi} and Lemma \ref{lem:nocovbi}, showing non-existence of two classes of bi-infinite polymer measures. 
The following  is the key step in the proof of Theorem \ref{thm:nodirbi}. 
%The proof is a lack of space argument originally due to Burton and Keane \cite{Bur-Kea-89}.
%Recall that we assume $\beta=1$ without loss of generality.

\begin{lemma}\label{lem:unifto0}
Let $B$ be a shift-covariant cocycle which recovers the potential. Then there is a Borel set \label{Omback}$\Omback{B} \subset \Omega$ with $\bbP(\Omback{B}) = 1$ so that for all $\w \in \Omback{B}$ and for all $x \in \bbZ^2$,
\begin{align*}
\lim_{n\to \infty}\max_{\substack{y \leq x \\ \abs{x-y}_1 = n}}\Pi_{y}^{B(\w)}(x) &=0.
\end{align*}
\end{lemma}

\begin{proof}
By shift-covariance of $B$, it is enough to deal with the case $x=0$.
Couple $\{\Pi_y^{B(\w)}:y\in\Z^2\}$ as described in Section \ref{sub:coupling} and denote the coupled paths by
$X^{y,\w}_{m,\infty}$, or $X^y$ for short, $y\in\V_m$, $m\in\Z$. 
Let $N_v= \{y\le v : v \in X^y\}$.
We will call a point $z\in\Z^2$ a {\sl junction point} if there exist distinct  $u,v\in\Z^2$ such that $\abs{N_u}=\abs{N_v}=\infty$ and 
$X^u$ and $X^v$ coalesce precisely at $z$.

Suppose now that $\bfP \otimes \bbP(\abs{N_0} = \infty) > 0$. The shift-covariance of $B$ implies  
%$\mfg_u(\tau_{v}\unif,T_v \w) = \mfg_{u+v}(\unif, \w)$ and so 
$N_u(\tau_v \unif,T_v\w) = N_{u+v}(\unif, \w)$. Hence,  by the ergodic theorem, with positive $\bfP \otimes \bbP$-probability there is a positive density of sites $v \in \bbZ^2$ with $\abs{N_v} = \infty$. 

By Theorem \ref{thm:RWREcoal}, for $\bfP \otimes \bbP$-almost every $(\unif,\w)$ and all $u,v \in \bbZ^2$, 
$X^u$ and $X^v$ coalesce.
%It follows from this and the previous paragraph that there is a positive $\bfP\otimes\bbP$-probability
%of having at least one junction point.  Again, by the ergodic theorem,
It follows from this and the previous paragraph that 
with positive $\bfP \otimes \bbP$-probability there is a positive density of junction points. 
%The lack of space argument in the proof of Theorem 4.6 in \cite{Geo-Ras-Sep-17-ptrf-2} applies now word for word and leads to a \addmath{uncomment the argument in the arxiv version}
%contradiction, proving that $\bfP\otimes\P(\N_0<\infty)=1$.

For $L \in \bbN$, let $J_L$ denote the union of the junction points in $[1,L]^2$ together with the vertices of the south-west boundary of $[1,L]^2$, $\{k e_i : 1 \leq k \leq L, i \in \{1,2\}\}$, with the property that one of the junction points lies on $X^{k e_i}$. For each junction point $z$, there are at least two such points on the south-west boundary. Decompose $J_L$ into finite binary trees by declaring that the two immediate descendants of a junction point $z$ are the two  closest points $u,v \in J_L$ with the property that $z \in X^u\cap X^v$. 
The leaves of these trees are points in $J_L$ which lie on the boundary and the junction points are the interior points of the trees. This tree cannot have more than $2L+1$ leaves, but this contradicts that there are on the order of $L^2$ junction points, since a binary tree has more leaves than interior points. Thus $\bfP\otimes\P(N_0<\infty)=1$.

Fix $\e >0$. We now know that 
%since $\abs{N_x}$ is $\bfP \otimes \bbP$-almost surely finite, 
$\bfP(\abs{N_0(\unif,\w)} < \infty) = 1$ for $ \bbP$-almost all $\w$. Then there exists an integer $n_0=n_0(\w)$ such  that $\bfP\left(\abs{N_0(\unif,\w)} \geq n\right) < \e$ for  $n\ge n_0$. The claim  follows from the observation that
$\Pi_y^{B}(0) =  \bfP(0\in X^y) \leq \bfP(\abs{N_0(\unif,\w)} \geq n)$ for $y\leq 0$ with $\abs{y}_1 = n$.
\end{proof}

We can now rule out the existence of polymer Gibbs measures satisfying the law of large numbers in a fixed direction and of metastates.

\begin{proof}[Proof of Theorem \ref{thm:nodirbi}]
Let \label{Ombi}$\Ombi = \Omback{B^\xi}\cap\OmBusPi$ and take $\w\in\Ombi$.
 Suppose  there exists a weakly $\Uset_\xi$-directed $\Pi \in \biDLR^{\w}$. 
Take any $x \in \bbZ^2$ 
such that $c=\Pi(x) > 0$. Fix $n\le x\cdot\ehat$. 
If $\Pi(x\,|\,y)\le c/2$ for all $y\in\V_n$ with $\Pi(y)>0$, then $\Pi(y,x)\le c\Pi(y)/2$ for all $y\in\V_n$ 
and adding over $y$
we get  $c=\Pi(x)\le c/2$, which contradicts $c>0$.
Hence, there exists a $y_n\le x$ such that $y_n\in\V_n$, $\Pi(y_{n}) > 0$, and 
$\Pi(x\,|\,y_n)>c/2$. But, by Lemma \ref{lm:cond-bi-to-root}, $\Pi(\acdot\, |\,y_n)$ is a weakly $\Uset_{\xi}$-directed 
element of $\DLR_{y_n}^{\w}$ and, by Theorem \ref{th:main3}, it must be that $\Pi(x\,|\,y_n)=\Pi^{\xi,\w}_{y_n}(x)$.
But then $\Pi_{y_n}^{\xi,\w}(x) > c/2$ for all $n$, which contradicts Lemma \ref{lem:unifto0} since 
Theorem \ref{th:DLR=B} says $\Pi^{\xi,\w}_{y_n}=\Pi^{B^\xi(\w)}_{y_n}$.
\end{proof}

\begin{proof}[Proof of Lemma \ref{lem:nocovbi}]
Suppose that $\Pi$ is a measure satisfying Definition \ref{de:scGibbs}\eqref{de:scGibbs1} and Definition \ref{de:scGibbs}\eqref{de:scGibbs3}.
Then for each $z\in\V_0$
\begin{align*}
 \bbE\left[\Pi^{\w}(X_0 = z) \right]= \bbE\left[\Pi^{T_z\w}(X_0 = 0) \right]=\bbE\left[\Pi^{\w}(X_0 = 0)\right].
\end{align*}
This is a contradiction since $\{X_0 = z\}$, $z\in\V_0$, form a partition of $\bbX$.
\end{proof}

\appendix
\section{Coupled RWRE paths with $\{\lowercase{e}_1,\lowercase{e}_2\}$ steps}\label{app:RWRE}
\subsection{Path coupling}\label{sub:coupling}
In this section we construct a coupling of a family of random walks in a  random environment (RWRE) with admissible steps $\{e_1,e_2\}$ that several arguments in this paper rely on.
%To avoid re-introducing the same ideas repeatedly, we will only write out the coupling explicitly once. 

Let $(\Omega, \sF, \bbP)$ satisfy the assumptions of Section \ref{sec:setting}. Let $\bfP$ denote the law of i.i.d.~Uniform[0,1] random variables $\unif = \{\unif(y) : y \in \bbZ^2\}$  on $[0,1]^{\bbZ^2}$, equipped with the Borel $\sigma$-algebra
and the natural group of coordinate shifts $\tau_x$. Define a family of shifts on the product space $[0,1]^{\bbZ^2} \times \Omega$ indexed by $x \in \bbZ^2$ in the natural way, via $\widetilde{T}_{x}(\nu, \w) = (\tau_x \nu, T_x \w)$. This shift preserves
 $\bfP \otimes \bbP$.
 
Let $\sA$ be some index set and let $\{p_x^\alpha : x\in\Z^2, \alpha \in \sA\}$ be a collection of $[0,1]$-valued $\sF$-measurable random variables. Abbreviate $ \mfG = \{e_1,e_2\}^{\bbZ^2}$.
% and let $\proj_x$ denote the natural coordinate projection in this space. 
For $\alpha \in \sA$, construct a random graph $\mfg^{\alpha}(\unif, \w) = \mfg^{\alpha} \in \mfG$, via
\begin{align*}
\mfg^{\alpha}_x =  \begin{cases}
e_1 & \text{ if } \unif(x) < p_x^{\alpha}( \w),\\
e_2 & \text{ if } \unif(x) \geq p_x^{\alpha}( \w).
\end{cases}
\end{align*} 
For each $x \in \bbV_m$, $m\in\Z$, let $X_{m,\infty}^{x, \alpha,\w} = X_{m,\infty}^{x,\alpha,\w} (\unif)$ denote the random path defined via $X_m^{x,\alpha,\w} =x$ and $X_k^{x,\alpha,\w} = X_{k-1}^{x,\alpha,\w} + \mfg^{\alpha}_{X_{k-1}^{x,\alpha,\w}}(\unif,\w)$ for $k >m$. We observe that under $\bfP$, for fixed $\alpha$, $X_{m,\infty}^{x,\alpha}$ has the law of a quenched RWRE with admissible steps $\{e_1,e_2\}$ started from $x$ and taking the step $e_1$ at site $y$ with probability $p_y^\alpha(\w)$. Two properties
follow immediately. %from the definitions. 
%This coupling has two important properties which follow immediately from the definition.

\begin{corollary} \label{cor:coupcoal}
The following hold for any $\w\in\Omega$ and $\unif\in[0,1]^{\Z^2}$.
\begin{enumerate}[label={\rm(\alph*)}, ref={\rm\alph*}]
\item\label{cor:coupcoal:a}{\rm(}Coalescence{\rm)} If for some $\alpha\in\sA$, $x,y \in \bbZ^2$, and $n\ge\max(x\cdot\ehat,y\cdot \ehat)$ we have 
$X_n^{x,\alpha,\w}(\unif) = X_n^{y,\alpha,\w}(\unif)$, then $X_{k}^{x,\alpha,\w}(\unif) = X_k^{y,\alpha,\w}(\unif)$ for all $k\ge n$. 
\item{\rm(}Monotonicity{\rm)} Fix $x\in\V_m$, $m\in\Z$, and $\alpha_1,\alpha_2\in\sA$. 
If $p_y^{\alpha_1}(\w) \leq p_y^{\alpha_2}(\w)$ for all $y\ge x$  then $X_{n}^{x,\alpha_1,\w}(\unif) \cdot e_1 \leq X_n^{x,\alpha_2,\w}(\unif) \cdot e_1$
for all $n\ge m$.
\end{enumerate}
\end{corollary}

The proof of Lemma \ref{lem:mono} is an example of how we use this coupling.

\begin{proof}[Proof of Lemma \ref{lem:mono}]
It suffices to work with a fixed $\beta \in (0,\infty)$. The case $\beta=\infty$ comes by taking a limit.
%By taking $\beta \to \infty$. 
Fix $n\in\Z$ and construct the coupled paths $X^{x,\beta,h,\w}_{m,\infty}(\unif)$, $x\in\V_m$, $m\in\Z$, as above, with
\begin{align*}
p_x(\w) = \begin{cases}
		e^{\beta \w_{x+e_1} + \beta h \cdot e_1}\frac{Z_{x+e_1,(n)}^{\beta,h}}{Z_{x,(n)}^{\beta,h}}&\text{if $\abs{x}_1 < n, x \ge 0$},\\ 
		1/2&\text{otherwise}.
		\end{cases}
\end{align*}
Note that for $x \in \bbV_m$, $m+1<n$, and $i,j \in \{1,2\}$
\begin{align*}
\partial_{h_i} F_{x,(n)}^{\beta,h} &= E_{x,(n)}^{\w,\beta,h}[e_i \cdot(X_n-x)] \quad\text{and}\quad \partial_{h_i} F_{x+e_j,(n)}^{\beta,h}= E_{x+e_j,(n)}^{\w,\beta,h}[e_i \cdot (X_n-x-e_j)].
\end{align*}
It follows that whenever $x \in \bbV_m$, $m<n$ and $i,j \in \{1,2\}$,
\begin{align*}
\partial_{h_i} B_n^{\beta,h}(x,x+e_j) %&= \partial_{h_i} \left(F_{x,(n)}^{\beta,h} - F_{x+e_j,(n)}^{\beta,h} - h \cdot e_j\right) \\
&= E_{x,(n)}^{\w,\beta,h}[e_i \cdot X_n] - E_{x+e_j,(n)}^{\w,\beta,h}[e_i\cdot(X_n - e_j)] - e_i  \cdot e_j\\
%&= E_{x,(n)}^{\w,\beta,h}[e_i \cdot X_n] - E_{x+e_j,(n)}^{\w,\beta,h}[e_i\cdot X_n] \\
&= \bfE \left[e_i\cdot(X_n^{x,\beta,h,\w} - X_{n}^{x+e_j,\beta,h,\w})\right] .
\end{align*}
Then Corollary \ref{cor:coupcoal}\eqref{cor:coupcoal:a} and planarity imply that
\begin{align*}
&\partial_{h_i} B_n^{\beta,h}(x,x+e_i) \leq 0  \quad\text{and}\quad \partial_{h_{3-i}} B_n^{\beta,h}(x,x+e_{i}) \geq 0. \qedhere
\end{align*}
\end{proof}

\subsection{Coalescence of RWRE paths}
We show that the quenched measures of a general 1+1-dimensional  random walk with $\{e_1,e_2\}$ steps in a stationary weakly elliptic random environment can be coupled so that the paths coalesce. The proof is an easier version of the well-known Licea-Newman \cite{Lic-New-96} argument for coalescence of first-passage percolation geodesics. Notably, the measurability issues which make the Licea-Newman argument somewhat involved in zero temperature vanish in positive temperature due to the extra layer of randomness coming from the coupling.

Let $p:\Omega\to[0,1]$ be a measurable function. 
Assume {\sl weak ellipticity}:
	\begin{align}\label{ass:ellip}
	\P(0<p<1)=1.
	\end{align}
Construct random variables $\{X_{m,\infty}^x : x \in \V_m,m\in\Z\}$ via the coupling described in Subsection \ref{sub:coupling} with $p_x(\w) = p(T_x\w) $ and write $P = \bfP \otimes \bbP$. Let $P^\w_x$ (with expectation operator $E^\w_x$) be the distribution (on $(\Z^2)^{\Z_+}$) of the corresponding random walk in random environment starting at $x$.

\begin{lemma}\label{lm:crossing}
Assume \eqref{ass:ellip}. For $\P$-almost every $\w$, for all $x\in\Z^2$,  $P^\w_x$-almost every path crosses all vertical lines to the right of $x$ and all
horizontal lines above $x$. 
\end{lemma}
\begin{proof}
By shift-invariance, it is enough to prove the claim for $x=0$. 
For $i\in\{1,2\}$ let $\mathcal I_i$ be the $T_{e_i}$-invariant $\sigma$-algebra.
Note that $\E[\log p\,|\,\mathcal I_1]<0$ and $\E[\log(1-p)\,|\,\mathcal I_2]<0$. The ergodic theorem implies then that 
\[\P\Big( \prod_{k=0}^\infty p(T_{ke_1}\w) = 0 \Big) = \P\Big( \prod_{k=0}^\infty (1-p(T_{ke_2}\w)) = 0 \Big) =1.\]
By a union bound, we have that
	\[\P\Big\{ P_0^\w(\text{the path has at most finitely many $e_2$ increments})>0\Big\}\le \sum_{x\in\Z_+^2}\P\Big( \prod_{k=0}^\infty p(T_{x+ke_1}\w) > 0 \Big)=0.\]
A similar argument works for the case of finitely many $e_1$ increments.
\end{proof}

\begin{theorem}\label{thm:RWREcoal}
Assume \eqref{ass:ellip}. Then $P$-almost surely, for any $u,v\in\Z^2$ there exists an $n\in\Z$ with $X^u_{n,\infty}=X^v_{n,\infty}$.
\end{theorem}

\begin{proof}
The proof comes by way of contradiction. 
Observe that if $X^u$ and $X^v$ ever intersect, say $X^u_n=X^v_n$, then we would have $X^u_{n,\infty}=X^v_{n,\infty}$.
So suppose  $P\{X_{r,\infty}^u\cap X_{k,\infty}^v=\varnothing\}>0$ for some $u\in\V_r$, $v\in\V_k$, $r,k\in\Z$.  
By Lemma \ref{lm:crossing} these paths  must cross any  vertical line to the right of $u$ and $v$.   
Restart  the paths  from the points where they exit some such vertical line.   By  stationarity  we can assume $u=0$ (hence $r=0$), $X_1^0=e_1$, $v=ke_2$,  and  $X^{ke_2}_{k+1}=ke_2+e_1$. 
Thus we have %take the following   assumption as the basis from which a contradiction will come.    
		\[P\big\{X_{0,\infty}^0\cap X_{k,\infty}^{k e_2}=\varnothing, \, X^0_1=e_1,X^{ke_2}_{k+1}=ke_2+e_1\big\}>0.\]
Again by  shift invariance   $X_{i,\infty}^{ie_2}\cap X_{i+k,\infty}^{(i+k)e_2}=\varnothing$ for infinitely many $i\ge0$,  with positive probability.
Consequently, there exists $i>k$ such that 
		\[P\big\{X_{0,\infty}^0\cap X_{k,\infty}^{k e_2}=\varnothing,\, X_{i,\infty}^{ie_2}\cap X_{i+k,\infty}^{(i+k)e_2}=\varnothing,\, X^0_1=e_1,\, X^{ke_2}_{k+1} = ke_2+e_1,\, X^{(i+k)e_2}_{i+k+1}=(i+k)e_2+e_1\big\}>0.\]
Let  $\ell=i+k$. If $X_{k,\infty}^{k e_2}\cap X_{\ell,\infty}^{\ell e_2}\neq\varnothing$ then by planarity   $X_{i,\infty}^{i e_2}$  intersects   $X_{\ell,\infty}^{\ell e_2}$.
So we have integers $0<k<\ell$ such that 
			\[P\big\{X_{0,\infty}^0\cap X_{k+1,\infty}^{k e_2+e_1}=\varnothing,\, X_{k+1,\infty}^{k e_2+e_1}\cap X_{\ell,\infty}^{\ell e_2}=\varnothing,\,X_1^0=e_1,X_{\ell+1}^{\ell e_2}=\ell e_2+e_1\big\}>0.\]
Let $\tau_\ell\in\N$ be the first coordinate of the point where $X^0_{0,\infty}$ first reaches the horizontal line $\ell e_2+\R_+ e_1$. Let $\tau' e_1$ be the point at which $X^0_{0,\infty}$ exits the horizontal line $\R_+ e_1$. 
Then we can also find integers $0<m<n$ such that
			\[P\big\{X_{0,\infty}^0\cap X_{k+1,\infty}^{k e_2+e_1}=\varnothing,\, X_{k+1,\infty}^{k e_2+e_1}\cap X_{\ell,\infty}^{\ell e_2}=\varnothing,\,X_1^0=e_1,X_{\ell+1}^{\ell e_2}=\ell e_2+e_1,\tau'=m,\tau_\ell=n\big\}>0.\]

	\begin{figure}[h]
		\begin{center}
		
			\begin{tikzpicture}[>=latex, scale =0.69]
				
				\draw(0,0)rectangle(8,4);
				\draw(0,0)node[below, left]{$0$};
				\draw(8,0)node[below]{$n$};
				\draw(3,0)node[below]{$m$};
				\draw(0,4)node[left]{$\ell$};
				\draw(0,2)node[left]{$k$};
				
				\draw[color=nicosred, line width= 2.5pt] (0,0.5)--(0,3.5);
				\draw[color=nicosred, line width= 2.5pt] (3.5,0)--(8,0);

				\draw[darkgreen, line width=1.2pt]plot coordinates{(0,4)(1,4)(1,4.5)(1.5,4.5)(1.5,5.5)(2.5,5.5)};
				\draw[darkgreen, line width=1.2pt,->]plot[smooth] coordinates{(2.5,5.5)(3.4,5.8)(4,6.8)(4.8,7.2)};
				\draw[darkgreen, line width=1.2pt]plot coordinates{(0,0)(3,0)(3,1)(3.5,1)(3.5,2)(5,2)(5,2.5)(6,2.5)(6,3)(7,3)(7,3.5)(8,3.5)(8,4)};
				\draw[darkgreen, line width=1.2pt,->]plot[smooth] coordinates{(8,4)(8.5,4.5)(9,5.5)(9.7,5.9)};
				\draw[darkgreen, line width=1.2pt]plot coordinates{(0.5,2)(2,2)(2,3)(3,3)(3,3.5)(3.5,3.5)};
				\draw[darkgreen, line width=1.2pt,->]plot[smooth] coordinates{(3.5,3.5)(4,3.7)(4.5,4.2)(5,4.7)(5.5,5.7)(6.2,6)};

		\end{tikzpicture}

		\end{center}
	\caption{The non-intersecting paths $X_{0,\infty}^0$, $X_{k+1,\infty}^{k e_2+e_1}$, and $X_{\ell,\infty}^{\ell e_2}$. The variables $\unif_x$ on the thick segment on the south edge of the rectangle are such that  $\mfg_x=e_1$.
	The variables $\unif_x$ on the west edge of the rectangle are such that $\mfg_x=e_2$.}
		
	\label{fig:geo}
\end{figure}
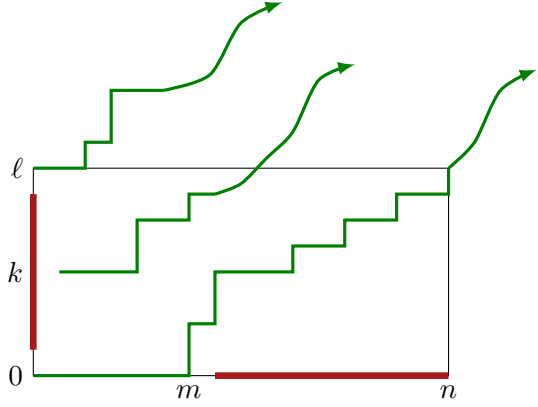

The event in the above probability is independent of the variables $\unif_{(m+1)e_1,ne_1}\cup \unif_{e_2,(\ell-1)e_2}$ (where for example $\unif_{(m+1)e_1,ne_1} = \{\unif_{(m+1 )e_1} , \dots, \unif_{n e_1}\}$). Since $\P(0<p<1)=1$ we have that
			\begin{align*}
			P\big\{&X_{0,\infty}^0\cap X_{k+1,\infty}^{k e_2+e_1}=\varnothing,\, X_{k+1,\infty}^{k e_2+e_1}\cap X_{\ell,\infty}^{\ell e_2}=\varnothing,\,X_1^1=e_1,X_{\ell+1}^{\ell e_2}=\ell e_2+e_1,\tau'=m,\tau_\ell=n,\\
			&\qquad\qquad\qquad\qquad \mfg_x =e_1\text{ for }x\in[m+1,n]e_1\text{ and } \mfg_x=e_2\text{ for }x\in[1,\ell-1]e_2\big\}>0.
			\end{align*}
Call the above event $A_1$. (See Figure \ref{fig:geo}.) On this event, path $X_{k+1,\infty}^{k e_2+e_1}$ is shielded by the arrows on the boundary $[m+1,n]e_1\cup[1,\ell-1]e_2$ and by paths $X_{0,\infty}^0$ and $X_{\ell,\infty}^{\ell e_2}$ and for any $u\not\ge 0$ with $u\cdot\ehat=s$
we have $X_{k+1,\infty}^{k e_2+e_1}\cap X_{s,\infty}^u=\varnothing$.

	The Burton-Keane lack of space argument \cite{Bur-Kea-89} furnishes the necessary contradiction.
Indeed, by $P(A_1)>0$ and the ergodic theorem
there exists an event $A_2$ of positive probability such that on $A_2$ for all  large enough $L$  and a small enough fixed  $\delta>0$,  event 
$A_1\circ T_z$ occurs for  at least $\delta L^2$ points $z\in[0,L]^2$ such that the rectangles $z+[0,n]\times[0,\ell]$ are pairwise disjoint and lie inside $[0,L]^2$.    Then with positive probability we have $\delta L^2$ pairwise disjoint paths that start inside $[0,L]^2$.  
Each of these paths must exit through a boundary point of  $[0,L]^2$,  but for large enough $L$ the number of   boundary points is $< \delta L^2$.    The theorem is proved.
\end{proof}

\section{A shape theorem for cocycles}\label{app:BuseShape}
The results in this section extend \cite[Theorem A.3]{Geo-etal-15} to the stationary setting. The proof is identical once one alters the definitions appropriately. Fix a dimension $d \in \bbN$ and let $\sR \subset \bbZ^d$ be an arbitrary finite set of admissible steps that contains at least one nonzero element. Admissible paths $x_{0,n}=(x_k)_{k=0}^n$ satisfy $x_k - x_{k-1} \in \sR$. Let $M = |\sR|$. 
%Define
%\begin{align*}
%\sG^+ &= \left\{ \sum_{z \in \sR}b_z  z : b_z \in \bbZ_+\right\}
%\end{align*} 
%and let $\sG = \sG^+ - \sG^+$ be the additive group generated by $\sR$.
Let $\sG=\{\sum_{z\in\sR}b_z z:b_z\in\Z\}$ be the additive group generated by $\sR$.

\begin{definition}
A shift covariant-cocycle is a Borel-measurable function $B: \sG \times \sG \times \Omega \to \bbR$ which satisfies
\begin{enumerate}
\item (Shift covariance) $B(x+z,y+z,\omega) = B(x,y,T_z \omega)$ for all $x,y,z\in \sG$ and $\bbP$-almost all $\omega$.
\item (Cocycle property) $B(x,z,\omega) = B(x,y,\omega) + B(y,z,\omega)$ for all $x,y,z \in \sG$ and $\bbP$-almost all $\omega$.
\end{enumerate}
A cocycle is said to be $L^p(\bbP)$ if $\bbE[|B(0,z)|^p]< \infty$ for all $z \in \sR$. 
\end{definition}

\begin{definition}
A function $V:\sR \times \Omega\to\R$ is in class $\sL$ if for every nonzero $z \in \sR$ %, $V(x, \acdot) \in L^1(\bbP)$ and
\begin{align}
\lim_{\e \searrow 0} \varlimsup_{n \to \infty} \max_{x \in \sG : |x|_1\leq n \e}\, \frac{1}{n} \sum_{0 \leq k \leq n \e}\abs{V(z, T_{x + k z} \omega)} = 0.
\end{align}
\end{definition}
Boundedness of $V$ is of course sufficient. If $V$ is a local measurable function of $\{\w_x:x\in\Z^d\}$ and these are i.i.d.\ with $2+\e$ moments for some $\e>0$, then $V\in\sL$. More generally, membership in $\sL$ depends on a combination of 
mixing of $\P$ and moments of $V$.  See Lemma A.4 of \cite{Ras-Sep-Yil-13} for a precise statement. 

%Fix an enumeration $\sR = \{z_1,\dots,z_M\}$. We introduce the notation 
Let
	\[D_n = \Bigl\{x \in \sG : \exists b_z\in\Z_+, \sum_{z\in\range} b_z z = x,  \sum_{z\in\range}b_z = n\Bigr\}\] 
denote the set of sites which can be reached in $n$ admissible steps. 

%Let $\sI$ be $\sigma$-algebra generated by $T$-invariant events. 
For $A\subset\sG$ let $\sI_A$ be the $\sigma$-algebra generated by events that are invariant under $T_x$ for all $x\in A$. Let $\sI=\sI_{\sG}$.

For a shift-covariant $L^1(\P)$ cocycle $c(x)=\E[B(0,x)\,|\,\sI]$ is an additive function on the group $\sG$ and hence there exists a vector $m(B)\in\R^d$ with $\E[B(0,x)\,|\,\sI]=m(B)\cdot x$ for all $x\in\sG$. This vector is not unique unless $\range$ spans $\R^d$, but the inner products $m(B)\cdot x$, $x\in\sG$, are well defined.  

%define a vector $h \in \bbR^M$ via $h = (h(z_1),\dots,h(z_M))$, where $h(z)=\E[B(0,z)\,|\,\sI]$ for $z \in \sR$ and $\sI$ is the $\sigma$-algebra generated by $T$-invariant events.
\begin{theorem}\label{thm:CocShp}
Suppose  $B$ is a shift-covariant $L^1(\bbP)$ cocycle and there exists a function $V(z,\omega) : \sR \times \bbR \to \bbR$ with $V\in\sL$ such that $B(0,z,\omega) \leq V(z,\omega)$ for all $z \in \sR$ and $\bbP$-almost every $\omega$. Then $\bbP$-almost surely,
\begin{align*}
\lim_{n\to\infty} \max_{x \in D_n } \frac{|B(0,x) - m(B)\cdot x|}{n} = 0.
\end{align*} 
%where $x_1,\dots,x_M \in \bbZ$ are arbitrary integers which satisfy $x = (x_1,\dots,x_M)\cdot (z_1,\dots,z_M)^t$.
\end{theorem}

The rest of this section proves the above theorem.

\begin{lemma}
Suppose that $B$ is a shift-covariant $L^1(\bbP)$ cocycle. Let $x \in \sG$. Then $\bbP$-almost surely, 
$\E[B(0,x)\,|\,\sI_x]=m(B)\cdot x$.
\end{lemma}

\begin{proof}
By shift-covariance and the cocycle property,
\begin{align*}
\frac{B(0,nx,\w)}{n} &= \frac{\sum_{i=0}^{n-1} B(ix, (i+1)x, \w)}{n} = \frac{\sum_{i=0}^{n-1} B(0, x,T_{ix}\w)}{n}.
\end{align*}
By Birkhoff's ergodic theorem the limit exists $\P$-almost surely and equals $L=\bbE[B(0,x)\,|\,\sI_x]$.
On the other hand, by shift-covariance and the cocycle property, we also have
\begin{align*}
\frac{B(0,nx, T_y\omega)}{n} = \frac{B(y,y+nx,\omega)}{n} =  \frac{B(y,0, \omega)}{n} + \frac{B(0,nx,\omega)}{n} + \frac{B(0,y , T_{nx}\omega)}{n}.
\end{align*}
The left-hand side converges $\P$-almost surely to $L\circ T_y$ and second term on the right-hand side converges $\P$-almost surely to $L$. The first term on the right-hand side converges $\P$-almost surely to $0$. This implies that the last term must also converge $\P$-almost surely. Since it converges to $0$ in probability, its almost sure limit must also be $0$.
Consequently, we have shown that  $\bbE[B(0,x)|\sI_x]  \circ T_y = \bbE[B(0,x)\,|\,\sI_x]$, which implies that 
$\bbE[B(0,x)|\sI_x] $ is $\sI$-measurable. Therefore $\bbE[B(0,x) | \sI_x] = \bbE\bigl[\bbE[B(0,x)\,|\,\sI_x ] \,|\, \sI\bigr] = \bbE[B(0,x)\,|\,\sI ]  = m(B)\cdot x$.
\end{proof}

A consequence of the above lemma is that if $A\subset\Z^d$ and $x\in A$, then 
	\begin{align}\label{very nice}
	\begin{split}
		\E[B(0,x)\,|\,\sI_A]&=\bbE\bigl[\bbE[B(0,x)\,|\,\sI_x ] \,|\, \sI_A\bigr]=\bbE\bigl[\bbE[B(0,x)\,|\,\sI ] \,|\, \sI_A\bigr]\\
		&=\bbE[B(0,x)\,|\,\sI ] = m(B)\cdot x.
		\end{split}
	\end{align}

Abbreviate $\Bbar(x,y)=B(x,y)-m(B)\cdot(y-x)$. Note that $\Bbar$ is also a shift-covariant cocycle and that
$\E[\Bbar(x,y)\,|\,\sI]=\E[\Bbar(0,y-x)\,|\,\sI]=0$.

\begin{lemma}\label{lem:apest0}
Suppose that $B$ is a shift-covariant $L^1(\bbP)$ cocycle. 
Let the integers $j,r$ such that $1\le j\le r\le M$ be given
and let $z_1,\dots,z_r \in \sR$ be distinct. Let $g:[0,1]^r\to\R$ be a continuous function. Then $\bbP$-almost surely
\begin{align*}
\lim_{n\to\infty} \frac{1}{n^r} \sum_{k_1=0}^{n-1} \dots \sum_{k_r=0}^{n-1} g(n^{-1}(k_1,\dots,k_r))\Bbar(0,z_j, T_{k_1 z_1 + \dots + k_r z_r}\w ) = 0.
\end{align*} 
\end{lemma}

\begin{proof}
The case $g\equiv1$ follows from the multidimensional ergodic theorem \cite[Theorem 6.2.8]{Kre-85}, 
which in this case says the $\P$-almost sure limit exists 
and equals $\E[\Bbar(0,z_j)\,|\,\sI_{\{z_1,\dotsc,z_r\}}]$, and an application of \eqref{very nice}.
The case of a general continuous $g$ comes by the familiar uniform approximation by step functions.
\end{proof}

\begin{lemma} \label{lem:apest1}
Suppose that $B$ is a shift-covariant $L^1(\bbP)$ cocycle. Let $r \in \{1,2,\dots,M\}$ be given and let $z_1,\dots,z_r \in \sR$ be distinct. Let $0 \leq a_i < b_i$ be given for $1 \leq i \leq r$. Then $\bbP$-almost surely
\begin{align*}
\lim_{n\to\infty} \frac{1}{n^r} \sum_{k_1=\lf n a_1 \rf}^{\lf n b_1\rf -1} \dots \sum_{k_r=\lf n a_r \rf}^{\lf n b_r\rf -1} \frac{\Bbar(0, k_1 z_1 + \dots + k_r z_r)}{n} &= 0.
\end{align*} 
\end{lemma}

\begin{proof}
%Note that we have shown that $h(x) = \bbE[B(0,x)]|\sI_A]$ where $\sI_A$ is the sigma algebra generated events which are invariant under any non-empty subset $ A\subseteq \{T_x : x \in \sG, x \neq 0\}$. It follows from the multidimensional ergodic theorem that
%\begin{align*}
%\lim_{n\to\infty} \frac{1}{n^r} \sum_{k_1=\lf n a_1 \rf}^{\lf n b_1\rf -1} \dots \sum_{k_r=\lf n a_r \rf}^{\lf n b_r\rf -1} B(0, x, T_{k_1 z_1 + \dots + k_r z_r}\omega ) - h(x,\omega) &= 0.
%\end{align*} 
%A uniform approximation argument by step functions now shows that for any continuous $g : [0,1]^r \to \bbR$, we have
%\begin{align*}
%\lim_{n\to\infty} \frac{1}{n^r} \sum_{k_1=0 }^{n -1} \dots \sum_{k_r=0}^{n -1} g(n^{-1}(k_1,\dots,k_r))\left(B(0, x, T_{k_1 z_1 + \dots + k_r z_r}\omega ) - h(x,\omega)\right) &= 0.
%\end{align*}
By taking differences and re-indexing, it suffices to consider the case $a_i = 0$ and $b_i = 1$ for all $i$. We now prove the result by induction on $r$.  
To start the induction, note that
\begin{align*}
\frac{1}{n} \sum_{k=0}^{n-1} \frac{\Bbar(0,k z_1, \omega)}{n} &= \frac{1}{n} \sum_{k=0}^{n-1} \frac{1}{n}\sum_{j=0}^{k-1} \Bbar(0,z_1, T_{j z_1} \omega)  
= \frac{1}{n} \sum_{j=0}^{n-1}\left(1 - \frac{j+1}{n}\right)\Bbar(0,z_1, T_{j z_1}\w ),
%  \mathop{\longrightarrow}_{n\to\infty} 0.
\end{align*}
which by Lemma \ref{lem:apest0} goes to $0$ as $n\to\infty$.
Now, suppose the result has been proven for $r-1$ and take $r \in \{2,\dots,M\}$. Write
\begin{align*}
&\frac{1}{n^r} \sum_{k_1=0}^{n-1} \dots \sum_{k_r=0}^{n-1} \frac{\Bbar(0, k_1 z_1 + \dots + k_r z_r, \omega)}{n}\\
&=\frac{1}{n^r} \sum_{k_1=0}^{n-1} \dots \sum_{k_r=0}^{n-1} \frac{\Bbar(0, k_1 z_1 + \dots + k_{r-1} z_{r-1},\omega )}{n} 
+ \frac{1}{n^r} \sum_{k_1=0}^{n-1} \dots \sum_{k_r=0}^{n-1} \frac{\Bbar(0, k_r z_r, T_{k_1 z_1 + \dots + k_{r-1} z_{r-1} } \omega)}{n}\\
&= \frac{1}{n^{r-1}} \sum_{k_1=0}^{n-1} \dots \sum_{k_{r-1}=0}^{n-1} \frac{\Bbar(0, k_1 z_1 + \dots + k_{r-1} z_{r-1},\omega )}{n} \\
&\qquad\qquad+ \frac{1}{n^r} \sum_{k_1=0}^{n-1} \dots \sum_{k_r=0}^{n-1} \left(1 - \frac{k_r+1}{n}\right)\Bbar(0,z_r, T_{k_1z_1 + \dots + k_r z_r}\omega). 
\end{align*}
The first term tends to $0$ by the induction hypothesis and the second tends to $0$ by Lemma \ref{lem:apest0}.
\end{proof}

\begin{proof}[Proof of Theorem \ref{thm:CocShp}]
 Fix a labelling $z_1, z_2, \dots, z_M$ of the admissible steps $\sR$.We first show that
\begin{align}\label{liminf B}
\varliminf_{n\to\infty} \min_{x \in D_n} \frac{\Bbar(0,x)}{n} \geq 0.
\end{align}
 Let $\delta > 0$ be given and define $a_k = k\delta/(4M)$ for $k \in \bbZ_+$. For $\bfk = (k_1,\dots,k_M) \in \bbZ_+^M$ introduce the notation
\begin{align*}
C_{n,\bfk} &= \Bigl\{\sum_{i=1}^M s_i z_i : \lf n a_{k_i}\rf \leq s_i < \lf n a_{k_i + 1}\rf \Bigr\}.
\end{align*}
%For $x \in D_n$ we have $x = \sum_{i=1}^M b_i z_i$ with $\sum_{i=1}^M b_i = n$ and $b_i \in\Z_+$. 
%Take $k_i$ minimal such that $\lf n a_{k_i} \rf \geq b_i$ to obtain 
%a vector $\bfk=\bfk(x)$ with the property that every $y \in C_{n,\bfk}$ can be reached from $x$ using an admissible path with at most $n \delta$ steps.  With this in mind, let $K = \Z_+^M\cap[0,4M/\delta +1]^M$. Then, using that $0 \leq b_i/n \leq 1$, we see that for all $n\in \bbN$ and for all $x \in D_n$, there exists a vector $\bfk \in K$ such that every point $y \in C_{n,\bfk}$ can be reached from $x$ in at most $n \delta$ steps. For each $x$ let $\bfk(x)$ be a $\bfk \in K$ with this property.
Let $K = \Z_+^M\cap[0,4M/\delta+1]^M$. For any $n\in \bbN$ and $x \in D_n$, write  
$x = \sum_{i=1}^M b_i z_i$ with $\sum_{i=1}^M b_i = n$ and $b_i \in\Z_+$, then 
take $k_i$ minimal such that 
$\lf n a_{k_i} \rf \geq b_i$. This way we obtain a vector $\bfk(x) \in K$ such that every point $y \in C_{n,\bfk(x)}$ can be reached from $x$ in at most $n \delta$ steps.
%\fl{na_{k_i-1}}<b_i implies na_{k_i-1}\le b_i implies n(k-1)\delta/(4M)\le n implies k\le 4M/\delta+1

For each $x \in D_n$ and each $y \in C_{n,\bfk(x)}$, fix a path $x_{0,\ell}$, $\ell\le\delta n$,  from $x$ to $y$ with the property that the steps $z_1,\dots,z_M$ are taken in order. Denote by $\rho_0 = \max\{|z|_1 : z \in \sR\}$. Then we have
\begin{align*}
\Bbar(0,x)&= \Bbar(0,y) - \Bbar(x,y) \\
&= \Bbar(0,y) - \sum_{i=0}^{\ell-1} \Bbar(0,x_{i+1}-x_i, T_{x_i}\omega) \\
&\geq \Bbar(0,y) - \sum_{i=0}^{\ell-1} V(x_{i+1}-x_i, T_{x_i}\omega)\one\{x_{i+1}\neq x_i\} - \ell \max_{i \leq M}|m(B)\cdot z_i|\\
&\geq \Bbar(0,y) - \max_{z \in \sR \backslash\{0\}} \Bigl\{ \max_{|u|_1 \leq 2n\rho_0} \sum_{0\leq i \leq n \delta} |V(z, T_{u + i z}\omega)|\Bigr\} - \ell \max_{i \leq M}|m(B)\cdot z_i |.
\end{align*}
Note that the error term is independent of $x$ and $y$. Average over $y \in C_{n,\bfk(x)}$ and then take a minimum over $x \in D_n$ to obtain
\begin{align*}
\min_{x \in D_n} \frac{\Bbar(0,x)}{n} &\geq \min_{\bfk \in K} \frac{1}{N_{n,\bfk}} \sum_{s_1 = \lf n a_{k_1}\rf}^{\lf n a_{k_1+1}\rf - 1} \dots \sum_{s_M = \lf n a_{k_M}\rf}^{\lf n a_{k_M+1}\rf - 1} \frac{\Bbar(0, \sum_{i=1}^M s_i z_i)}{n} \\
&- \sum_{z \in \sR \backslash \{0\}} \Bigl\{\max_{|u|_1 \leq 2n \rho_0} \frac{1}{n} \sum_{0 \leq i \leq n \delta} |V(z,T_{u+iz}\omega)| \Bigr\} -\delta \max_{i \leq M}|m(B)\cdot z_i|,
 \end{align*}
where $N_{n,\bfk} = \prod_{i=1}^M (\lf n a_{k_i+1}\rf - \lf n a_{k_i}\rf) \sim n^M$ as $n \to \infty$. The first term on the right-hand side tends to zero by Lemma \ref{lem:apest1} and the second tends to zero after taking $n \to \infty$ and then $\delta \to 0$ by the hypothesis that $V\in\sL$. 
Thus, we have shown that \eqref{liminf B} holds.

Now, for $\emptyset \neq I \subset \{1,\dots,M\}$ and $\bfk = (k_i)_{i \in I} \subset \bbZ_+^{|I|}$, define
\begin{align*}
C_{n,I,\bfk} &= \Bigl\{\sum_{i \in I} s_i z_i : \lf n a_{k_i-1}\rf \le s_i < \lf n a_{k_i}\rf\Bigr\}
\end{align*}
For  $x \in D_n$ write $x=\sum_{i=1}^M b_iz_i$ with $\sum_{i=1}^M b_i=n$ and $b_i\in\Z_+$ 
and let $I(x)=\{i:b_i>\fl{na_1}\}$. For $i\in I(x)$ choose $k_i$ maximal such that $\fl{a_{k_i}}\le b_i$.  This way we get a vector $\bfk(x)\in\Z_+^M\cap[0,8M/\delta]^M=K$ 
such that $x$ can be reached from every point $y \in C_{n,I(x),\bfk(x)}$ with an admissible path of at most $n \delta$ steps. 
%The choice of $I(x)$ allows us to account for points $x = \sum_{i=1}^M b_i z_i$ such that $b_i < \lf n a_1 \rf$ for some $i$. 
%As above, there is a finite set $K$ such that $\bfk(x) \in K$ for all $x\in D_n$ and all $n\in\N$ large enough.

For each $x,y$, take a path from $y$ to $x$ such that the steps $z_j,$ $j \in I(x)$ are taken in order. Call this path $x_{0,\ell}$. Then
\begin{align*}
\Bbar(0,x) &= \Bbar(0,y) + \sum_{i=0}^{\ell-1} \Bbar(0,x_{i+1}-x_i, T_{x_i}\omega) \\
& \leq \Bbar(0,y) + \sum_{i=0}^{\ell-1} V(x_{i+1}-x_i, T_{x_i}\omega) + \ell \max_{i \leq M}|m(B)\cdot z_i|.
\end{align*}
Averaging over $y\in C_{n,I(x),\bfk(x)}$ then taking a maximum over $x\in D_n$ we obtain
\begin{align*}
\max_{x \in D_n} \frac{\Bbar(0,x)}{n} &\leq \max_{\substack{\bfk \in K\\\varnothing\ne I\subset\{1,\dotsc,M\}}} \frac{1}{N_{n,I,\bfk}} \sum_{s_1 = \lf n a_{k_{j_1}-1}\rf}^{\lf n a_{k_{j_1}}\rf - 1} \dots \sum_{s_{\abs{I}} = \lf n a_{k_{j_{\abs{I}}}-1}\rf}^{\lf n a_{k_{j_{\abs{I}}}}\rf - 1} \frac{\Bbar(0, \sum_{i=1}^{\abs{I}} s_i z_{j_i})}{n} \\
&+ \sum_{z \in \sR \backslash \{0\}} \Bigl\{\max_{|u|_1 \leq 2n \rho_0} \frac{1}{n} \sum_{0 \leq i \leq n \delta} |V(z,T_{u+iz}\omega)| \Bigr\} + \delta \max_{i \leq M}|m(B)\cdot z_i|,
\end{align*}
where $N_{n,I,\bfk}= \prod_{i=1}^{\abs{I}} (\lf n a_{k_{j_i}}\rf - \lf n a_{k_{j_i}-1}\rf) \sim n^{\abs{I}}$
and $\{j_1,\dotsc,j_{\abs{I}}\}=I$.
The same argument as above now gives
\begin{align*}
\varlimsup_{n\to\infty} \max_{x \in D_n} \frac{\Bbar(0,x)}{n} \leq 0.
\end{align*}
The theorem is proved.
\end{proof}

\section{Auxiliary lemmas}\label{app:lemmas}

We start with a lemma that gives an analogue of J.B.~Martin's result \cite[Theorem 2.4]{Mar-04} on the boundary behavior of the shape function for directed last-passage percolation, in the positive temperature setting. It follows immediately from that result by bounding $\fe^{\beta}$ above and below using $\fe^{\infty}$ and counting paths.

\begin{lemma}\label{lem:polymermartin}
For each $\beta>0$, as $s \searrow 0$ 
	\[\fe^{\beta}(s e_1 + e_2) = \fe^{\beta}(e_1 + s e_2) = \bbE[\w_0] + 2\sqrt{ s \Var(\w_0)} + o(\sqrt{s}).\]
\end{lemma}

\begin{proof}
For any $\beta$ we have $\fe^{\beta}(s e_1 + e_2) = \fe^{\beta}(e_1 + s e_2)$. Using Stirling's formula, we obtain
\begin{align*}
{{N+ \lf Ns \rf}\choose{N}} &= \frac{\sqrt{1 + \lf Ns \rf/N}}{\sqrt{2\pi \lf N s \rf}} \left(1 + \frac{\lf Ns \rf}{N}\right)^N \left(1 + \frac{N}{\lf Ns \rf}\right)^{\lf Ns \rf}(1+o(1))
\end{align*}
By path counting and approximating each path by the largest path, we also observe that
\begin{align*}
N^{-1} G_{0, \lf N(e_1 + s e_2)\rf} \leq (\beta N)^{-1} \log Z_{0,\lf N (e_1 + s e_2)\rf}^\beta \leq N^{-1} G_{0, \lf N(e_1 + s e_2)\rf} + (\beta N)^{-1} \log {{N + \lf Ns \rf}\choose{N}}.
\end{align*}
It follows that
\begin{align*}
\fe^{\infty}(s e_1 + e_2)  \leq \fe^{\beta}(s e_1 + e_2) \leq \fe^{\infty}(s e_1 + e_2) + \beta^{-1}\left[\log(1 + s) + s  \log (1 + s^{-1})\right].
\end{align*}
The result follows from $ \log(1 + s) + s \log (1 + s^{-1}) = o(\sqrt{s})$ as $s \searrow 0$ and \cite[Theorem 2.4]{Mar-04}.
\end{proof}

Now, we provide the proofs we deferred to this appendix. We begin with the following lemma, which applies the above and explains why in the definition of $\partial \fe^{\beta}(\Uset)$ we do not consider the cases $h \in \partial \fe^{\beta}(e_1)$ or $h \in \partial \fe^{\beta}(e_2)$ and which is used in the proof of Lemma \ref{lem:insub}.

\begin{lemma}\label{lem:superdifempty}
$\partial \fe^{\beta}(e_1) = \partial \fe^{\beta}(e_2) = \varnothing$.
\end{lemma}

\begin{proof}
Suppose that $\partial \fe^{\beta}(e_1)$ is not empty and let $-h \in \partial \fe^{\beta}(e_1)$. Take $\xi = e_1 + s e_2$ in \eqref{eq:superdef}. Then
\begin{align*}
s h \cdot e_2 \leq \fe^{\beta}(e_1) - \fe^{\beta}(e_1 + s e_2).
\end{align*}
Observe that $\fe^\beta(e_1)=\fe^\beta(e_2)=\E[\w_0]$.
%By Lemma \ref{lem:polymermartin}, the right hand side is negative and of order $\sqrt{s}$. 
Then taking $s \searrow 0$ and applying Lemma \ref{lem:polymermartin}, 
we must have $h \cdot e_2 = -\infty$, a contradiction. The other claim is similar.
\end{proof}

\begin{proof}[Proof of Lemma \ref{lem:insub}]
Let $\xi \in \ri \Uset$ satisfy $-h \in \partial \fe^{\beta}(\xi)$. Setting $\zeta = \lambda \xi$ and using homogeneity, for all $\lambda > 0$ we must have
\begin{align*}
-(1-\lambda) h \cdot \xi \leq (1-\lambda) \fe^\beta(\xi).
\end{align*}
Taking $\lambda>1$ and $\lambda <1$ and dividing through by $(1-\lambda)$ gives that we must have $-h \cdot \xi = \fe^\beta(\xi)$.
%Extend $\fe^\beta$ by setting $\fe^{\beta}(\xi) = -\infty$ for $\xi \notin \bbR_+^2$, which preserves lower semicontinuity and convexity of $-\fe^{\beta}$ as a function of $\xi \in \bbR^2$.  
%Then by \cite[Theorem 23.5 (d)]{Roc-70} 
%the Legendre-Fenchel transform of $-\fe^\beta$ satisfies 
%	\[(-\fe^{\beta})^*(h) = \sup_{\zeta\in\R^2}\{h\cdot\zeta+\fe^\beta(\zeta)\}=0.\]
%But
%	\[\fepl^{\beta}(h)  =\sup_{\zeta\in\Uset}\{h\cdot\zeta+\fe(\zeta)\}\leq (-\fe^{\beta})^*(h).\] 
For any $\zeta\in\R_+^2$ we have
	\[h\cdot\zeta+\fe^\beta(\zeta)=-h\cdot(\xi-\zeta)-(\fe^\beta(\xi)-\fe^\beta(\zeta))\le0.\]
Taking a supremum over $\zeta\in\Uset$ gives $\fepl^\beta(h)\le0$. 
Since $h \cdot \xi + \fe^\beta(\xi) = 0$, the supremum is achieved and we must have $\fepl^\beta(h) = 0$.

Conversely, suppose that $\fepl^\beta(h) = 0$. Then by continuity, there is $\xi \in \Uset$ with $h \cdot \xi+\fe^\beta(\xi)=0$. For any $\zeta\in\R^2_+\setminus\{0\}$ we have 
	\[h\cdot\zeta+\fe^\beta(\zeta)=\abs{\zeta_1}\,\Bigl[h\cdot\frac{\zeta}{\abs{\zeta}_1}+\fe^\beta\Bigl(\frac{\zeta}{\abs{\zeta}_1}\Bigr)\Bigr]\le\abs{\zeta}_1\fepl^\beta(h)=0\] 
	and hence
	\[-h\cdot(\xi-\zeta)\le\fe^\beta(\xi)-\fe^\beta(\zeta).\]
This implies that $-h\in\partial\fe^\beta(\xi)$.
%But then we have
%\begin{align*}
%(-\fe^{\beta})^*(h) &= \sup_{\xi \in \bbR_+^2} \left\{h \cdot \xi + \fe^\beta(\xi)\right\} = \sup_{\xi \in \bbR_+^2 \backslash \{0\}} \left\{\abs{\xi}_1\left(h \cdot \left(\frac{\xi}{\abs{\xi}_1}\right) + \fe^\beta\left(\frac{\xi}{\abs{\xi}_1}\right)\right)\right\} \leq 0,
%\end{align*}
%where in the last step we used the fact that $h \cdot \left(\frac{\xi}{\abs{\xi}_1}\right) + \fe^\beta\left(\frac{\xi}{\abs{\xi}_1}\right) \leq 0$ for all $\xi \in \bbR_+^2 \backslash \{0\}$. Since there is $\xi \in \Uset$ for which equality holds, we conclude that $(-\fe^{\beta})^*(h) = 0$. By \cite[Theorem 23.5 (d)]{Roc-70} the equality $-\xi \cdot h = \fe^\beta(\xi)$ then implies that $-h \in \partial \fe^\beta(\xi)$. 
But by Lemma \ref{lem:superdifempty} $\partial \fe^{\beta}(e_1)$ and $\partial \fe^{\beta}(e_2)$ are empty, so $\xi \in \ri \Uset$.
\end{proof}

\begin{proof}[Proof of Lemma \ref{ker-consist}]
Write
	\begin{align*}
	\ker_{k,\ell}^\w\ker_{r,s}^\w f(x_{m,\infty})
%	=\int f(\bar X_{m,\infty}) \ker^\w_{k,\ell}(x_{m,k},x_{\ell,\infty},dX_{m,\infty})\ker^\w_{r,s}(X_{m,r},X_{s,\infty},d\bar X_{m,\infty})\\
%	&\quad=\int f(x_{m,k}X_{k,r}\bar X_{r,s}X_{s,\ell}x_{\ell,\infty}) \ker^\w_{k,\ell}(x_{m,k},x_{\ell,\infty},dX_{m,\infty})\ker^\w_{r,s}(X_{m,r},X_{s,\infty},d\bar X_{m,\infty})\\
%	&\quad
%	=\sum_{\substack{x_{k,\ell},\bar x_{r,s}\\ \bar x_r=x_r,\bar x_s=x_s}}f(x_{m,r}\bar x_{r,s}x_{s,\infty}) Q^\w_{x_k,x_\ell}(X_{k,\ell}=x_{k,\ell})Q^\w_{x_r,x_s}(X_{r,s}=\bar x_{r,s})\\
%	&\quad=\sum_{\substack{x_{k,r},\bar x_{r,s},x_{s,\ell}\\ \bar x_r=x_r,\bar x_s=x_s}}f(x_{m,r}\bar x_{r,s}x_{s,\infty})\sum_{x_{r,s}} \frac{e^{\sum_{i=k}^{r-1}\w_{x_i}}e^{\sum_{i=r}^{s-1}\w_{x_i}}e^{\sum_{i=s}^{\ell-1}\w_{x_i}}e^{\sum_{i=r}^{s-1}\w_{\bar x_i}}}{Z_{x_k,x_\ell}Z_{x_r,x_s}}\\
%	&\quad=\sum_{\substack{x_{k,r},\bar x_{r,s},x_{s,\ell}\\ \bar x_r=x_r,\bar x_s=x_s}}f(x_{m,r}\bar x_{r,s}x_{s,\infty})\frac{e^{\sum_{i=k}^{r-1}\w_{x_i}}e^{\sum_{i=r}^{s-1}\w_{\bar x_i}}e^{\sum_{i=s}^{\ell-1}\w_{x_i}}}{Z_{x_k,x_\ell}}\\
%	&\quad=\sum_{x_{k,\ell}}f(x_{m,\infty})Q^\w_{x_k,x_\ell}(x_{k,\ell})=\ker_{k,\ell}^\w f(x_{m,k},x_{\ell,\infty}).\qedhere
&= \sum_{y_{k,\ell} \in \bbX_{x_k}^{x_{\ell}}}\sum_{\bar{y}_{r,s} \in \bbX_{y_r}^{y_s}} f(x_{m,k}y_{k,r}\bar{y}_{r,s}y_{s,\ell}x_{\ell,\infty}) Q_{y_r,y_s}^{\w}(\bar{y}_{r,s}) Q_{x_k,x_\ell}(y_{k,\ell}) \\
&= \sum_{y_{k,\ell} \in \bbX_{x_k}^{x_{\ell}}}\sum_{\bar{y}_{r,s} \in \bbX_{y_r}^{y_s}} f(x_{m,k}y_{k,r}\bar{y}_{r,s}y_{s,\ell}x_{\ell,\infty}) \frac{e^{ \sum_{i=r}^{s-1} \w_{\bar{y}_i} +\sum_{j=k}^{\ell-1} \w_{y_j}  }}{Z_{y_r,y_s}^{\w} Z_{x_k,x_\ell}^{\w} } \\
&= \sum_{\substack{x_k \leq y_r \leq y_s \leq x_{\ell}\\ y_r \in \V_r, y_s \in \V_s}} \sum_{y_{k,r} \in \bbX_{x_k}^{y_r} }\sum_{\bar{y}_{r,s} \in \bbX_{y_r}^{y_s}}\sum_{y_{s,\ell} \in \bbX_{y_s}^{x_\ell} } f(x_{m,k}y_{k,r}\bar{y}_{r,s}y_{s,\ell}x_{\ell,\infty}) \times \\
&\qquad\qquad\qquad\qquad\left(\sum_{y_{r,s} \in \bbX_{y_r}^{y_s}}e^{\sum_{j=r}^{s-1} \w_{y_j}} \frac{ e^{ \sum_{i=r}^{s-1} \w_{\bar{y}_i} +\sum_{j=k}^{r-1} \w_{y_j} +\sum_{j=s}^{\ell-1} \w_{y_j}}}{Z_{y_r,y_s}^{\w} Z_{x_k,x_\ell}^{\w} }\right) \\
&= \sum_{y_{k,\ell} \in \bbX_{x_k}^{x_\ell}}f(x_{m,k}y_{k,\ell}x_{\ell,\infty})Q_{x_k,x_\ell}(y_{k,\ell}) = \ker_{k,\ell}^\w f(x_{m,\infty}).\qedhere
	\end{align*}
\end{proof}

\begin{proof}[Proof of Lemma \ref{lm:DLR}]
If $\Pi_x$ is a Gibbs measure, then \eqref{eq:DLR} comes from 
	\[E^{\Pi_x}[f]=E^{\Pi_x}\bigl[E^{\Pi_x}[f\,|\,\cX_{(k,\ell)^c}]\bigr]=E^{\Pi_x}[\ker_{k,\ell}^\w f]=\Pi_x\ker_{k,\ell}^\w f.\]
For the other direction take a bounded measurable function $f$ and a bounded $\cXx_{(k,\ell)^c}$-measurable function $g$. Then, using \eqref{proper} and \eqref{eq:DLR} we get
	\[E^{\Pi_x}[g\ker_{k,\ell}^\w f]=E^{\Pi_x}[\ker_{k,\ell}^\w(gf)]=E^{\Pi_x}[gf].\]
This proves that $E^{\Pi_x}[f\,|\,\cXx_{(k,\ell)^c}]=\ker_{k,\ell}^\w f$.
\end{proof}

\begin{proof}[Proof of Lemma \ref{lm:Pi-cons}]
If $\Pi_x\in\DLR_x^\w$, then \eqref{Pi-cons} follows from 
	\[\Pi_x(x_{m,n})=\Pi_x\ker_{m,n}^\w(x_{m,n}).\]
%for a given up-right path $x_{0,n}$ we have 
%	\[\mu(X_{0,n}=x_{0,n})=\mu\ker_{x_0,x_n}^\w(X_{0,n}=x_{0,n})=\int\ker_{x_0,x_n}^\w(X_{n,\infty},\{X_{0,n}=x_{0,n}\})\mu(dX_{n,\infty})=Q^\w_{x_0,x_n}(X_{0,n}=x_{0,n})\mu(X_n=x_n).\]
Conversely, assume \eqref{Pi-cons} holds for all up-right paths $x_{m,n}$ with $x_m=x$.  Fix $\ell\ge n\ge k\ge m$ and an up-right path $x_{m,\ell}$ with $x_m=x$. Then
	\begin{align*}
	\Pi_x\ker^\w_{k,n}(x_{m,\ell})
	&=Q^\w_{x_k,x_n}(x_{k,n})\Pi_x(x_{m,k},x_{n,\ell})\\
%	&=\sum_{\substack{\bar x_{k,n}\\ \bar x_k=x_k,\bar x_n=x_n}}\!\!\!\!\!\!\!\!Q^\w_{x_k,x_n}(x_{k,n})\mu(x_{m,k},\bar x_{k,n},x_{n,\ell})\\
%	&=\sum_{\substack{\bar x_{k,n}\\ \bar x_k=x_k,\bar x_n=x_n}}\!\!\!\!\!\!\!\!Q^\w_{x_k,x_n}(x_{k,n})Q^\w_{x_m,x_\ell}(x_{m,k},\bar x_{k,n},x_{n,\ell})\mu(x_\ell)\\
%	&=Q^\w_{x_m,x_\ell}(x_{m,\ell})\mu(x_\ell)=\mu(x_{m,\ell}).
	&=\sum_{y_{k,n} \in \bbX_{x_k}^{x_n}} Q_{x_k,x_n}^{\w}(x_{k,n}) \Pi_x(x_{m,k}y_{k,n}x_{n,\ell}) \\
	&=\sum_{y_{k,n} \in \bbX_{x_k}^{x_n}} Q_{x_k,x_n}^{\w}(x_{k,n}) Q_{x_m,x_\ell}^\w(x_{m,k}y_{k,n}x_{n,\ell})\Pi_x(x_\ell)\\
&=Q^\w_{x_m,x_\ell}(x_{m,\ell})\Pi_x(x_\ell)=\Pi_x(x_{m,\ell}).
	\end{align*}
The penultimate equality came by a cancellation similar to the one in the proof of Lemma \ref{ker-consist}.
\end{proof}

\begin{proof}[Proof of Lemma \ref{lm:p2pQ-pi}]
Let $x\cdot\ehat=m$ and $y\cdot\ehat=n$. For an admissible path $x_{m,n}$ from $x$ to $y$ write
	\[Q^\w_{x,y}(x_{m,n})=\frac{e^{\sum_{i=m}^{n-1}\w_{x_i}}}{Z_{x,y}}=\prod_{i=m}^{n-1}\frac{e^{\w_{x_i}}Z_{x,x_i}}{Z_{x,x_{i+1}}}=\prod_{i=m}^{n-1}\backpix_{x_{i+1},x_i}(\w).\qedhere\]
\end{proof}

%\subsection{Polymer comparison}
Lastly, a lemma that allows us to compare ratios of partition functions.

%\subsection{Polymer comparison}
%\begin{lemma}[Lemma A.1, \cite{Geo-etal-15}]
%For any $x,y \in \bbZ_+^2$ with, coordinatewise, $y - e_1 \geq x + e_1 + e_2$, we have
%\begin{align*}
%\frac{Z_{x+e_1, y - e_1}^\beta}{Z_{x+e_1,y}^\beta} &\leq \frac{Z_{x+e_1+e_2,y-e_1}^\beta}{Z_{x+e_1+e_2,y}^\beta} \leq \frac{Z_{x+e_2,y-e_1}^\beta}{Z_{x+e_2,y}^\beta}.
%\end{align*}
%Similarly, for any $x,y \in \bbZ_+^2$ with, coordinatewise, $y - e_2 \geq x + e_1 + e_2$, we have
%\begin{align*}
%\frac{Z_{x+e_1, y - e_2}^\beta}{Z_{x+e_1,y}^\beta} &\geq \frac{Z_{x+e_1+e_2,y-e_2}^\beta}{Z_{x+e_1+e_2,y}^\beta} \geq \frac{Z_{x+e_2,y-e_2}^\beta}{Z_{x+e_2,y}^\beta}.
%\end{align*}
%\end{lemma}

\begin{lemma}
For any $\w\in\Omega$, $x\in\Z^2$, and $u,v\in x+e_1+e_2+\Z_+^2$ with $u\cdot e_1\ge v\cdot e_1$ and $u\cdot e_2\le v\cdot e_2$
\begin{align}\label{comparison}
\frac{Z^\beta_{x+e_1,u}}{Z^\beta_{x,u}}\ge\frac{Z^\beta_{x+e_1,v}}{Z^\beta_{x,v}}\quad\text{and}\quad
\frac{Z^\beta_{x+e_2,u}}{Z^\beta_{x,u}}\le\frac{Z^\beta_{x+e_2,v}}{Z^\beta_{x,v}}.
\end{align}
\end{lemma}

\begin{proof}
Reversing the picture in \cite[Lemma A.1]{Geo-etal-15} via  $x\mapsto-x$ gives
	\[\frac{Z^\beta_{x+e_1,y}}{Z^\beta_{x,y}}\ge\frac{Z^\beta_{x+e_1,y-e_1}}{Z^\beta_{x,y-e_1}}\ge\frac{Z^\beta_{x+e_1,y-e_1+e_2}}{Z^\beta_{x,y-e_1+e_2}}\]
for all $x,y\in\Z^2$ with $y\ge x$ and any choice of $\w\in\Omega$. The first equality in \eqref{comparison} comes by applying this repeatedly with $y$ on any up-left path from $u$ to $v$. The second equality is similar.
%
%For any $x,y \in \bbZ^2$ with, coordinatewise, $y - e_1 \geq x + e_1 + e_2$, we have
%\begin{align*}
%\frac{Z_{x+e_1, y - e_1}^\beta}{Z_{x+e_1,y}^\beta} &\leq \frac{Z_{x+e_1+e_2,y-e_1}^\beta}{Z_{x+e_1+e_2,y}^\beta} \leq \frac{Z_{x+e_2,y-e_1}^\beta}{Z_{x+e_2,y}^\beta}.
%\end{align*}
%Similarly, for any $x,y \in \bbZ_+^2$ with, coordinatewise, $y - e_2 \geq x + e_1 + e_2$, we have
%\begin{align*}
%\frac{Z_{x+e_1, y - e_2}^\beta}{Z_{x+e_1,y}^\beta} &\geq \frac{Z_{x+e_1+e_2,y-e_2}^\beta}{Z_{x+e_1+e_2,y}^\beta} \geq \frac{Z_{x+e_2,y-e_2}^\beta}{Z_{x+e_2,y}^\beta}.
%\end{align*}
\end{proof}

\section*{Acknowledgements}
We thank Timo Sepp\"al\"ainen and Louis Fan for many valuable comments on the manuscript.

\footnotesize

\bibliographystyle{aop-no-url}

\bibliography{firasbib2010}

\begin{thebibliography}{62}
\expandafter\ifx\csname natexlab\endcsname\relax\def\natexlab#1{#1}\fi
\expandafter\ifx\csname url\endcsname\relax
  \def\url#1{\texttt{#1}}\fi
\expandafter\ifx\csname urlprefix\endcsname\relax\def\urlprefix{URL }\fi

\bibitem{Aiz-Weh-90}
\textsc{Aizenman, M.} and \textsc{Wehr, J.} (1990).
\newblock Rounding effects of quenched randomness on first-order phase
  transitions.
\newblock \textit{Comm. Math. Phys.} \textbf{130} 489--528.

\bibitem{Arg-etal-14}
\textsc{Arguin, L.-P.}, \textsc{Newman, C.~M.}, \textsc{Stein, D.~L.} and
  \textsc{Wehr, J.} (2014).
\newblock Fluctuation bounds for interface free energies in spin glasses.
\newblock \textit{J. Stat. Phys.} \textbf{156} 221--238.

\bibitem{Bak-13}
\textsc{Bakhtin, Y.} (2013).
\newblock The {B}urgers equation with {P}oisson random forcing.
\newblock \textit{Ann. Probab.} \textbf{41} 2961--2989.

\bibitem{Bak-16}
\textsc{Bakhtin, Y.} (2016).
\newblock Inviscid {B}urgers equation with random kick forcing in noncompact
  setting.
\newblock \textit{Electron. J. Probab.} \textbf{21} Paper No. 37, 50.

\bibitem{Bak-Cat-Kha-14}
\textsc{Bakhtin, Y.}, \textsc{Cator, E.} and \textsc{Khanin, K.} (2014).
\newblock Space-time stationary solutions for the {B}urgers equation.
\newblock \textit{J. Amer. Math. Soc.} \textbf{27} 193--238.

\bibitem{Bak-Kha-10}
\textsc{Bakhtin, Y.} and \textsc{Khanin, K.} (2010).
\newblock Localization and {P}erron-{F}robenius theory for directed polymers.
\newblock \textit{Mosc. Math. J.} \textbf{10} 667--686, 838.

\bibitem{Bak-Kha-18}
\textsc{Bakhtin, Y.} and \textsc{Khanin, K.} (2018).
\newblock On global solutions of the random {H}amilton-{J}acobi equations and
  the {KPZ} problem.
\newblock \textit{Nonlinearity} \textbf{31} R93--R121.

\bibitem{Bak-Li-18-}
\textsc{Bakhtin, Y.} and \textsc{Li, L.} (2018).
\newblock Thermodynamic limit for directed polymers and stationary solutions of
  the burgers equation.
\newblock \textit{Communications on Pure and Applied Mathematics.} To appear.

\bibitem{Bog-07}
\textsc{Bogachev, V.~I.} (2007).
\newblock \textit{Measure theory. {V}ol. {I}, {II}}.
\newblock Springer-Verlag, Berlin.

\bibitem{Bur-Kea-89}
\textsc{Burton, R.~M.} and \textsc{Keane, M.} (1989).
\newblock Density and uniqueness in percolation.
\newblock \textit{Comm. Math. Phys.} \textbf{121} 501--505.

\bibitem{Car-Sou-17}
\textsc{Cardaliaguet, P.} and \textsc{Souganidis, P.~E.} (2017).
\newblock On the existence of correctors for the stochastic homogenization of
  viscous {H}amilton-{J}acobi equations.
\newblock \textit{C. R. Math. Acad. Sci. Paris} \textbf{355} 786--794.

\bibitem{Cat-Pim-11}
\textsc{Cator, E.} and \textsc{Pimentel, L. P.~R.} (2011).
\newblock A shape theorem and semi-infinite geodesics for the {H}ammersley
  model with random weights.
\newblock \textit{ALEA Lat. Am. J. Probab. Math. Stat.} \textbf{8} 163--175.

\bibitem{Cat-Pim-12}
\textsc{Cator, E.} and \textsc{Pimentel, L. P.~R.} (2012).
\newblock Busemann functions and equilibrium measures in last passage
  percolation models.
\newblock \textit{Probab. Theory Related Fields} \textbf{154} 89--125.

\bibitem{Cat-Pim-13}
\textsc{Cator, E.} and \textsc{Pimentel, L. P.~R.} (2013).
\newblock Busemann functions and the speed of a second class particle in the
  rarefaction fan.
\newblock \textit{Ann. Probab.} \textbf{41} 2401--2425.

\bibitem{Com-17}
\textsc{Comets, F.} (2017).
\newblock \textit{Directed polymers in random environments}, vol. 2175 of
  \textit{Lecture Notes in Mathematics}.
\newblock Springer, Cham.
\newblock Lecture notes from the 46th Probability Summer School held in
  Saint-Flour, 2016.

\bibitem{Com-Shi-Yos-04}
\textsc{Comets, F.}, \textsc{Shiga, T.} and \textsc{Yoshida, N.} (2004).
\newblock Probabilistic analysis of directed polymers in a random environment:
  a review.
\newblock In \textit{Stochastic analysis on large scale interacting systems},
  vol.~39 of \textit{Adv. Stud. Pure Math.} Math. Soc. Japan, Tokyo, 115--142.

\bibitem{Com-Var-06}
\textsc{Comets, F.} and \textsc{Vargas, V.} (2006).
\newblock Majorizing multiplicative cascades for directed polymers in random
  media.
\newblock \textit{ALEA Lat. Am. J. Probab. Math. Stat.} \textbf{2} 267--277.

\bibitem{Com-Yos-06}
\textsc{Comets, F.} and \textsc{Yoshida, N.} (2006).
\newblock Directed polymers in random environment are diffusive at weak
  disorder.
\newblock \textit{Ann. Probab.} \textbf{34} 1746--1770.

\bibitem{Dam-Han-14}
\textsc{Damron, M.} and \textsc{Hanson, J.} (2014).
\newblock Busemann functions and infinite geodesics in two-dimensional
  first-passage percolation.
\newblock \textit{Comm. Math. Phys.} \textbf{325} 917--963.

\bibitem{Dam-Han-17}
\textsc{Damron, M.} and \textsc{Hanson, J.} (2017).
\newblock Bigeodesics in {F}irst-{P}assage {P}ercolation.
\newblock \textit{Comm. Math. Phys.} \textbf{349} 753--776.

\bibitem{Hol-09}
\textsc{den Hollander, F.} (2009).
\newblock \textit{Random polymers}, vol. 1974 of \textit{Lecture Notes in
  Mathematics}.
\newblock Springer-Verlag, Berlin.
\newblock Lectures from the 37th Probability Summer School held in Saint-Flour,
  2007.

\bibitem{Dir-Sou-05}
\textsc{Dirr, N.} and \textsc{Souganidis, P.~E.} (2005).
\newblock Large-time behavior for viscous and nonviscous {H}amilton-{J}acobi
  equations forced by additive noise.
\newblock \textit{SIAM J. Math. Anal.} \textbf{37} 777--796.

\bibitem{Dur-10}
\textsc{Durrett, R.} (2010).
\newblock \textit{Probability: theory and examples}, vol.~31 of
  \textit{Cambridge Series in Statistical and Probabilistic Mathematics}.
\newblock 4th edn. Cambridge University Press, Cambridge.

\bibitem{E-etal-00}
\textsc{E, W.}, \textsc{Khanin, K.}, \textsc{Mazel, A.} and \textsc{Sinai, Y.}
  (2000).
\newblock Invariant measures for {B}urgers equation with stochastic forcing.
\newblock \textit{Ann. of Math. (2)} \textbf{151} 877--960.

\bibitem{Fan-Sep-18-}
\textsc{Fan, W.-T.~L.} and \textsc{Sepp{{\"a}}l{{\"a}}inen, T.} (2018).
\newblock Joint distribution of {B}usemann functions in the exactly solvable
  corner growth model. Preprint ({\tt arXiv 1808.09069}).

\bibitem{Fer-Pim-05}
\textsc{Ferrari, P.~A.} and \textsc{Pimentel, L. P.~R.} (2005).
\newblock Competition interfaces and second class particles.
\newblock \textit{Ann. Probab.} \textbf{33} 1235--1254.

\bibitem{Gar-Mar-05}
\textsc{Garet, O.} and \textsc{Marchand, R.} (2005).
\newblock Coexistence in two-type first-passage percolation models.
\newblock \textit{Ann. Appl. Probab.} \textbf{15} 298--330.

\bibitem{Geo-Ras-Sep-16}
\textsc{Georgiou, N.}, \textsc{Rassoul-Agha, F.} and
  \textsc{Sepp{{\"a}}l{{\"a}}inen, T.} (2016).
\newblock Variational formulas and cocycle solutions for directed polymer and
  percolation models.
\newblock \textit{Comm. Math. Phys.} \textbf{346} 741--779.

\bibitem{Geo-Ras-Sep-17-ptrf-2}
\textsc{Georgiou, N.}, \textsc{Rassoul-Agha, F.} and
  \textsc{Sepp{\"a}l{\"a}inen, T.} (2017{\natexlab{a}}).
\newblock Geodesics and the competition interface for the corner growth model.
\newblock \textit{Probab. Theory Related Fields} \textbf{169} 223--255.

\bibitem{Geo-Ras-Sep-17-ptrf-1}
\textsc{Georgiou, N.}, \textsc{Rassoul-Agha, F.} and
  \textsc{Sepp{\"a}l{\"a}inen, T.} (2017{\natexlab{b}}).
\newblock Stationary cocycles and {B}usemann functions for the corner growth
  model.
\newblock \textit{Probab. Theory Related Fields} \textbf{169} 177--222.

\bibitem{Geo-etal-15}
\textsc{Georgiou, N.}, \textsc{Rassoul-Agha, F.},
  \textsc{Sepp{{\"a}}l{{\"a}}inen, T.} and \textsc{Yilmaz, A.} (2015).
\newblock Ratios of partition functions for the log-gamma polymer.
\newblock \textit{Ann. Probab.} \textbf{43} 2282--2331.

\bibitem{Gom-etal-05}
\textsc{Gomes, D.}, \textsc{Iturriaga, R.}, \textsc{Khanin, K.} and
  \textsc{Padilla, P.} (2005).
\newblock Viscosity limit of stationary distributions for the random forced
  {B}urgers equation.
\newblock \textit{Mosc. Math. J.} \textbf{5} 613--631, 743.

\bibitem{Hoa-Kha-03}
\textsc{Hoang, V.~H.} and \textsc{Khanin, K.} (2003).
\newblock Random {B}urgers equation and {L}agrangian systems in non-compact
  domains.
\newblock \textit{Nonlinearity} \textbf{16} 819--842.

\bibitem{Hof-05}
\textsc{Hoffman, C.} (2005).
\newblock Coexistence for {R}ichardson type competing spatial growth models.
\newblock \textit{Ann. Appl. Probab.} \textbf{15} 739--747.

\bibitem{Hof-08}
\textsc{Hoffman, C.} (2008).
\newblock Geodesics in first passage percolation.
\newblock \textit{Ann. Appl. Probab.} \textbf{18} 1944--1969.

\bibitem{How-New-97}
\textsc{Howard, C.~D.} and \textsc{Newman, C.~M.} (1997).
\newblock Euclidean models of first-passage percolation.
\newblock \textit{Probab. Theory Related Fields} \textbf{108} 153--170.

\bibitem{How-New-01}
\textsc{Howard, C.~D.} and \textsc{Newman, C.~M.} (2001).
\newblock Geodesics and spanning trees for {E}uclidean first-passage
  percolation.
\newblock \textit{Ann. Probab.} \textbf{29} 577--623.

\bibitem{Hus-Hen-85}
\textsc{Huse, D.~A.} and \textsc{Henley, C.~L.} (1985).
\newblock Pinning and roughening of domain walls in ising systems due to random
  impurities.
\newblock \textit{Phys. Rev. Lett.} \textbf{54} 2708--2711.

\bibitem{Itu-Kha-03}
\textsc{Iturriaga, R.} and \textsc{Khanin, K.} (2003).
\newblock Burgers turbulence and random {L}agrangian systems.
\newblock \textit{Comm. Math. Phys.} \textbf{232} 377--428.

\bibitem{Jan-Ras-18-ecp-}
\textsc{Janjigian, C.} and \textsc{Rassoul-Agha, F.} (2018).
\newblock Uniqueness and ergodicity of stationary directed polymer models on
  the square lattice. Preprint ({\tt arXiv 1812.08864}).

\bibitem{Jan-Ras-19-pams-}
\textsc{Janjigian, C.} and \textsc{Rassoul-Agha, F.} (2019).
\newblock Existence, uniqueness, and stability of global solutions of a
  discrete stochastic {B}urgers equation. Preprint.

\bibitem{Kif-97}
\textsc{Kifer, Y.} (1997).
\newblock The {B}urgers equation with a random force and a general model for
  directed polymers in random environments.
\newblock \textit{Probab. Theory Related Fields} \textbf{108} 29--65.

\bibitem{Kre-85}
\textsc{Krengel, U.} (1985).
\newblock \textit{Ergodic theorems}, vol.~6 of \textit{de Gruyter Studies in
  Mathematics}.
\newblock Walter de Gruyter \& Co., Berlin.
\newblock With a supplement by Antoine Brunel.

\bibitem{Lac-10}
\textsc{Lacoin, H.} (2010).
\newblock New bounds for the free energy of directed polymers in dimension
  {$1+1$} and {$1+2$}.
\newblock \textit{Comm. Math. Phys.} \textbf{294} 471--503.

\bibitem{Lic-New-96}
\textsc{Licea, C.} and \textsc{Newman, C.~M.} (1996).
\newblock Geodesics in two-dimensional first-passage percolation.
\newblock \textit{Ann. Probab.} \textbf{24} 399--410.

\bibitem{Lig-85}
\textsc{Liggett, T.~M.} (1985).
\newblock An improved subadditive ergodic theorem.
\newblock \textit{Ann. Probab.} \textbf{13} 1279--1285.

\bibitem{Mai-Pra-03}
\textsc{Mairesse, J.} and \textsc{Prabhakar, B.} (2003).
\newblock The existence of fixed points for the {$\cdot/GI/1$} queue.
\newblock \textit{Ann. Probab.} \textbf{31} 2216--2236.

\bibitem{Mar-04}
\textsc{Martin, J.~B.} (2004).
\newblock Limiting shape for directed percolation models.
\newblock \textit{Ann. Probab.} \textbf{32} 2908--2937.

\bibitem{New-95}
\textsc{Newman, C.~M.} (1995).
\newblock A surface view of first-passage percolation.
\newblock In \textit{Proceedings of the {I}nternational {C}ongress of
  {M}athematicians, {V}ol.\ 1, 2 ({Z}{\"u}rich, 1994)}. Birkh{\"a}user, Basel.

\bibitem{New-97}
\textsc{Newman, C.~M.} (1997).
\newblock \textit{Topics in disordered systems}.
\newblock Lectures in Mathematics ETH Z{\"u}rich, Birkh{\"a}user Verlag, Basel.

\bibitem{Phe-01}
\textsc{Phelps, R.~R.} (2001).
\newblock \textit{Lectures on {C}hoquet's theorem}, vol. 1757 of
  \textit{Lecture Notes in Mathematics}.
\newblock 2nd edn. Springer-Verlag, Berlin.

\bibitem{Pra-03}
\textsc{Prabhakar, B.} (2003).
\newblock The attractiveness of the fixed points of a {$\cdot/GI/1$} queue.
\newblock \textit{Ann. Probab.} \textbf{31} 2237--2269.

\bibitem{Ras-Sep-14}
\textsc{Rassoul-Agha, F.} and \textsc{Sepp{{\"a}}l{{\"a}}inen, T.} (2014).
\newblock Quenched point-to-point free energy for random walks in random
  potentials.
\newblock \textit{Probab. Theory Related Fields} \textbf{158} 711--750.

\bibitem{Ras-Sep-15-ldp}
\textsc{Rassoul-Agha, F.} and \textsc{Sepp{{\"a}}l{{\"a}}inen, T.} (2015).
\newblock \textit{A course on large deviations with an introduction to {G}ibbs
  measures}, vol. 162 of \textit{Graduate Studies in Mathematics}.
\newblock American Mathematical Society, Providence, RI.

\bibitem{Ras-Sep-Yil-13}
\textsc{Rassoul-Agha, F.}, \textsc{Sepp{{\"a}}l{{\"a}}inen, T.} and
  \textsc{Yilmaz, A.} (2013).
\newblock Quenched free energy and large deviations for random walks in random
  potentials.
\newblock \textit{Comm. Pure Appl. Math.} \textbf{66} 202--244.

\bibitem{Roc-70}
\textsc{Rockafellar, R.~T.} (1970).
\newblock \textit{Convex analysis}.
\newblock Princeton Mathematical Series, No. 28, Princeton University Press,
  Princeton, N.J.

\bibitem{Sep-12-corr}
\textsc{Sepp{{\"a}}l{{\"a}}inen, T.} (2012).
\newblock Scaling for a one-dimensional directed polymer with boundary
  conditions.
\newblock \textit{Ann. Probab.} \textbf{40} 19--73.
\newblock Corrected version available at {\tt arXiv:0911.2446}.

\bibitem{Sin-91}
\textsc{Sina\u\i, Y.~G.} (1991).
\newblock Two results concerning asymptotic behavior of solutions of the
  {B}urgers equation with force.
\newblock \textit{J. Statist. Phys.} \textbf{64} 1--12.

\bibitem{Sin-93}
\textsc{Sina\u\i, Y.~G.} (1993).
\newblock A random walk with a random potential.
\newblock \textit{Teor. Veroyatnost. i Primenen.} \textbf{38} 457--460.

\bibitem{Weh-Was-16}
\textsc{Wehr, J.} and \textsc{Wasielak, A.} (2016).
\newblock Uniqueness of translation-covariant zero-temperature metastate in
  disordered {I}sing ferromagnets.
\newblock \textit{J. Stat. Phys.} \textbf{162} 487--494.

\bibitem{Wut-02}
\textsc{W{{\"u}}thrich, M.~V.} (2002).
\newblock Asymptotic behaviour of semi-infinite geodesics for maximal
  increasing subsequences in the plane.
\newblock In \textit{In and out of equilibrium ({M}ambucaba, 2000)}, vol.~51 of
  \textit{Progr. Probab.} Birkh{\"a}user Boston, Boston, MA, 205--226.

\bibitem{Yil-09-aop}
\textsc{Yilmaz, A.} (2009).
\newblock Large deviations for random walk in a space-time product environment.
\newblock \textit{Ann. Probab.} \textbf{37} 189--205.

\end{thebibliography}

\end{document}